\documentclass{amsart}
\usepackage[margin=1.2in]{geometry}
\usepackage{mathtools, amssymb, amsthm}
\usepackage[usenames,dvipsnames]{xcolor}
\definecolor{orange}{HTML}{FFE194}
\definecolor{darkblue}{HTML}{B8DFD8}
\definecolor{lightblue}{HTML}{E8F6EF}
\usepackage{tikz-cd}
\usetikzlibrary{decorations.pathreplacing}
\usepackage{stackengine} %for picture superposition
\usepackage{fancyvrb} %center verbatims
\usepackage{enumitem}
\usepackage[noend]{algpseudocode}
\usepackage[colorlinks = true, linkcolor = blue, citecolor=blue]{hyperref}

\setlist[enumerate]{label=(\alph*)}

\newtheorem{thm}{Theorem}[section]
\newtheorem{theorem}[thm]{Theorem}
\newtheorem{lm}[thm]{Lemma}
\newtheorem{lemma}[thm]{Lemma}
\newtheorem{cor}[thm]{Corollary}

\newtheorem{prop}[thm]{Proposition}
\newtheorem{proposition}[thm]{Proposition}

\newtheorem{conjecture}[thm]{Conjecture}

\newtheorem*{thm*}{Theorem}
\newtheorem*{conj*}{Conjecture}

\theoremstyle{definition}
\newtheorem{de}[thm]{Definition}
\newtheorem{definition}[thm]{Definition}
\newtheorem{ex}[thm]{Example}
\newtheorem{example}[thm]{Example}
\newtheorem{re}[thm]{Remark}
\newtheorem{remark}[thm]{Remark}

\DeclareMathOperator{\NN}{\mathbb{N}}
\DeclareMathOperator{\ZZ}{\mathbb{Z}}
\DeclareMathOperator{\QQ}{\mathbb{Q}}
\DeclareMathOperator{\RR}{\mathbb{R}}

\DeclareMathOperator{\supp}{supp}
\DeclareMathOperator{\sign}{sign}

\DeclareMathOperator{\sgn}{sign}
\DeclareMathOperator{\contr}{contr}
\DeclareMathOperator{\relcoord}{relcoord}
\DeclareMathOperator{\relset}{relset}

\DeclareMathOperator{\Prob}{Prob}

\newcommand{\mybreak}{\medskip}
\newcommand{\newpar}{}
\newcommand{\itembreak}{\medskip}

\newcommand{\indep}{\perp \!\!\! \perp}
\newcommand{\floor}[1]{\lfloor #1 \rfloor}

\newcommand{\rev}[1]{#1}
\newcommand{\revml}[1]{#1}
\newcommand{\rrev}[1]{#1}
\newcommand{\rrevml}[1]{#1}

\newcommand{\ratmodel}{rational one-dimensional model}
\newcommand{\ratmodels}{rational one-dimensional models}
\newcommand{\Ratmodels}{Rational one-dimensional models}
\newcommand{\redratmodel}{R1d model}
\newcommand{\redratmodels}{R1d models}

\newcommand{\irredundant}{irredundant}

\newcommand{\moved}{}

\title[Classifying one-dimensional discrete models with maximum likelihood degree one]{Classifying one-dimensional discrete models\\ with maximum likelihood degree one}

\author{Arthur Bik}
\address[Arthur Bik]{New York, NY}
\urladdr{\url{http://arthurbik.nl}}
\author{Orlando Marigliano}
\address[Orlando Marigliano]{University of Genoa\\Via Dodecaneso 35\\16146 Genova GE\\Italy}
\urladdr{\url{https://orlandomarigliano.com}}
\email{orlando.marigliano@edu.unige.it}

\begin{document}
\begin{abstract}
We propose a classification of all one-dimensional discrete statistical models with maximum likelihood degree one based on their rational parametrization. We show how all such models can be constructed from members of a smaller class of `fundamental models' using a finite number of simple operations.
We introduce `chipsplitting games', a class of combinatorial games on a grid which we use to represent fundamental models.
This combinatorial perspective enables us to show that there are only finitely many fundamental models in the probability simplex $\Delta_n$ for $n\leq 4$.
\end{abstract}
\maketitle
\section{Introduction}\label{sec:intro}

A \emph{discrete statistical model} is a subset of the simplex $\Delta_n := \{p\in \rev{\mathbb R^{n+1}_{\geq 0}} \mid \sum_\nu p_\nu = 1\}$ of probability distributions on $n+1$ events for some $n\in \mathbb N$. In algebraic statistics, we are interested in models which are \emph{algebraic}, meaning that the model is the intersection of $\Delta_n$ and some semialgebraic set in $\mathbb R^{n+1}$. Here, models with \emph{maximum likelihood  degree one} are of special interest because for these, the maximum likelihood (ML) estimation problem is a rational function in the entries of the observed data, and therefore algebraically simplest.

\begin{example}\label{expl:indep} 
Consider the set $\Delta_2$ of probability distributions on three events and of this, the subset $\mathcal M_{\indep}$ that models throwing a biased coin twice and recording the number of times it shows heads.
An empirical observation is then represented by a triple $u = (u_0, u_1, u_2)$ of numbers indicating the number of times we observed the result of no heads, one head, and two heads, respectively. From this data, the most reasonable guess for the probability \rrev{$\theta$} that the coin will show heads is
\[
\theta = \frac{2u_2 + u_1}{2(u_2 + u_1 + u_0)}.
\]
\rrevml{More precisely, this expression maximizes the likelihood
\[
\Prob(u=(u_0, u_1, u_2)\mid \theta) =
\binom{u_0,u_1,u_2}{u_0+u_1+u_2}
(1-\theta)^{2u_0}
\theta^{u_1}
(1-\theta)^{u_1}
\theta^{2u_2}
\]
of observing the empirical distribution $u$ given the parameter $\theta$. For this reason, the above expression for $\theta$ is called the \emph{maximum likelihood (ML) estimate} of $\mathcal M_{\indep}$ for the data $u$. }Since this expression is a \emph{rational} expression in the entries of $u$, the model $\mathcal M_{\indep}$ has ML degree one.
\end{example}

\mybreak
\rrev{In general, the ML degree of an algebraic model can be higher than one. In this case, the ML estimate of the model for generic data $u$ is obtained by taking a finite field extension of $\mathbb{C}(u)$ of correspondingly high degree. Thus, the ML degree is a measure of the algebraic complexity of the ML estimate. An ML degree of one indicates the simplest case where no field extension is taken.}

\mybreak
Algebraic statistical models with ML degree one have been a recurring object of study in recent years. The problems considered range from identifying ML degree one members of a given family of models to studying properties, parametrizations, and normal forms of ML degree one models in general. Articles have been written both on the discrete~\cite{ToricInvariant, ToricFano,  QuasiIndependence, BMS, Huh} and Gaussian~\cite{GaussianGraphical, DifferentialEquations, KroneckerCovariance, MultivariateGaussians} case.

\mybreak
In particular, two articles \cite{BMS, Huh} study discrete models of ML degree one in general. These works explain which form these models may take and provide systematic parametrizations. However, a general classification seems out of reach. More specifically, we would like to divide the set of all discrete algebraic models with ML degree one contained in the simplex $\Delta_n$ into finitely many easy to understand families. But at the time~\cite{BMS} was published, there was no way to do so even for the simplest case $n=2$.

\mybreak
In this paper we provide such a classification for $n=2$, and extend this to $n=3,4$ in the case where the models in question are \emph{one-dimensional} as algebraic varieties. Since one-dimensional discrete algebraic models with maximum likelihood degree one are the focus of this paper, \rev{we will refer to these models often. We call them ``\ratmodels'' for short, sometimes shortening this further to ``R1d models''.} %we use the word `model' to refer specifically to these.

\mybreak
We start by stratifying the set of \ratmodels{} in $\Delta_n$ by their \emph{degree} as algebraic curves. We find that for a fixed $d$, there are essentially finitely many ways to construct \ratmodels{} of degree $\leq d$. We make this precise by introducing the notion of \emph{fundamental models}, from which all other \ratmodels{} can be constructed.
Since there are finitely many fundamental models of degree $\leq d$, we are satisfied with our classification if we can find an upper bound for $\deg(\mathcal M)$, where $\mathcal M$ ranges over all \ratmodels{} in $\Delta_n$. This would imply that there are finitely many fundamental models in $\Delta_n$.

\mybreak
Our main theorem gives such an upper bound for \ratmodels{} contained in $\Delta_2$, $\Delta_3$, and $\Delta_4$.

\begin{thm}\label{thm:main_models}
Let $n\leq 4$ and let $\mathcal M\subseteq \Delta_n$ be a one-dim.\ discrete model with ML degree one. Then \[\deg(\mathcal M)\leq 2n-1.\]
\end{thm}

To prove Theorem~\ref{thm:main_models} we use a strategy inspired by the literature on chip-firing~\cite{Klivans}, which motivates the formulation of an equivalent combinatorial problem. \rev{In Proposition~\ref{first-step}}, we observe that \ratmodels{} admit a parametrization
\[
p\colon [0,1]\to \Delta_n,\quad t \mapsto (w_\nu t^{i_\nu}(1-t)^{j_\nu})_{\nu=0}^n
\]
which enables us to represent \rev{these models} as sets of integers on a grid.

\begin{example}[\rrev{Example~\ref{expl:indep} continued}]\label{expl:indep2}
The model \rev{$\mathcal M_{\indep} \subseteq \Delta_2$  is} parametrized by the function $p(t) = (t^2, 2t(1-t), (1-t)^2)$ \rev{and} can be represented by the following picture.

\medskip
\begin{center}
\begin{BVerbatim}
 1
 · 2
-1 · 1
\end{BVerbatim}
\end{center}
\smallskip
In such a picture, the grid point with coordinates $(i,j)$ represents the monomial $t^i(1-t)^j$. The integer entry at that point represents the coefficient of that monomial in the parametrization, where a dot represents the entry 0. The entry $-1$ at the point $(0,0)$ indicates that the coordinates of the parametrization add up to $1$.
\revml{More precisely, it is the coefficient of the constant term of the polynomial
\[
1\cdot t^2 + 2\cdot ts + 1\cdot s^2 - 1
\]
which becomes zero after the substitution $s\mapsto (1-t)$.}
We think of these \rev{integer} entries as `chips' on the grid,  allowing for negative chips. Thus we call such a representation a \emph{chip configuration}.
\end{example}

\mybreak
Any chip on the grid can be split into two further chips, which are then placed directly to the north and to the east of the original chip. We can `split a chip' where there are none by adding a negative chip. Finally, we can unsplit a chip by performing a splitting move in reverse. Starting from the zero configuration, these chipsplitting moves can be used to produce models. For instance, we get the model \rev{$\mathcal M_{\indep}$} by performing chipsplitting moves at $(0,0)$, $(1,0)$, and $(0,1)$, as visualized below.
\medskip
\begin{center}
\begin{BVerbatim}
·          ·           ·          1        
· ·        1 ·         1 1        · 2      
0 · ·     -1 1 ·      -1 · 1     -1 · 1 
\end{BVerbatim}
\end{center}
\smallskip
In this view, Theorem~\ref{thm:main_models} becomes a combinatorial statement about the possible outcomes of these sequences of chipsplitting moves, which we call \emph{chipsplitting games}.

\subsection*{Outline}
\rev{In the following Preliminaries section we introduce chipsplitting games and formulate the combinatorial equivalent of Theorem~\ref{thm:main_models}.} In Section~\ref{sec:models} we explain how to use Theorem~\ref{thm:main_models} to classify all \ratmodels{} in $\Delta_n$ for $n\leq 4$. In Section~\ref{sec:chipsplitting} we introduce \rev{combinatorial tools for proving our main result}. In Section~\ref{sec:model2chipsplitting} we explain the connection between \ratmodels{} and chipsplitting games. In Sections~\ref{sec:supp+<=3}–\ref{sec:supp+==5} we prove Theorem~\ref{thm:main_models} in the language of chipsplitting games for $n\leq2$, $n=3$, and $n=4$, respectively. In Section~\ref{sec:computations} we discuss examples, computations, and future directions.

\subsection*{Code} We use the computer algebra system Sage~\cite{sage} to assist us in our proofs, especially in Section~\ref{sec:supp+==5}, and to implement our algorithm for finding fundamental models in Section~\ref{sec:computations}. The code is available on MathRepo at \url{https://mathrepo.mis.mpg.de/ChipsplittingModels/}.

\subsection*{Acknowledgements}
We thank Bernd Sturmfels and Caroline Klivans for helpful discussions at the early stages of this project. \rev{We thank the anonymous referee for helping us clarify and improve our manuscript.}
The first author was partially supported by Postdoc.Mobility Fellowship P400P2\_199196 from the Swiss National Science Foundation. The second author was supported by Brummer \& Partners MathDataLab and the European Union (101061315–MIAS–HORIZON-MSCA-2021-PF-01).

\section*{Preliminaries}

\rev{
For indices $i$ and $j$, the Kronecker delta symbol $\delta_{ij}$ equals $1$ if $i=j$ and $0$ otherwise. For sets $A$ and $B$ we denote the set of functions $A\to B$ by $B^A$. We often write elements of $B^A$ as $A$-indexed collections of elements of $B$. The set of subsets of $A$ is denoted by $2^A$. The cardinality of $A$ is $\# A$.}

\mybreak
Let $(V,E)$ be a directed graph without loops.

\begin{de}\label{de:chipsplitting}
Let $V'\subseteq V$ be the subset of \rev{non-sinks}.
\begin{enumerate}
\item[(a)] A {\em chip configuration} is a vector $w=(w_v)_{v\in V}\in \ZZ^V$ such that $\#\{v\in V\mid w_v\neq0\}<\infty$. 
\item[(b)] The {\em initial configuration} is the zero vector $0\in\ZZ^V$.
\item[(c)] A {\em splitting move} at $p\in V$ maps a chip configuration $w=(w_v)_{v\in V}$ to the chip configuration $\widetilde{w}=(\widetilde{w}_v)_{v\in V}$ defined by
\[
\widetilde{w}_v:=\left\{\begin{array}{cl}w_v-1&\mbox{if $v=p$,}\\w_v+1&\mbox{if $(p,v)\in E$, i.e., $E$ contains an edge from $p$ to $v$,}\\w_v&\mbox{otherwise.}\end{array}\right.
\]
An {\em unsplitting move} at $p$ maps $\widetilde{w}$ back to $w$.
\item[(d)] A {\em chipsplitting game} $f$ is a finite sequence of splitting and unsplitting moves. The {\em outcome} of~$f$ is the chip configuration obtained from the initial configuration after executing all the moves in~$f$.
\item[(e)] A {\em (chipsplitting) outcome} is the outcome of any chipsplitting game.
\end{enumerate}
\end{de}

Note that the moves in our game are all reversible and commute with each other. In particular, the order of the moves in a game does not matter. \rev{Furthermore, every chipsplitting outcome can be obtained as the outcome of a chipsplitting game} such that \rrev{there is no vertex in $V$ where} both a splitting and an unsplitting move occur. We call games that have this property {\em \irredundant{}}. We usually assume chipsplitting games are \irredundant{}. \rev{Moreover, we consider two games $f$, $g$ equivalent ($f\sim g$) if they are equal up to reordering.} \rev{Given an irredundant chipsplitting game $f$, we count the number of moves in $f$ at each non-sink vertex $p$ of $V$, counting unsplitting moves negatively. We obtain the bijection}
\begin{eqnarray*}
\{\mbox{\irredundant{} chipsplitting games on }(V,E)\}/\sim&\to&\{g\colon V'\to\ZZ\mid\#\{p\in V'\mid g(p)\neq0\}<\infty\}\\
f&\mapsto&\left(p\mapsto \mbox{number of moves at }p\mbox{ in }f\right).
\end{eqnarray*}
\rev{Thus,} we identify an \irredundant{} chipsplitting game $f$ with its corresponding function $V'\to\ZZ$. The outcome $w=(w_v)_{v\in V}$ of $f$ now satisfies
\[
w_v=-f(v)+\sum_{\substack{p\in V'\\(p,v)\in E}}f(v),
\]
where we write $f(v)=0$ when $v\not\in V'$.

\begin{re}\label{a-valued}
Let $A$ be an abelian group. The definitions above naturally extend from $\ZZ$ to $A$, i.e., to the setting where the number of chips at a \rev{vertex} and number of times a move is repeated are both allowed to be any element of $A$. Here (resp. when $A=\QQ,\RR$), we say that the chip configurations, chipsplitting games and outcomes are {\em $A$-valued} (resp. {\em rational}, {\em real}).
\end{re}

We now define the directed graphs $(V_d,E_d)$
we consider in this paper. For $d\in\NN\cup\{\infty\}$, write
\begin{eqnarray*}
V_d&:=&\{(i,j)\in\ZZ_{\geq0}^2\mid \rev{i+j}\leq d\},\\
E_d&:=&\{(v,v+e)\mid v\in V_{d-1}, e\in\{(1,0),(0,1)\}\}.
\end{eqnarray*}

\rev{We think of $V_d$ as the integer points of the plane triangle delimited by $(0,0)$, $(d,0)$, and $(0,d)$. We consider the hypothenuse as the $d$th diagonal of this figure, and similarly we think of the vertex $(i,j)$ as lying in the $(i+j)$-th diagonal.
To emphasize this we define $\deg(i,j):=i+j$, the {\em degree} of $(i,j)\in\ZZ_{\geq0}^2$. Next, we define some notions about chip configurations on $V_d$ that will be used throughout the paper.}

\begin{de}\label{chipsplitting-notions} \moved{}
Let $w=(w_{i,j})_{(i,j)\in V_d}$ be a chip configuration.
\begin{enumerate}
\item The {\em positive support} of $w$ is $\supp^+(w):=\{(i,j)\in V_d\mid w_{i,j}>0\}$.
\item The {\em negative support} of $w$ is $\supp^-(w):=\{(i,j)\in V_d\mid w_{i,j}<0\}$.
\item The {\em support} of $w$ is $\supp(w):=\{(i,j)\in V_d\mid w_{i,j}\neq0\}=\supp^+(w)\cup\supp^-(w)$.
\item The {\em degree} of $w$ is $\deg(w):=\max\{\deg(i,j)\mid (i,j)\in\supp(w)\}$.
\item We say that $w$ is {\em valid} when $\supp^-(w)\subseteq\{(0,0)\}$.
\item We say that $w$ is {\em weakly valid} when for all $(i,j)\in\supp^-(w)$ one of the following holds:
\begin{samepage}
\begin{enumerate}
\item[(i)] $0\leq i,j\leq 3$,
\item[(ii)] $0\leq i\leq3$ and $\deg(i,j)\geq d-3$, or 
\item[(iii)] $0\leq j\leq3$ and $\deg(i,j)\geq d-3$.
\end{enumerate}
\end{samepage}
\end{enumerate}
\end{de}

\medskip

Figure~\ref{fig:weakly-valid} illustrates the notion of a weakly valid outcome, which will first be used in Section~\ref{subsec:contraction-hyperfield}.

{
\definecolor{cornercolor}{HTML}{FFE194}
\newcommand{\cornercolor}{orange}
\begin{figure}[h]
\begin{tikzpicture}[scale=0.25]
\node (a) at (0,0) {};
\node (b) at (4,0) {};
\node (c) at (12.5,0) {};
\node (d) at (17,0) {};
\node (e) at (13,4) {};
\node (f) at (4,13) {};
\node (g) at (0,17) {};
\node (h) at (0,12.5) {};
\node (i) at (0,4) {};
\node (j) at (4,4) {};
\node (k) at (8.5,4) {};
\node (l) at (4,8.5) {};
\draw[fill=cornercolor]
  (a.center)--(b.center)--(j.center)--(i.center)--cycle;
\draw[fill=cornercolor]
  (c.center)--(d.center)--(e.center)--(k.center)--cycle;
\draw[fill=cornercolor]
  (g.center)--(h.center)--(l.center)--(f.center)--cycle;
\draw
  (b.center)--(c.center)--(k.center)--(j.center)--cycle;
\draw
  (l.center)--(k.center)--(e.center)--(f.center)--cycle;
\draw
  (i.center)--(j.center)--(l.center)--(h.center)--cycle;
%label last area
%add curly braces
\draw[decorate, decoration = {brace, mirror, raise = 2pt}, semithick] (a.center)--(b.center);
\draw[decorate, decoration = {brace, mirror, raise = 2pt}, semithick] (i.center)--(a.center);
\draw[decorate, decoration = {brace, mirror, raise = 2pt}, semithick] (c.center)--(d.center);
\draw[decorate, decoration = {brace, mirror, raise = 2pt}, semithick] (g.center)--(h.center);
%label the curly braces
\node at (2,-1.5) {$4$};
\node at (-1.5,2) {$4$};
\node at (14.75,-1.5) {$4$};
\node at (-1.5,14.75) {$4$};
\end{tikzpicture}
\caption{The corners of the four-entries wide outer ring of the triangle $V_d$. A chip configuration is weakly valid if its negative support is contained in the \cornercolor{} area.}\label{fig:weakly-valid}
\end{figure}
}

\begin{ex}
We depict a chip configuration $w=(w_{i,j})_{(i,j)\in V_d}\in\ZZ^{V_d}$ as a triangle of numbers with~$w_{i,j}$ being the number in the $i$th column from the left and $j$th row from the bottom.

\medskip
\begin{center}
\begin{BVerbatim}
·           ·            ·           ·           1           1           1
· ·         · ·          · ·         1 ·         · 1         · 1         · ·
· · ·       1 · ·        1 1 ·       · 2 ·       · 2 ·       · 2 1       · 3 ·
0 · · ·    -1 1 · ·     -1 · 1 ·    -1 · 1 ·    -1 · 1 ·    -1 · · 1    -1 · · 1
\end{BVerbatim}
\end{center}
\medskip
When $w_{i,j}=0$, we usually write \verb|·| at position $(i,j)$ instead of \verb|0|. In the examples above, we have $d=3$. The leftmost configuration is the initial configuration. From left to right, we obtain the next five configurations by successively executing splitting moves at $(0,0)$, $(1,0)$, $(0,1)$, $(0,2)$, and $(2,0)$, respectively. Finally, we obtain the rightmost configuration $w$ by applying an unsplitting move at $(1,1)$. \rev{The positive support of $w$ is \{(0,3), (1,1), (3,0)\}. Its negative support is \{(0,0)\}. Its support is the union of the previous two sets. The degree of $w$ is 3, since the furthermost diagonal that supports $w$ is the third one. All configurations shown in this example are valid and therefore weakly valid.}
\end{ex}

\rev{The notion of a valid outcome is essential for establishing the connection between chipsplitting games and \ratmodels{}: Proposition~\ref{prop:reduced_models2valid_outcomes} shows that valid real chipsplitting outcomes correspond precisely to reduced \redratmodels{}. For instance, the model $\mathcal M_{\indep}$ from Section~\ref{sec:intro} corresponds to the middle configuration in the above sequence.}

\begin{re}
The notion of chipsplitting games is inspired by that of chipfiring games. For a thorough treatment of the latter, see~\cite{Klivans}. In fact, by using powers of two one can prove that our chipsplitting games are equivalent to certain chipfiring games, provided the latter allow `unfiring', or reversing a firing move. All notions of Definiton~\ref{de:chipsplitting} have chipfiring equivalents. In this paper, we use chipsplitting games as they relate more directly to the statistical models of Section~\ref{sec:models}.
\end{re}

\rev{We can now state} our main result in the language of valid outcomes.

\begin{thm}\label{thm:main_outcomes}
Let $n\leq 4$ and let $w$ be a valid outcome with a positive support of size $n+1$. Then \[\deg(w)\leq 2n-1.\]
\end{thm}

\rev{
The equivalence of Theorems~\ref{thm:main_models} and~\ref{thm:main_outcomes} will be proven in Proposition~\ref{are-equivalent}.} We will prove Theorem~\ref{thm:main_outcomes} in Sections~\ref{sec:supp+<=3}–\ref{sec:supp+==5} (See Theorems ~\ref{thm:valid-outcomes-n-three},~\ref{thm:possup=4} and~\ref{thm:pos_supp=5}).

\section{Fundamental models}\label{sec:models}

\rev{In this section we develop the statistical side of our paper and prove our main classification theorem  using Theorem~\ref{thm:main_models}.}

\mybreak
A \emph{one-dimensional \emph{(parametric, discrete)} algebraic statistical model}
%$\mathcal M$
is a
%one-dimensional
subset of %the simplex
$\Delta_n$
%:= \{p\in \mathbb R^{n+1}_{\geq 0} \mid \sum_\nu p_\nu = 1\}$
%of probability distributions on $n+1$ events for some $n\in \mathbb N$. More precisely, it
which is the image of a rational map $p\colon I \to \Delta_n$ whose components $p_0(t),\dotsc, p_n(t)$ are rational functions in $t$, where $I \subseteq \mathbb R$ is a union of closed intervals such that $p(\partial I) \subseteq \partial \Delta_n$.
Alternatively,
%$\mathcal M$
such a model can be described as the intersection of $\Delta_n$ with a parametrized curve $\{\gamma(t)\mid t\in\mathbb R\}\subseteq \mathbb R^{n+1}$ with rational entries in the $t$.
%This characterization is equivalent to the parametric one apart from the fact that it allows the empty model.

\mybreak
Let $\mathcal M\subseteq \Delta_n$ be a one-dimensional algebraic model which is parametrized by the rational functions $p_0(t),\dotsc, p_n(t).$ The equation $\sum_\nu p_\nu(t) = 1$ holds for infinitely many and thus for all $t\revml{\in \mathbb R}$.
We multiply it by the least common denominator of the $p_\nu(t)$ to obtain an equation of the form $\sum_\nu a_\nu(t) = b(t)$, where $a_0(t),\dotsc,a_n(t), b(t)$ are polynomials in $t$. Thus, $\mathcal M$ is determined by a collection $(a_0,\dotsc, a_n, b)$ of polynomials in $t$ satisfying $\sum_\nu a_\nu = b$. The parametrization of $\mathcal M$ is recovered by setting $p_\nu = a_\nu/b$, where we may assume that the polynomials $a_0,\dotsc,a_n,b$ share no factor common to all of them.

\mybreak
In maximum likelihood estimation, one seeks to maximize
\rrev{the likelihood
$\mathcal L_u(p) = \operatorname{const}(u)\prod_{\nu}p_\nu^{u_\nu}$ of observing a given empirical distribution $u\in\Delta_n$, over all $p\in \mathcal M$. The term $\operatorname{const}(u)$ is a multinomial coefficient that depends only on $u$, so it can be dropped. Then, the problem is reduced to maximizing the \emph{log-likelihood} $\ell_u(p) \coloneqq \sum_\nu u_\nu \log(p_\nu) \propto \log \mathcal L_u(p).$}
This can be accomplished by first finding all the critical points of $\mathcal \ell_u$. When $\mathcal M$ is one-dimensional, finding these critical points amounts to finding the zeros of the derivative $\ell_u(p(t))'$ with respect to $t$. In our notation, we have 
\[
\ell_u(p(t))' = \sum_\nu u_\nu \frac{a_\nu'}{a_\nu} - \sum_\nu u_\nu\frac{b'}{b},
\]
a rational expression in $t$ which we abbreviate as $\ell_u'$.
In algebraic statistics, the \emph{maximum likelihood degree} $\operatorname{mld}(\mathcal M)$ of $\mathcal M$ is the number of solutions over $\mathbb C$ to this equation for general $u\in \mathbb C^n$. In our case, this number can be determined in terms of the roots of the $a_\nu$ and $b$, as the next lemma shows.

\begin{lemma}\label{complex-factors}
Let $f$ be the product of all the distinct complex linear factors occurring among the polynomials $a_0,\dotsc, a_n, b.$ Then $\operatorname{mld}(\mathcal M) = \deg(f)-1$.
\end{lemma}
\begin{proof}
Every factor of a polynomial $g$ with \rev{multiplicity} $k$ occurs in $g'$ with multiplicity $k-1$. So the expression
\[
f \ell_u' = \sum_\nu u_\nu \frac{fa_\nu'}{a_\nu} - \sum_\nu u_\nu\frac{fb'}{b}
\]
is a polynomial in $t$ of degree $\deg(f)-1$. All roots of the rational function $\ell_u'$ are roots of $f\ell'_u$. It remains to show that no new roots were introduced. That is, that no root of $f$ is also a root of $f\ell'_u$. Thus, let $\zeta$ be a complex linear factor of $f$ and $\zeta_0\in \mathbb C$ its derivative. Rewrite $f\ell_u'$ as
\[
\sum_{\nu=0}^{n+1}u_\nu\frac{fa_\nu'}{a_\nu}
\]  
with $a_{\rev{n+1}}:=b$ and $u_{n+1} := -\sum_{\nu=0}^n u_\nu.$ For $\nu = 0,\dotsc, n+1,$ write $a_\nu = \zeta^{k_\nu}r_\nu$ and $f = \zeta r$ such that $\zeta\nmid r_\nu, r$. Then for all $\nu$ we have $fa_\nu'/a_\nu = \zeta r k_\nu \zeta_0/\zeta + \zeta r r_\nu'/r_\nu \equiv \zeta_0 k_\nu r \pmod \zeta$. Consequently,
\[
f\ell_u' \equiv \zeta_0 r \sum_{\nu=0}^{n+1}u_\nu k_\nu \equiv \zeta_0 r \sum_{\nu=0}^{n}u_\nu (k_{\nu} - k_{n+1}) \pmod \zeta.
\]
Not all the $(k_\nu - k_{n+1})$ for $\nu=0,\dotsc,n$ can be zero since $\zeta$ is a factor of some $a_\nu$ for $\nu=0,\dotsc, n+1,$ but not all of them \rev{since by assumption the $a_0,\dotsc, a_{n+1}$ share no factor common to all of them.} Hence, because the $u_\nu$ are generic we may assume that $\sum_\nu u_\nu (k_\nu - k_{n+1}) \neq 0$. \rev{Since $\zeta$ divides $f$ only once, we have} $\zeta_0 r\not\equiv 0 \pmod \zeta$. \rev{Therefore}, $f\ell_u' \not\equiv 0 \pmod \zeta$, so $\zeta \nmid f\ell_u'$.
\end{proof}
\newpar
In this paper we are interested in classifying one-dimensional models of ML degree \emph{one}. The next proposition is the first step in our classification.

\begin{proposition}\label{first-step} Every one-dimensional discrete model $\mathcal M$ of ML degree one has a parametrization of the form
\[
p\colon [0,1]\to \Delta_n,\quad t \mapsto (w_\nu t^{i_\nu}(1-t)^{j_\nu})_{\nu=0}^n 
\]
for some nonnegative exponents $i_\nu,j_\nu$ and positive real coefficients $w_\nu$ for $\nu = 0,\dotsc,n$. 
\end{proposition}

\begin{proof}
Let $\mathcal M$ be defined by the polynomials $a_0,\dotsc, a_n, b$ with $\sum_\nu a_\nu = b$.
By Lemma~\ref{complex-factors}, these polynomials split as products of the same two complex factors.
The $n+1$ faces of $\Delta_n$ lie on the $n+1$ coordinate hyperplanes of $\mathbb R^n$. Thus, the set $I$ in the parametrization $p\colon I\to \mathcal M$ is a single closed interval because $p(\partial I) \subseteq \partial \Delta_n$ and the $a_\nu,b$ have exactly two zeros among them. In particular, these zeros are real and coincide with the endpoints of $I$. Without changing $\mathcal M$, we may reparametrize and assume that $I=[0,1]$. 
We may write
\begin{align*}
a_\nu(t) &= w_\nu t^{i_\nu}(1-t)^{j_\nu} \\
b(t) &= w t^i(1-t)^j,
\end{align*}
for $w_\nu, w\in \mathbb R_{> 0}$ and $i_\nu, j_\nu,i,j\in \mathbb Z_{\geq 0}$ for all $\nu$. 
If $i>0$, then $i_\nu=0$ for some $\nu$ and we arrive at a contradiction by evaluating
the equation $\sum_\nu a_\nu = b$ at $t=0$.
So $i=0$. Similarly, we must have $j=0$.
By dividing by $w$ we now arrive at the required form for $p$.
\end{proof}

Thus, our goal is to provide a classification of the parametrizations of models specified by Proposition~\ref{first-step}, \rrev{i.e.\ \ratmodels{}}. We will show how these models can be built up from progressively simpler models, the simplest of which we will call `fundamental models'.

\mybreak
Proposition~\ref{first-step} shows that every \ratmodel{} $\mathcal M\subseteq \Delta_n$ can be represented by a finite sequence $(w_\nu, i_\nu, j_\nu)_{\nu=0}^n$ for some nonnegative exponents $i_\nu, j_\nu$ and positive real coefficients $w_\nu$. The degree of $\mathcal M$ as an algebraic variety, denoted by $\deg(\mathcal M)$, is $\max\{\deg(i_\nu,j_\nu)\mid \nu\in \{0,\dotsc,n\}\}$ where $\deg(i,j):=i+j$.

\mybreak
We consider two \ratmodels{} in $\Delta_n$ equivalent if they differ only by a relabeling of the coordinates on $\Delta_n$. \rev{Nevertheless, we shall attempt to maintain some consistency when indicating points of $\Delta_n$ indexed by pairs $(i,j)$ by ordering the coordinates of these points lexicographically.} Although $(w_\nu, i_\nu, j_\nu)_{\nu = 0}^n$ and $(w_\nu, j_\nu, i_\nu)_{\nu = 0}^n$ represent the same subset of $\Delta_n$, we shall count these two representations as distinct \ratmodels{} unless they are equal up to reordering. These two models differ by the reparametrization $t\mapsto p(t-1)$.

\mybreak
We now define our first simpler subclass of the class of \ratmodels{}.

\begin{definition}
A \ratmodel{} represented by $(w_\nu, i_\nu, j_\nu)_{\nu = 0}^n$ is \emph{reduced} if the exponent pairs $(i_\nu, j_\nu)$ are not equal to $(0,0)$ and pairwise distinct. \rev{For brevity we call such a model a \emph{reduced R1d model}.}
\end{definition}

\begin{proposition} \label{reduced-to-model}
Every one-dimensional discrete model of ML degree one is the image of a reduced \redratmodel{} under a chain of linear embeddings
of the form
\begin{align}\label{embedding-one}
\Delta_{n-1}\to \Delta_n,\quad
(p_0,\dotsc,\hat p_\nu, \dotsc, p_n)
&\mapsto
(\lambda p_0,
\dotsc,
%\lambda p_{\nu-1},
1-\lambda,
%\lambda p_{\nu+1},
\dotsc,\lambda p_n),
\quad \lambda\in[0,1]
\end{align}
or
\begin{align}\label{embedding-two}
\Delta_{n-1}\to \Delta_n,\quad
(p_0,\dotsc, p_\nu,\dotsc,\hat p_\mu,\dotsc, p_n)
&\mapsto
\left(p_0,\dotsc,
%p_{\nu-1},
\lambda p_\nu,
%p_{\nu+1},
\dotsc,
%p_{\mu-1},
(1-\lambda)p_\nu,
%p_{\mu+1},
\dotsc,
p_n\right),
\quad \lambda\in[0,1].
\end{align}
\end{proposition}

\begin{proof}
Let $(w_\nu, i_\nu, j_\nu)_{\nu = 0}^n$ represent a \ratmodel{} $\mathcal M$. If $(i_\nu,j_\nu) = (0,0)$ for some $\nu$ then $w_\nu < 1$. Let $\lambda := 1-w_\nu$. Then $\mathcal M$ is the image under the linear embedding~\eqref{embedding-one}
of the \rev{reduced} \redratmodel{} represented by
\[
(w_\iota/(1-w_\nu),i_\iota,j_\iota)_{\iota = 0, \iota\neq\nu}^n.
\]
Similarly, suppose that $(i_\nu, j_\nu)=(i_\mu,j_\mu)$ for some $\nu\neq\mu$ and let $\lambda:=w_\nu/(w_\nu + w_\mu)$. Then $\mathcal M$ is the image under the linear embedding~\eqref{embedding-two}
of the \rev{reduced} \redratmodel{} represented by
\[(w_\iota + \delta_{\iota\nu}w_\mu,i_\iota,j_\iota)_{\iota = 0, \iota \neq \mu}^n.\qedhere\]
\end{proof}

\begin{re}\label{re:reduced_models_suffice}
If $\Delta_n$ contains a \ratmodel{} of degree $d$, then $\Delta_{n'}$ must contain a reduced \redratmodel{} of degree $d$ for some $n'\leq n$. Therefore, to prove Theorem~\ref{thm:main_models} it is enough to consider reduced \redratmodels{} only.
\end{re}

\begin{definition}\label{def:fundamental}
A reduced \redratmodel{} represented by $(w_\nu,i_\nu,j_\nu)_{\nu=0}^n$ is a \emph{fundamental model} if, given the exponents $(i_\nu, j_\nu)$, the weights $(w_\nu)$ are uniquely determined by the constraint $\sum_\nu p_\nu = 1$.
\end{definition}

Thus, for any given set of exponents $(i_\nu, j_\nu)$, we can check whether there is a fundamental model with these exponents by solving a system of affine-linear equations in the weights $w_\nu$. Similarly, the set of reduced \redratmodels{} with these fixed exponents is always an affine-linear half space of dimension at most $n+1$.

\begin{example}
Consider the sequence of exponents $((2,0), (1,1), (0,2))$. The polynomial constraint $w_0 t^2 + w_1 t (1-t) + w_2 (1-t)^2 = 1$ leads to the affine-linear system $w_0 - w_1 + w_2 = 0$, $w_1 -2w_2 = 0$, $w_2 - 1 = 0$. The unique solution $(1,2,1)$ defines the fundamental model $t\mapsto (t^2,2t (1-t),(1-t)^2)$.
\end{example}

We shall now see that every reduced \redratmodel{} can be constructed from finitely many fundamental models in a finite number of steps. For this, we represent a reduced \redratmodel{} by the function $f\colon\mathbb Z^2\to \mathbb R_{\geq 0}$ that sends an exponent pair $(i_\nu,j_\nu)$ to its associated coefficient $w_\nu$. We call the set of exponent pairs $(i_\nu, j_\nu)$ the \emph{support} of $\mathcal M$. It equals $\supp(f)$.

\begin{definition}\label{composite}
Let $\mathcal M_1$ and $\mathcal M_2$ be reduced \redratmodels{} represented by the functions $f_1, f_2 : \mathbb Z^2\to \mathbb R_{\geq 0}$. Let $0<\mu<1$. The \emph{composite} $\mathcal M_1 *_{\mu}\mathcal M_2$ of $\mathcal M_1$ and $\mathcal M_2$ is the reduced \redratmodel{} represented by
\[
g:\mathbb Z^2\to \mathbb R_{\geq 0}, \quad g(i,j) \coloneqq \mu f_1(i,j) + (1-\mu) f_2(i,j).
\]
\end{definition}

\begin{proposition} \label{fundamental-to-reduced}
Every reduced \redratmodel{} is the composite of finitely many fundamental models.
\end{proposition}
\begin{proof}
Let $\mathcal M$ be a reduced \redratmodel{} represented by $(w_\nu, i_\nu, j_\nu)_{\nu=0}^n$.
If $n\leq 1$ then $\mathcal M$ is fundamental. So, let $n\geq 2$ and $\mathcal M$ not fundamental. It suffices to show that $\mathcal M$ is the composite of two reduced \redratmodels{} whose supports are proper subsets of $S$.
%Let $\mathcal M$ be a reduced model represented by $g\colon\mathbb Z^2\to \mathbb R_{\geq 0}$. Write $S\coloneqq \supp(g)$, call it the \emph{support} of $\mathcal M$. If $n\leq 1$ then $\mathcal M$ is fundamental because it is reduced. Hence we may assume $n\geq 2$ and $\mathcal M$ not fundamental. It suffices to show that $\mathcal M$ is the composite of two models whose supports are proper subsets of $S$. Since $\mathcal M$ is not fundamental, there exist $x_0,\dotsc,x_n\in\mathbb R$, not all zero, such that \sum_{\}
Since $\mathcal M$ is not fundamental, there exist $x_0,\dotsc,x_n\in\RR$, not all zero, such that $\sum_{\nu=0}^n x_\nu t^{i_\nu}(1-t)^{j_\nu} = 0$. Since this equality holds for all $t\in(0,1)$, we have at least one positive and one negative $x_\nu$. Let
\[
\lambda := \min\{w_\nu/|x_\nu| \mid\nu\in\{0,\ldots,n\},x_\nu < 0\},\quad u_\nu := w_\nu + \lambda x_\nu\mbox{ for }\nu\in\{0,\ldots,n\},
\]
and $S_1 := \{(i_\nu, j_\nu) \mid\nu\in\{0,\ldots,n\}, u_\nu \neq 0\}$. Then we have $\lambda > 0$ and 
$u_\nu \geq 0$ for all $\nu\in\{0,\ldots,n\}$, the latter of which we verify by distinguishing between the cases $x_\nu \geq 0$ and $x_\nu < 0$.
For all $\nu$ we have $u_\nu=0$ if and only if $x_\nu<0$ and $\lambda=w/|x_\nu|$. Thus $S_1$ is a nonempty proper subset of $S$.
Since $\sum_{\nu=0}^n u_\nu s_\nu = 1,$  the coefficients $u_\nu$ for $(i_\nu,j_\nu)\in S_1$ define a reduced \redratmodel{} $\mathcal M_1$ with support $S_1$. Let
\[
\mu := \min\{w_\nu/u_\nu \mid \nu\in\{0,\ldots,n\},u_\nu \neq 0\},\quad v_\nu := (w_\nu - \mu u_\nu)/(1-\mu)\text{ for } \nu\in\{0,\ldots,n\},
\]
and $S_2 := \{(i_\nu, j_\nu)\mid \nu\in\{0,\ldots,n\}, v_\nu\neq 0\}$.
Then $\mu>0$. Since at least one of the $x_\nu$ is positive, we have $u_\nu > w_\nu$ for some $\nu$, and thus $\mu<1.$ 
We have $v_\nu \geq 0$ by the definition of $\mu$
and $v_\nu=0$
if and only if $u_\nu\neq 0$ and $\mu = w_\nu/u_\nu$.
Thus $S_2$ is a nonempty proper subset of $S$ and we have $S_1\cup S_2 = S$.
Since $\sum_{\nu=0}^n v_\nu x_\nu = 1$, the coefficients $v_\nu$ for $(i_\nu, j_\nu)\in S_2$ define a reduced \redratmodel{} $\mathcal M_2$ with support $S_2$.
We conclude by noting that $w_\nu = \mu u_\nu + (1-\mu) v_\nu$ for all $\nu\in\{0,\dotsc,n\}$. Thus, $\mathcal M = \mathcal M_1 *_\mu \mathcal M_2$.
\end{proof}

\begin{remark}\label{re:fundamental_models_suffice}
If a reduced \redratmodel{} $\mathcal M \subseteq \Delta_n$ is not fundamental, then by Proposition~\ref{fundamental-to-reduced} there exists $n'<n$ and a fundamental model in $\Delta_{n'}$ of the same degree as $\mathcal M$. Thus, it suffices to prove Theorem~\ref{thm:main_models} for fundamental $\mathcal M$. In turn, Theorem~\ref{thm:main_models} implies that there are only finitely many fundamental models in $\Delta_n$ for $n\leq 4$. This is because for all $d$ there are only finitely many subsets $S\subseteq \mathbb Z^2$ that can be the support of a fundamental model $\mathcal M$ of degree $d$, and $S$ determines $\mathcal M$ uniquely.
\end{remark}

Our classification of one-dimensional discrete models of ML degree one is now complete. We summarize it in Theorem~\ref{classification-theorem}, all elements of which we already established in this section. Part (c) uses Theorem~\ref{thm:main_models}, which we will prove in Sections~\ref{sec:supp+<=3}--\ref{sec:supp+==5}. We visualize our classification in Figure~\ref{classification-figure}.

\begin{theorem}\label{classification-theorem}\phantom{text}
\begin{enumerate}
\item Every one-dimensional discrete model of ML degree one $\mathcal M\subseteq \Delta_n$ is the image of a reduced \redratmodel{} $\mathcal M'\subseteq \Delta_{n'}$ under a linear embedding $\Delta_{n'}\to \Delta_n$ for some $n'\leq n$.

\item  Every reduced \redratmodel{} $\mathcal M'\subseteq \Delta_{n'}$
%at the place $S$
can be written as the composite
\[
\mathcal M' = \mathcal M_1 *_{\mu_1} (\dotsb *_{\mu_{m-1}} (\mathcal M_{m-1} 
*_{\mu_m} \mathcal M_m)\dotsc)
\]
of finitely many
%at most $2^{|S|}$
fundamental models $\mathcal M_1,\dotsc, \mathcal M_m$.
%with $\supp(\mathcal M') = \supp(\mathcal M_1)\cup \dotsb \cup\supp(\mathcal M_m)$.
%respectively,
%with $S = S_1 \cup \dotsb \cup S_m$.
%\item If $\mathcal M'$ is fundamental then it is the only model at the place $S$. If not, then $\mathcal M'$ is part of an infinite family of non-fundamental models at the place $S$. More precisely, this family is indexed by an affine-linear half-space of dimension $n'+1-\dim\Span_{\mathbb R}\{t^i(1-t)^j\mid (i,j)\in S\}$.\qed
\item For $n\leq 4$, there are only finitely many fundamental models in $\Delta_n$.\qed
\end{enumerate}
\end{theorem}

\begin{figure}[h!]
\[\begin{tikzcd}
|[draw=black, rectangle]|
\parbox{2cm}{\centering Fundamental models}
&
|[draw=black, rectangle]|
\parbox{2cm}{\centering Reduced \redratmodels}
&
|[draw=black, rectangle]|
\parbox{4cm}{\centering\Ratmodels}
	\arrow["\subseteq"{description}, draw=none, from=1-1, to=1-2]
	\arrow["\subseteq"{description}, draw=none, from=1-2, to=1-3]
	\arrow["{\text{(Prop.\ \ref{fundamental-to-reduced})}}"', shift left = 2, bend right = 45, distance = 1cm, looseness=1, dashed, shorten <=11pt, shorten >=11pt, from=1-2, to=1-1]
	\arrow["{\text{(Prop.\ \ref{reduced-to-model})}}"', shift left = 2, bend right = 45, distance = 1cm, looseness=1, dashed, shorten <=11pt, shorten >=11pt, from=1-3, to=1-2]
\end{tikzcd}\]
\caption{A classification of one-dimensional discrete models of ML degree one \rev{(right box)}.
%Every such model is the image of a reduced model under a linear embedding (Proposition~\ref{reduced-to-model}). In turn, every reduced model can be written as the composite (Definition~\ref{composite}) of finitely many fundamental models (Proposition~\ref{fundamental-to-reduced}). In Sections~\ref{sec:supp+<=3}--\ref{sec:supp+==5} we prove that there are only finitely many fundamental models for $n\leq 4$.
}
\label{classification-figure}
\end{figure}

\begin{example}
Let us classify all one-dimensional models $\mathcal M$ of ML degree one in the triangle $\Delta_2$, up to coordinate permutations. The unique \ratmodel{} $\mathcal M_0$ in $\Delta_1$ is parametrized by $t\mapsto (t, (1-t))$. Since $\mathcal M_0 *_{\mu} \mathcal M_0 = \mathcal M_0$, all \ratmodels{} in $\Delta_2$ are either fundamental or non-reduced. Theorem~\ref{thm:main_models} gives a bound for the algebraic degree of $\mathcal M$: we have $\deg(\mathcal M) \leq 3$. Hence, to find all fundamental models we check all possible sets of exponent pairs (or \emph{supports}) $S\subseteq \{(i,j)\mid 0<i+j\leq 3\}$ of size $n+1 = 3.$ We report the results in Figure~\ref{fundamental-n-3-figure}.

\mybreak
As for non-reduced \redratmodels{}, there are up to coordinate permutations only two linear embeddings $\Delta_1\to \Delta_2$ of the form~\eqref{embedding-one} or~\eqref{embedding-two} that can be used to construct $\mathcal M$ from $\mathcal M_0$. These can vary with the parameter $\lambda$ and are reported in Figure~\ref{nonreduced-n-3-figure} for $\lambda = 1/3$.
\end{example}
%\begin{example}
%\rev{Consider two reduced \redratmodels{} given as follows:}
%\begin{align*}
%\mathcal M_{1}:{} t\mapsto (t^2, 2t(1-t), (1-t)^2), %\quad \quad
%\mathcal M_2:{} t\mapsto ((1-t), t(1-t), (1-t)^2).
%\end{align*}
%\rev{These models live in $\Delta_2$. The composite $\mathcal M_{1} *_{1/2} \mathcal M_{2}$ lives in $\Delta_3$ and is given by
%\[
%t\mapsto \left(t^2, \frac{3}{2}t(1-t), \frac{1}{2}(1-t)^2, \frac{1}{2}(1-t)\right).
%\]}
%\end{example}

\begin{figure}[h]
\includegraphics[scale=0.3]{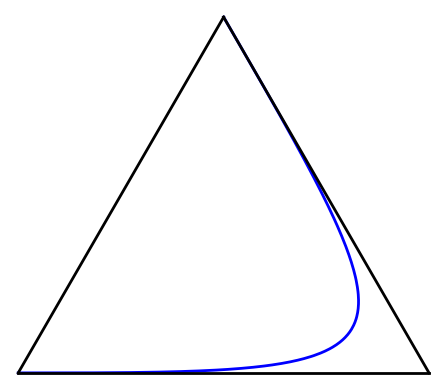}
\includegraphics[scale=0.3]{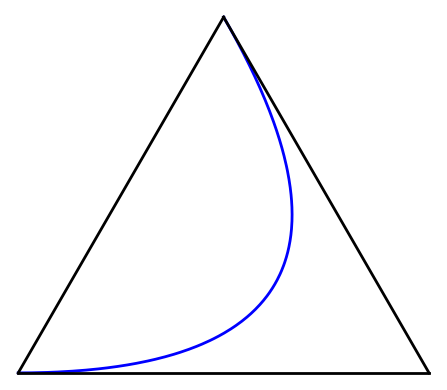}
\includegraphics[scale=0.3]{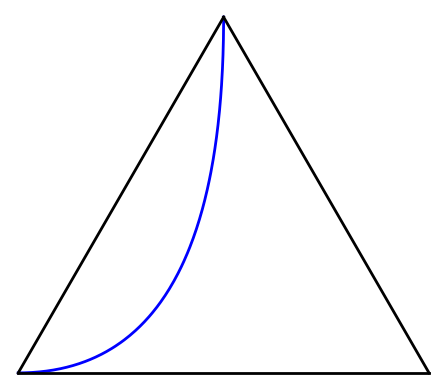}
\caption{Fundamental models in $\Delta_2$. These correspond to the parametrizations $t\mapsto ((1-t)^3, 3t(1-t), t^3)$, $t\mapsto ((1-t)^2, 2(1-t)t, t^2)$, and $t\mapsto ((1-t),t(1-t), t^2),$ from left to right. Their supports are \{(0,3), (1,1), (3,0)\}, \{(0,2), (1,1), (2,0)\}, and \{(0,1), (1,1), (2,0)\}, respectively. In $\Delta_2$ there is a further fundamental model with support \{(0,2),(1,0),(1,1)\}, but it is identical to the third model in this picture after a permutation of the coordinates of $\Delta_n$ and the reparametrization $t\mapsto 1-t$.}
\label{fundamental-n-3-figure}
\end{figure}

\begin{figure}[h]
\includegraphics[scale=0.3]{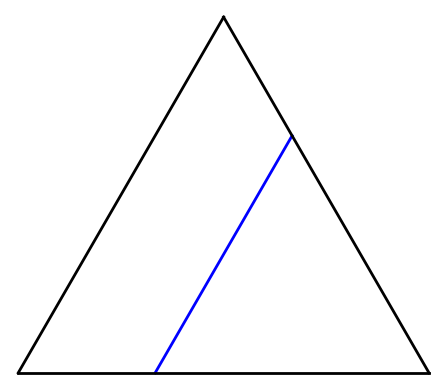}
\includegraphics[scale=0.3]{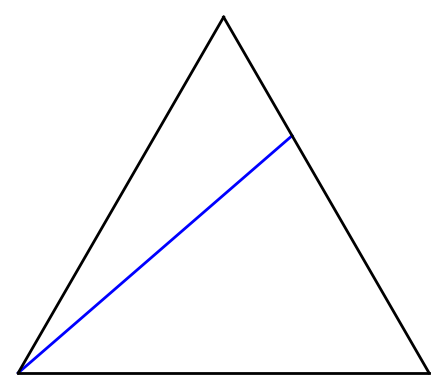}
\caption{Non-reduced \redratmodels{} in $\Delta_2.$ These arise from linear embeddings $\Delta_1\to \Delta_2$ of type \eqref{embedding-one} and \eqref{embedding-two}, respectively. They are given by $t\mapsto ((1-\lambda)t, \lambda, (1-\lambda)(1-t))$ and $t\mapsto ((1-t), \lambda t, (1-\lambda)t)$, where $\lambda:=1/3.$ All other non-reduced one-dimensional models of ML degree one in $\Delta_2$ arise from these two by varying $\lambda$ and permuting the coordinates of $\Delta_2$.}
\label{nonreduced-n-3-figure}
\end{figure}

\section{Chipsplitting games}\label{sec:chipsplitting}

In this section we lay some groundwork for proving Theorem~\ref{thm:main_outcomes}. \rev{In particular, \emph{symmetry} structures will help us cut down the number of cases considered, while \emph{Pascal equations} will help us distinguish chipsplitting outcomes from non-outcomes}.%, \rev{collecting} results about chipsplitting outcomes that will help us later.

\subsection{Symmetry}\label{subsec:symmetry}
For every $d\in\NN\cup\{\infty\}$, define an action of the group $S_2=\langle(12)\rangle$ on $\ZZ^{V_d}$ by setting
\[
(12)\cdot (w_{i,j})_{(i,j)\in V_d}:=(w_{j,i})_{(i,j)\in V_d},
\]
where clearly $(12)\cdot((12)\cdot w))=w$ for all $w\in\ZZ^{V_d}$. We also let $S_2$ act on $V_d$ by $(12)\cdot (i,j):=(j,i)$.

\mybreak
The initial configuration is fixed by $S_2$. Let $w\in\ZZ^{V_d}$, $p\in V_{d-1}$, and let $\widetilde{w}$ be the result of applying an (un)splitting move at $p$ to $w$. Then $(12)\cdot\widetilde{w}$ is the result of applying an (un)splitting move at $(12)\cdot p$ to $(12)\cdot w$. So we see that if $w$ is the outcome of an \irredundant{} chipsplitting game $f$, then $(12)\cdot w$ is the outcome of the chipsplitting game $(i,j)\mapsto f(j,i)$. Hence the space of outcomes is closed under the action of $S_2$. Let $w\in\ZZ^{V_d}$ be a chip configuration. Then
\begin{align*}
\supp^+((12)\cdot w)&=(12)\cdot\supp^+(w),& \supp^-((12)\cdot w)&=(12)\cdot\supp^-(w),\\
\supp((12)\cdot w)&=(12)\cdot\supp(w),&\deg((12)\cdot w)&=\deg(w).
\end{align*}
Furthermore, $w$ is (weakly) valid if and only if $(12)\cdot w$ is (weakly) valid.

\begin{example}
\rev{The following two valid outcomes are mapped to each other by the element $(12)$.}

\medskip
\begin{center}
\begin{BVerbatim}
 1           ·               
 · 1         1 ·       
 · 2 ·       · 2 1    
-1 · 1 ·    -1 · · 1
\end{BVerbatim}
\end{center}

\rev{The configuration corresponding to the model $\mathcal M_{\indep}$ from Section~\ref{sec:intro} is invariant under the $S_2$-action.}

\end{example}

\subsection{Pascal equations}
Another way to study the space of outcomes is via the set of linear forms that vanish on it. A {\em linear form} on $\ZZ^{V_d}$ is a function $\ZZ^{V_d}\to\ZZ$ of the form 
\[
(w_{i,j})_{(i,j)\in V_d}\mapsto\sum_{(i,j)\in V_d}c_{i,j}w_{i,j},
\]
which we will denote by $\sum_{(i,j)\in V_d}c_{i,j}x_{i,j}$. The group $S_2$ acts on the space of linear forms on $\ZZ^{V_d}$ via 
\[
(12)\cdot \sum_{(i,j)\in V_d}c_{i,j}x_{i,j}:=\sum_{(i,j)\in V_d}c_{j,i}x_{i,j}.
\]

\begin{de}
We say that a linear form $\sum_{(i,j)\in V_d}c_{i,j}x_{i,j}$ is a {\em Pascal equation} when 
\[
c_{i,j}=c_{i+1,j}+c_{i,j+1}
\]
for all $(i,j)\in V_{d-1}$. 
\end{de}

This terminology is inspired by the Pascal triangle, whose entries satisfy the same condition. The space of Pascal equations is closed under the action of $S_2$.

\mybreak

\rev{Pascal equations will help us throughout the rest of this article to distinguish chip configurations which are outcomes from those which are not. In particular, these equations will play an essential role in formulating and proving the Invertibility Criterion (Proposition~\ref{invertibility-criterion}), Hyperfield Criterion (Proposition~\ref{hyperfield-criterion}), and Hexagon Criterion (Proposition~\ref{hexagon-criterion}). }

\begin{prop}\label{prop:basis_via_1st_row/1st_column}
\revml{Let $(a_0,\dotsc, a_d)$ be any vector of $d+1$ integers.}
%Let $(a_k)_{k=0}^d\in\ZZ^{\NN_{\leq d}}$ be any vector.
\begin{enumerate}
\item There exists a unique Pascal equation $\sum_{(i,j)\in V_d}c_{i,j}x_{i,j}$ such that $c_{0,j}=a_j$ for all $0\leq j\leq d$.
\item There exists a unique Pascal equation $\sum_{(i,j)\in V_d}c_{i,j}x_{i,j}$ such that $c_{i,0}=a_i$ for all $0\leq i\leq d$.
\end{enumerate}
\end{prop}
\begin{proof}
(a)
Set $c_{0,j}:=a_j$ for all integers $0\leq j\leq d$ and define
\[
c_{i+1,j}:=c_{i,j}-c_{i,j+1}
\]
for all $(i,j)\in V_d$ via recursion on $i>0$. Then $\sum_{(i,j)\in V_d}c_{i,j}x_{i,j}$ is a Pascal equation such that $c_{0,j}=a_j$ for all integers $0\leq j\leq d$. Clearly, it is the only Pascal equation with this property.

\itembreak
(b)
\rev{Let $d_{i,j}\coloneqq c_{j,i}$, so that}
\[
(12)\cdot \sum_{(i,j)\in V_d}c_{i,j}x_{i,j}=\sum_{(i,j)\in V_d}d_{i,j}x_{i,j}.
\]

Then $c_{k,0}=a_k$ if and only if $d_{0,k}=a_k$ and hence the statement follows from (a). 
\end{proof}

Our next goal is to prove that a chip configuration is an outcome if and only if all Pascal equations vanish at it. 

\begin{prop}\label{prop:pascal-equations-vanish-at-outcomes}
Let $w\in\ZZ^{V_d}$ be a chip configuration. Then the value at $w$ of any given Pascal equation on $\ZZ^{V_d}$ is invariant under (un)splitting moves. In particular, all Pascal equations on $\ZZ^{V_d}$ vanish at all outcomes.
\end{prop}
\begin{proof}
Let $w=(w_{i,j})_{(i,j)\in V_d}$ be a chip configuration and suppose we obtain $\widetilde{w}=(\widetilde{w}_{i,j})_{(i,j)\in V_d}$ from $w$ by applying a chipsplitting move at $(i',j')\in V_{d-1}$. Let $\sum_{(i,j)\in V_d}c_{i,j}x_{i,j}$ be a Pascal equation. Then we see that
\[
\sum_{(i,j)\in V_d}c_{i,j}\widetilde{w}_{i,j}=\sum_{(i,j)\in V_d}c_{i,j}\left\{\begin{array}{cl}w_{i,j}-1&\mbox{if $(i,j)=(i',j')$,}\\w_{i,j}+1&\mbox{if $(i,j)=(i'+1,j)$,}\\w_{i,j}+1&\mbox{if $(i,j)=(i',j'+1)$,}\\w_{i,j}&\mbox{otherwise}\end{array}\right\}=\sum_{(i,j)\in V_d}c_{i,j}w_{i,j}
\]
since $c_{i'+1,j'}+c_{i',j'+1}-c_{i',j'}=0$, which proves the first claim. For the second claim it suffices to note that
all Pascal equations vanish at the initial configuration.
\end{proof}

Let $w=(w_{i,j})_{(i,j)\in V_d}$ be a degree-$e$ chip configuration. Then there exists a unique \irredundant{} chipsplitting game that uses only moves at $(i,j)\in V_{d-1}$ with $\deg(i,j)=e-1$ and that sets the values $w_{0,e},w_{1,e-1},\ldots,w_{e-1,1}$ to $0$. Note that these moves do not alter the alternating sum $\sum_{k=0}^e (-1)^kw_{k,e-k}$. So, if $\sum_{k=0}^e (-1)^kw_{k,e-k}=0$, this chipsplitting game also sets $w_{e,0}$ to $0$. This motivates the following definition.

\begin{de}
Let $w=(w_{i,j})_{(i,j)\in V_d}$ be a degree-$e$ chip configuration such that 
\[
\sum_{k=0}^e (-1)^kw_{k,e-k}=0.
\]
The {\em retraction} of $w$ is the unique chip configuration obtained from $w$ using moves at \rev{vertices} $(i,j)\in V_{d-1}$ with $\deg(i,j)=e-1$ such that $\deg(w)<e$.
\end{de}

\begin{example} \rev{In the following picture, the rightmost chip configuration is the retraction of the leftmost one. The retraction is obtained by a sequence of three chipsplitting moves on the second ($\deg(i,j) = 2$) diagonal of the grid: one unsplitting move followed by two splitting moves. This works because in the leftmost configuration, the alternating sum of the entries in the outermost diagonal ($\deg(i,j) = 3$) is zero, therefore it is possible to set that diagonal to zero via chipsplitting moves.}

\medskip
\begin{center}
\begin{BVerbatim}
 1            ·             ·            ·
 · ·          1 -1          1 ·          1 ·
 · 3 ·        ·  3 ·        · 2 1        · 2 ·
-1 · · 1     -1  · · 1     -1 · · 1     -1 · 1 ·
\end{BVerbatim}
\end{center}
\end{example}

\begin{prop}\label{prop:retraction}
Let $w=(w_{i,j})_{(i,j)\in V_d}$ be a degree-$e$ chip configuration. Then $w$ is an outcome if and only if $\sum_{k=0}^e (-1)^kw_{k,e-k}=0$ and the retraction of $w$ is an outcome.
\end{prop}
\begin{proof}
If $\sum_{k=0}^e (-1)^kw_{k,e-k}=0$, then $w$ and its retraction are obtained from each other using finite sequences of moves. So it suffices to prove that $\sum_{k=0}^e (-1)^kw_{k,e-k}=0$ holds when $w$ is an outcome. Assume that $w$ is the outcome of an \irredundant{} chipsplitting game $f$. Then $e-1$ is the maximal degree of a \rev{vertex} in $V_{d-1}$ at which a move in $f$ occured. As moves at $(i,j)$ preserve the value of $\sum_{k=0}^e (-1)^kw_{k,e-k}$ for all $(i,j)\in V_{d-1}$ with $\deg(i,j)\leq e-1$, we see that $\sum_{k=0}^e (-1)^kw_{k,e-k}=0$.
\end{proof}

\begin{prop}\label{prop:if-pascal-vanish-then-outcome}
Let $w\in\ZZ^{V_d}$ be a chip configuration and suppose that all Pascal \rev{equations} on $\ZZ^{V_d}$ vanish at $w$. Then $w$ is an outcome.
\end{prop}
\begin{proof}
By Proposition~\ref{prop:basis_via_1st_row/1st_column}, for every integer $0\leq e\leq d$ there exists a Pascal equation
\[
\phi^{(e)}:=\sum_{(i,j)\in V_d}c^{(e)}_{i,j}x_{i,j}
\]
with $c^{(e)}_{0,j}=0$ for $j<e$ and $c^{(e)}_{0,e}=1$. Note that $c^{(e)}_{i,j}=0$ for all $(i,j)\in V_d$ with $\deg(i,j)<e$ and $c^{(e)}_{k,e-k}=(-1)^k$ for $k\in\{0,\ldots,e\}$. Next, note that for $e=\deg(w)$ we have
\[
\sum_{k=0}^e(-1)^kw_{i,j}=\phi^{(e)}(w)=0
\]
and hence $w$ has a retraction $w'$, at which all Pascal equations also vanish. Repeating the same argument, we see that $w'$ also has a retraction $w''$, at which all Pascal equations again vanish. After repeating this $e+1$ times, we arrive at a chip configuration of degree $<0$, which must be the initial configuration. Hence by Proposition~\ref{prop:retraction}, we see that $w$ is an outcome.
\end{proof}

\begin{example}\rev{
Let $w_{\indep}$ be the chip configuration associated to the model $\mathcal M_{\indep}$ from Section~\ref{sec:intro}. A general Pascal equation evaluated at $w_{\indep}$ gives}
\[
\rev{c_{20}+2c_{11}+c_{02} - c_{00} = c_{10} + c_{01} - c_{00} = 0,}
\]
\rev{therefore $w_{\indep}$ is an outcome. The first equality above corresponds to passing to the retraction of $w_{\indep}$.}
\end{example}

A chip configuration $w\in \mathbb Z^{V_d}$ is an outcome if and only if all Pascal equations vanish at $w$. In particular using a larger or smaller \rev{$V_{d'}$} for the same $w$, provided $d'\geq \deg(w)$, does not change the fact that $w$ is a chipsplitting outcome. Later in this section we see however that fixing a finite $d$ is useful as it provides an additional basis to the space of Pascal equations.

\begin{de}
Let $0\leq k\leq d$ be an integer.
\begin{enumerate}
\item We write $\psi_k$ for the unique Pascal equation $\sum_{(i,j)\in V_d}c_{i,j}x_{i,j}$ such that $c_{0,j}=\delta_{jk}$ 
\item We write $\overline{\psi}_k:=(12)\cdot\psi_k$.
\end{enumerate}
\end{de}

\begin{prop}\label{psi-equations}~
\begin{enumerate}
\item We have
\[
\psi_k=(-1)^k\sum_{(i,j)\in V_d}(-1)^j\binom{i}{k-j}x_{i,j}\quad\text{and}\quad\overline{\psi}_k=(-1)^k\sum_{(i,j)\in V_d}(-1)^i\binom{j}{k-i}x_{i,j}
\]
for all integers $0\leq k\leq d$.
\itembreak
\item Every Pascal equation can be written uniquely as 
\[
\sum_{k=0}^da_k\psi_k\quad \rev{(a_k\in\mathbb Z)} \quad\text{\rev{as well as}}\quad \sum_{k=0}^d b_k\overline{\psi}_k\quad \rev{(b_k\in\mathbb Z)}.
\]
When $d<\infty$, the $\psi_k$ and $\overline{\psi}_k$ form two bases of the space of Pascal equations.
\end{enumerate}

\end{prop}
\begin{proof}
(a)
We have $(-1)^{k+j}\binom{0}{k-j}=\delta_{jk}$ and so it suffices to prove that 
\[
\sum_{(i,j)\in V_d}(-1)^j\binom{i}{k-j}x_{i,j}
\]
is in fact a Pascal equation. Indeed, we have
\[
(-1)^j\binom{i}{k-j}=(-1)^j\binom{i+1}{k-j}+(-1)^{j+1}\binom{i}{k-(j+1)}
\]
for all $(i,j)\in V_d$ as $\binom{a+1}{b+1}=\binom{a}{b+1}+\binom{a}{b}$ for all integers $a,b$.

\itembreak
(b)
Write
\[
\sum_{(i,j)\in V_d}c_{i,j}x_{i,j}=\sum_{k=0}^da_k\psi_k=\sum_{k=0}^d b_k\overline{\psi}_k.
\]
Then we see that
\[
c_{i,j}=\sum_{k=j}^{\deg(i,j)}a_k(-1)^{k+j}\binom{i}{k-j}=\sum_{k=i}^{\deg(i,j)}b_k(-1)^{k+i}\binom{j}{k-i}
\]
for all $(i,j)\in V_d$. We see that each $c_{i,j}$ is a finite sum. We also see that $c_{0,j}=a_j$ and $c_{i,0}=b_i$ for all $i,j\in\NN_{\leq d}$. So now the statement follows from Proposition~\ref{prop:basis_via_1st_row/1st_column}.
\end{proof}

\begin{ex}
For $d=7$ and $k=3$, the Pascal equation $\psi_k$ can be visualised by writing the coefficients $c_{i,j}$ on the grid $V_{d}$ as follows:
\medskip
\begin{center}
\begin{BVerbatim}
·
·  ·
·  ·  ·
·  ·  ·  ·
1  1  1  1  1
· -1 -2 -3 -4  -5
·  ·  1  3  6  10  15
·  ·  · -1 -4 -10 -20 -35
\end{BVerbatim}
\end{center}
\medskip
We note \rev{that the resulting picture is a rotated Pascal's triangle with minus signs on even rows.}
%that the signs of the rows alternate and that $|c_{i,j}|=|c_{i-1,j}|+|c_{i-1,j+1}|$ for all $i>0$.
\end{ex}

\subsection{Additional structure for \texorpdfstring{$d<\infty$}{d<oo}}\label{additional-structure}
In this subsection, we consider a $V_d$ with $d<\infty$. \rev{Since chipsplitting outcomes are characterized by vanishing at all Pascal equations (Propositions~\ref{prop:pascal-equations-vanish-at-outcomes} and~\ref{prop:if-pascal-vanish-then-outcome}), it will be useful to have multiple parametrizations of the space of all Pascal equations.} By Proposition~\ref{psi-equations}, we know that the $\psi_k$ and $\overline{\psi}_k$ form two bases of the space of Pascal equations on $\ZZ^{V_d}$. When $d<\infty$, we also have another natural basis \rev{which will be of further help in Sections~\ref{sec:supp+<=3}--\ref{sec:supp+==5}}.

\mybreak

\rev{To introduce the new basis, we first prove that Pascal equations are characterized by their coefficients on the $d$-th diagonal.}

\begin{prop}
For every vector $(a_0,\ldots,a_d)\in\ZZ^{d+1}$, there exists a unique Pascal equation 
\[
\sum_{(i,j)\in V_d}c_{i,j}x_{i,j}
\]
such that $c_{k,d-k}=a_k$ for all integers $0\leq k\leq d$.
\end{prop}
\begin{proof}
Let $(a_0,\ldots,a_d)\in\ZZ^{d+1}$, set $c_{k,d-k}:=a_k$ for $k\in\{0,\ldots,d\}$ and, for $e=d-1,\ldots,0$, set $c_{k,e-k}=c_{k+1,e-k}+c_{k,e-k+1}$ for $k\in\{0,\ldots,e\}$ recursively. Then 
\[
\sum_{(i,j)\in V_d}c_{i,j}x_{i,j}
\]
is a Pascal equation such that $c_{k,d-k}=a_k$ for all integers $0\leq k\leq d$. Clearly, this Pascal equation is unique with this property.
\end{proof}

\begin{de}
Let $(a,b)\in V_d$ with $\deg(a,b)=d$. We write $\varphi_{a,b}$ for the unique Pascal equation $\sum_{(i,j)\in V_d}c_{i,j}x_{i,j}$ such that $c_{i,j}=\delta_{ia}$ (or equivalently $c_{i,j}=\delta_{jb}$) for all $(i,j)\in V_d$ with $\deg(i,j)=d$.
\end{de}

\begin{samepage}
\begin{prop}\label{phi-equations}~
\begin{enumerate}
\item We have
\[
\varphi_{a,b}=\sum_{(i,j)\in V_d} \binom{d-(i+j)}{a-i}x_{i,j}=\sum_{(i,j)\in V_d} \binom{d-(i+j)}{b-j}x_{i,j}
\]
for all $(a,b)\in V_d$ with $\deg(a,b)=d$.
\itembreak
\item The $\varphi_{a,b}$ form a basis for the space of all Pascal equations. 
\end{enumerate}
\end{prop} 
\end{samepage}
\begin{proof}
(a)
We have
\[
\delta_{ia}=\binom{d-(i+j)}{a-i}=\binom{d-(i+j)}{b-j}=\delta_{jb}
\]
for all $(i,j)\in V_d$ with $\deg(i,j)=d$. So it suffices to show that 
\[
\sum_{(i,j)\in V_d} \binom{d-(i+j)}{a-i}x_{i,j}
\]
is a Pascal equation. Indeed, we have
\[
\binom{d-(i+j)}{a-i}=\binom{d-(i+1+j)}{a-(i+1)}+\binom{d-(i+j+1)}{a-i}
\]
for all $(i,j)\in V_{d-1}$ as $\binom{a+1}{b+1}=\binom{a}{b+1}+\binom{a}{b}$ for all integers $a,b$.

\itembreak
(b)
Every Pascal equation can be uniquely written as
\[
\sum_{(i,j)\in V_d}c_{i,j}x_{i,j}=\sum_{(a,b)\in V_d\setminus V_{d-1}}c_{a,b}\varphi_{a,b}.
\]
So we see that the $\varphi_{a,b}$ form a basis for the space of all Pascal equations. 
\end{proof}

\begin{ex}
For $d=7$ and $(a,b)=(3,4)$, the Pascal equation $\varphi_{a,b}$ can be visualised by writing the coefficients $c_{i,j}$ on the grid $V_{d}$ as follows:
\medskip
\begin{center}
\begin{BVerbatim}
 ·
 ·  ·
 ·  ·  ·
 1  1  1  1
 4  3  2  1  ·
10  6  3  1  ·  ·
20 10  4  1  ·  ·  ·
35 15  5  1  ·  ·  ·  ·
\end{BVerbatim}
\end{center}
\medskip
We note that the coefficients form a Pascal triangle.
\end{ex}

\moved Next, we define an action of $S_3$ on $V_d$. For $(i,j)\in V_d$ we set
\begin{align*}
(12)\cdot(i,j)&:=(j,i),&(\rev{132})\cdot(i,j)&:=(d-\deg(i,j),i),\\
(13)\cdot(i,j)&:=(d-\deg(i,j),j),&(\rev{123})\cdot(i,j)&:=(j,d-\deg(i,j)),\\
(23)\cdot(i,j)&:=(i,d-\deg(i,j)).
\end{align*}

\rev{We use this action to define an action of $S_3$ on $\mathbb Z^{V_d}$ by setting}
%Next we define an action of $S_3=\langle(12),(123)\rangle$ on $\ZZ^{V_d}$. We set
\begin{eqnarray*}
(12)\cdot(w_{i,j})_{(i,j)\in V_d}&:=&(w_{j,i})_{(i,j)\in V_d}\\
(123)\cdot(w_{i,j})_{(i,j)\in V_d}&:=&((-1)^{d-j}w_{j,d-\deg(i,j)})_{(i,j)\in V_d}
\end{eqnarray*}
for all $w=(w_{i,j})_{(i,j)\in V_d}\in\ZZ^{V_d}$. It is a routine computation to verify that this determines a well-defined action of $S^3$. Under this action, we have
\begin{eqnarray*}
(13)\cdot (w_{i,j})_{(i,j)\in V_d}&=&((-1)^{d-j}w_{d-\deg(i,j),j})_{(i,j)\in V_d}\\
(23)\cdot (w_{i,j})_{(i,j)\in V_d}&=&((-1)^{d-i}w_{i,d-\deg(i,j)})_{(i,j)\in V_d},\\
(132)\cdot (w_{i,j})_{(i,j)\in V_d}&=&((-1)^{d-i}w_{d-\deg(i,j),i})_{(i,j)\in V_d}.
\end{eqnarray*}

\moved We have $\sigma\cdot \supp(w)=\supp(\sigma\cdot w)$ for all $w\in\ZZ^{V_d}$ and $\sigma\in S_3$.

\mybreak

The way $(12)$, $(13)$ and $(23)$ act \rev{on $\mathbb Z^{V_d}$} is vizualized below. The permutation $(12)$ switches the order of all entries of the same degree. The permutation $(13)$ switches the order of all entries of the same row and changes the signs of alternating rows. The permutation $(23)$ acts similarly on columns. 

\begin{center}
\raisebox{3.5pt}
%\raisebox{5pt}
{\begin{tikzpicture}[scale=0.5]%[scale=0.4]
\filldraw (0.1,0.1) circle (3pt);
\draw [stealth-stealth](1,0) -- (0,1);
\draw [stealth-stealth](2,0) -- (0,2);
\draw [stealth-stealth](3,0) -- (0,3);
\draw [stealth-stealth](4,0) -- (0,4);
\draw [stealth-stealth](5,0) -- (0,5);
\draw [stealth-stealth](6,0) -- (0,6);
\draw [stealth-stealth](7,0) -- (0,7);
%\draw [stealth-stealth](8,0) -- (0,8);
%\draw [stealth-stealth](9,0) -- (0,9);
%\draw [stealth-stealth](10,0) -- (0,10);
\end{tikzpicture}}\quad\quad\quad
\begin{tikzpicture}[scale=0.5]%[scale=0.4]
%\node at (-.5,0) {\tiny+};
%\node at (-.5,1) {\tiny-};
%\node at (-.5,2) {\tiny+};
\node at (-.5,3) {\tiny $-$};
\node at (-.5,4) {\tiny +};
\node at (-.5,5) {\tiny $-$};
\node at (-.5,6) {\tiny +};
\node at (-.5,7) {\tiny $-$};
\node at (-.5,8) {\tiny +};
\node at (-.5,9) {\tiny $-$};
\node at (-.5,10) {\tiny +};
\filldraw (0.1,10) circle (3pt);
%\draw [stealth-stealth](10,0) -- (0,0);
%\draw [stealth-stealth](9,1) -- (0,1);
%\draw [stealth-stealth](8,2) -- (0,2);
\draw [stealth-stealth](7,3) -- (0,3);
\draw [stealth-stealth](6,4) -- (0,4);
\draw [stealth-stealth](5,5) -- (0,5);
\draw [stealth-stealth](4,6) -- (0,6);
\draw [stealth-stealth](3,7) -- (0,7);
\draw [stealth-stealth](2,8) -- (0,8);
\draw [stealth-stealth](1,9) -- (0,9);
\end{tikzpicture}\quad\quad\quad
\raisebox{3.5pt}
{\begin{tikzpicture}[scale=0.5]%[scale=0.4]
%\node at (0,10.8) {\tiny+};
%\node at (1,9.8) {\tiny-};
%\node at (2,8.8) {\tiny+};
\node at (3,7.8) {\tiny $-$};
\node at (4,6.8) {\tiny +};
\node at (5,5.8) {\tiny $-$};
\node at (6,4.8) {\tiny +};
\node at (7,3.8) {\tiny $-$};
\node at (8,2.8) {\tiny +};
\node at (9,1.8) {\tiny $-$};
\node at (10,.8) {\tiny +};
\filldraw (10,.1) circle (3pt);
%\draw [stealth-stealth](0,10) -- (0,0);
%\draw [stealth-stealth](1,9) -- (1,0);
%\draw [stealth-stealth](2,8) -- (2,0);
\draw [stealth-stealth](3,7) -- (3,0);
\draw [stealth-stealth](4,6) -- (4,0);
\draw [stealth-stealth](5,5) -- (5,0);
\draw [stealth-stealth](6,4) -- (6,0);
\draw [stealth-stealth](7,3) -- (7,0);
\draw [stealth-stealth](8,2) -- (8,0);
\draw [stealth-stealth](9,1) -- (9,0);
\end{tikzpicture}}
\end{center}

\begin{prop}
The space of outcomes is closed under the action of $S_3$.
\end{prop}
\begin{proof}

Let $w=(w_{i,j})_{(i,j)\in V_d}$ be an outcome. We already know that $(12)\cdot w$ is again an outcome. So it suffices to prove that $(123)\cdot w$ is an outcome as well. This is indeed the case since

\begin{eqnarray*}
\psi_k((123)\cdot w)&=&(-1)^k\sum_{(i,j)\in V_d}(-1)^j\binom{i}{k-j}(-1)^{d-j}w_{j,d-\deg(i,j)}\\
&=&(-1)^{d-k}\sum_{(i',j')\in V_d}\binom{d-(i'+j')}{k-i'}w_{i',j'}\\
&=&(-1)^{d-k}\varphi_{k,d-k}(w)=0
\end{eqnarray*}

for all integers $0\leq k\leq d$.
\end{proof}

\begin{example}[\rrev{Example~\ref{expl:indep} continued}]
\rev{Let $w_{\indep}$ be the chip configuration associated to $\mathcal M_{\indep}$ as in \rrev{Example~\ref{expl:indep2}.} The orbit of $(123)\in S_3$ acting on $\supp(w_{\indep})$ is the following sequence of supports.}

\medskip
\begin{center}
\begin{BVerbatim}
*            *            *
· *          · ·          * ·
* · *        * * *        * · *
\end{BVerbatim}
\end{center}
\end{example}

\subsection{Valid outcomes}

In this paper, we are mostly interested in valid outcomes, since they correspond to reduced \redratmodels{} as explained in Section~\ref{sec:model2chipsplitting}. 

\begin{lm}\label{lm:supp-=empty}
Let $w=(w_{i,j})_{(i,j)\in V_d}\in\ZZ^{V_d}$ be an outcome and suppose that $\supp^-(w)=\emptyset$. Then $w$ is the initial configuration.
\end{lm}
\begin{proof}
We may assume that $d<\infty$. We have $w_{i,j}\geq0$ for all $(i,j)\in V_d$. For every $(a,b)\in V_d$ of degree $d$, the equation $\varphi_{a,b}(w)=0$ shows that $w_{i,j}=0$ for all $i\in\{0,\ldots,a\}$ and $j\in\{0,\ldots,b\}$. Combined, this shows that $w_{i,j}=0$ for all $(i,j)\in V_d$.
\end{proof}

\rev{In particular, a valid outcome $w$ with $w_{0,0}=0$ is the initial configuration.}

\begin{prop}
Let $w=(w_{i,j})_{(i,j)\in V_d}\in\ZZ^{V_d}$ be an outcome and suppose that $\#\supp^-(w)=1$. Write $c_0=\min\{i\mid (i,j)\in V_d\mid w_{i,j}\neq0\}$, $r_0=\min\{j\mid (i,j)\in V_d\mid w_{i,j}\neq0\}$ and $d'=d-c_0-r_0$. Then
\[
(w_{c_0+i,r_0+j})_{(i,j)\in V_{d'}}\in\ZZ^{V_{d'}}
\]
is a valid outcome. In particular, if $c_0=r_0=0$, then $w$ is a valid outcome. 
\end{prop}
\begin{proof}
We may assume that $d<\infty$. First we suppose that $c_0=r_0=0$. Then the equations $\varphi_{0,d}(w)=0$ and $\varphi_{d,0}(w)=0$ show that $w_{0,j}<0$ and $w_{i,0}<0$ for some $i,j\in\{0,\ldots,d\}$. Since $\#\supp^-(w)=1$, it follows that $i=j=0$ and $\supp^-(w)=\{(0,0)\}$. Hence $w$ is indeed valid.

\mybreak
In general, we note that $\varphi_{c_0+a,r_0+b}$ vanishes on $w$ for all $(a,b)\in V_{d'}\setminus V_{d'-1}$. So $\varphi_{a,b}$ vanishes on $(w_{c_0+i,r_0+j})_{(i,j)\in V_{d'}}$ for all $(a,b)\in V_{d'}\setminus V_{d'-1}$. This means that $(w_{c_0+i,r_0+j})_{(i,j)\in V_{d'}}$ is an outcome to which we can apply the previous case.
\end{proof}

%\begin{prop}
%Let $w=(w_{i,j})_{(i,j)\in V_d}$ be a valid outcome. If $w_{0,0}=0$, then $w$ is the initial configuration.
%\end{prop}
%\begin{proof}
%This follows directly from Lemma~\ref{lm:supp-=empty}.
%\end{proof}

\section{From reduced \redratmodels{} to valid outcomes and back}\label{sec:model2chipsplitting}

In this section, we continue to use the notion of a chipsplitting game and related concepts (Definitions~\ref{de:chipsplitting} and~\ref{chipsplitting-notions}). We augment this notion by allowing chip configurations to have rational or real entries (see Remark~\ref{a-valued}). We start by establishing a further characterization of the space of outcomes.

\begin{lm}\label{lm:outcomes=kernel}
The space of integral (resp.\ rational, real) outcomes equals the kernel of the linear map
\begin{eqnarray*}
\alpha_d\colon R^{V_d}&\to& R[t]_{\leq d}\\
(w_{i,j})_{(i,j)\in V_d}&\mapsto&\sum_{(i,j)\in V_d}w_{i,j}t^i(1-t)^j
\end{eqnarray*}
where $R=\ZZ$ (resp. $R=\QQ,\RR$). 
\end{lm}
\begin{proof}
%For $d=\infty$,
The map $\alpha_{\infty}$
%\begin{eqnarray*}
%\alpha_d\colon R^{V_d}&\to&R[t]_{\leq d}\\
%(w_{i,j})_{(i,j)\in V_d}&\mapsto&\sum_{(i,j)\in V_d}w_{i,j}t^i(1-t)^j
%\end{eqnarray*}
is the direct limit of the maps $\alpha_e$ for $e<\infty$. So we may assume that $d<\infty$. In this case, we know that the space of outcomes has codimenion $d+1$ by Proposition~\ref{psi-equations}(b). For a given polynomial $p = \sum_{j=0}^d c_j t_j\in R[t]_{\leq d}$, set $w_{i,j}=c_i$ when $j=0$ and $w_{i,j}=0$ otherwise. Then $\alpha_d(w_{i,j})_{(i,j)\in V_d}=p$. So we see that $\alpha_d$ is surjective. Hence the kernel of $\alpha_d$ has the same codimension as the space of outcomes. It now suffices to show that every outcome is contained in the kernel of $\alpha_d$. Note that the initial configuration is contained in the kernel of $\alpha_d$. And, for $w\in R^{V_d}$, the value of $\alpha_d(w)$ does not change when we execute a chipsplitting move at $(i,j)\in V_{d-1}$. Indeed, we have
\[
-t^i(1-t)^j+t^{i+1}(1-t)^j+t^i(1-t)^{j+1}=t^i(1-t)^j(-1+t+(1-t))=0
\]
and so every outcome is contained in the kernel of $\alpha_d$.
\end{proof}

Let $\mathcal{M}=(w_\nu,i_\nu,j_\nu)_{\nu=0}^n$ be a reduced \redratmodel{}. Then \rev{$\mathcal{M}$} induces a real chip configuration $w(\mathcal{M})=(w_{i,j})_{(i,j)\in V_\infty}$ by setting
\[
w_{i,j}:=\left\{\begin{array}{cl}-1&\mbox{if $(i,j)=(0,0)$,}\\w_\nu&\mbox{if $(i,j)=(i_\nu,j_\nu)$ \rev{for some $\nu\in\{0,\dotsc,n\}$},}\\0&\mbox{otherwise}\end{array}\right.
\]
We have the following result.

\begin{prop}\label{prop:reduced_models2valid_outcomes}\phantom{text}
\begin{enumerate}
\item The map 
$
\mathcal{M}\mapsto w(\mathcal{M})
$
is a bijection between the set of reduced \redratmodels{} and the set of valid real outcomes $w\in\RR^{V_{\infty}}$ with $w_{0,0}=-1$.
\item Let $S$ be the support of $\mathcal{M}$. Then $\supp^+(w(\mathcal{M}))=S$.
\item The map $\mathcal{M}\mapsto w(\mathcal{M})$ is degree-preserving.
\item The chip configuration $w(\mathcal{M})$ is rational if and only if the coefficients of $\mathcal{M}$ are all rational.
\item Every valid rational outcome $w\in\QQ^{V_{\infty}}$ is of the form $\lambda\hat{w}$ for some $\lambda\in\QQ_{>0}$ and valid integral outcome $\hat{w}\in\ZZ^{V_{\infty}}$.
\item Let $w\in\RR^{V_{\infty}}$ be a valid real outcome with $w_{0,0}=0$. Then $w=0$.
\end{enumerate}
\end{prop}
\begin{proof}
(a)
From Lemma~\ref{lm:outcomes=kernel}, it follows that $w(\mathcal{M})$ is indeed a valid real outcome with value $-1$ at $(0,0)$. Clearly, the map $\mathcal{M}\mapsto w(\mathcal{M})$ is injective. Let $w\in\RR^{V_{\infty}}$ be a valid real outcome with $w_{0,0}=-1$ and write $\supp^+(w)=\{(i_0,j_0),\ldots,(i_n,j_n)\}$ and take $w_\nu:=w_{i_\nu,j_\nu}$ for $\nu=0,\ldots,n$. Then $(w_\nu,i_\nu,j_\nu)_{\nu=0}^n$ is a reduced \redratmodel{} by Lemma~\ref{lm:outcomes=kernel}. Hence the map $\mathcal{M}\mapsto w(\mathcal{M})$ is also surjective.

%\itembreak
%(b)
%This holds by the definition of $w(\mathcal{M})$.

%\itembreak
%(c)
%This holds since $\supp^-(w(\mathcal{M}))=\{(0,0)\}\cup\supp^+(w)=\{(0,0)\}\cup S$ for all reduced models $\mathcal{M}$.

%\itembreak
%(d)
%This holds by the definition of $w(\mathcal{M})$.

\itembreak
(b)--(d) hold by definition.

\itembreak
(e)
For every valid rational outcome $w\in\QQ^{V_{\infty}}$ there exist an $n\in\NN$ such that $nw_{i,j}\in\ZZ$ for all $(i,j)$ in the finite set $\supp(w)$. Take $\hat{w}:=nw$ and $\lambda:=1/n\in\QQ_{>0}$. Then $\hat{w}\in\ZZ^{V_{\infty}}$ is an valid integral outcome using Lemma~\ref{lm:outcomes=kernel} and $w=\lambda\hat{w}$.

\itembreak
(f)
Since $w$ is an outcome with $\supp^-(w)=\emptyset$, we know by Lemma~\ref{lm:outcomes=kernel} that
$
\sum_{(i,j)}w_{i,j}t^i(1-t)^j
%=\sum_{(i,j)\in V_\infty}w_{i,j}t^i(1-t)^j
=0
$
for $(i,j)$ ranging over $\supp^+(w)$
and, by evaluating at $t=1/2$, we see that $\supp^+(w)$ can only be the empty set. Hence $w=0$.
\end{proof}

\begin{example}[\rrev{Example~\ref{expl:indep} continued}]
\rev{One can verify that $w(\mathcal M_{\indep}) = w_{\indep}$. }
\end{example}

\begin{proposition}\label{are-equivalent}
%Conjectures~\ref{conj:main_models} and~\ref{conj:main_outcomes} are equivalent.
Theorems~\ref{thm:main_models} and~\ref{thm:main_outcomes} are equivalent.
\end{proposition}

\begin{proof}
%By Remark~\ref{re:fundamental_models_suffice}, we know that for Conjecture~\ref{conj:main_models} it suffices to only consider fundamental models. By Remark~\ref{re:fundamental_coeff_rat}, we know that the coefficients of a fundamental model are rational. Hence for Conjecture~\ref{conj:main_models} it also suffices to consider all reduced models with rational coefficients.
By Remark~\ref{re:fundamental_models_suffice}, we know that for Theorem~\ref{thm:main_models} it suffices to only consider fundamental models. Since the constraint $\sum_\nu p_\nu = 1$ of Definition~\ref{def:fundamental} has coefficients in the rational numbers, the coefficients of a fundamental model are rational. Hence it suffices to only consider rational coefficients.

\mybreak
By Proposition~\ref{prop:reduced_models2valid_outcomes} (e), every valid rational outcome is a positive multiple of a valid integral outcome. The space of outcomes is closed under scaling, and scaling does not change the degree or size of the positive support of a chip configuration.
%Hence for Conjecture~\ref{conj:main_outcomes} we may also consider all valid rational outcomes $w$ with $w_{0,0}=-1$.
Hence for Theorem~\ref{thm:main_outcomes} it suffices to consider all valid rational outcomes $w$ with $w_{0,0}=-1$.

\mybreak
%By Proposition~\ref{prop:reduced_models2valid_outcomes} we know that the map $\mathcal{M}\mapsto w(\mathcal{M})$ is a degree-preserving bijection between the set of reduced models in $\Delta_n$ with rational coefficients and the set of valid rational outcomes $w$ with $w_{0,0}=-1$ and $\#\supp^+(w)=n+1$. This shows the %three
%required equivalences.
The required equivalence is now given by Proposition~\ref{prop:reduced_models2valid_outcomes} (a)--(d).
\end{proof}

Next, we consider the chipsplitting equivalent of fundamental models.

\begin{de}\label{def:fundamental_outcome}
A valid outcome $w\in\ZZ^{V_d}\setminus\{0\}$ is called {\em fundamental} if it cannot be written as 
\[
w=\mu_1 w_1+\mu_2w_2,
\]
where $\mu_1,\mu_2\in\QQ_{>0}$ and $w_1,w_2\in\ZZ^{V_d}$ are valid outcomes with $\supp^+(w_1),\supp^+(w_2)\subsetneq\supp^+(w)$.
\end{de}

Applying Proposition~\ref{fundamental-to-reduced} and keeping track of rational coefficients, we conclude the following.

\begin{prop}\label{prop:fundamental=fundamental}\ 
\begin{enumerate}
\item Let $\mathcal{M}$ be a \redratmodel{} with rational coefficients and let $n\in\NN$ be any integer such that $w=nw(\mathcal{M})$ is an integral chip configuration. Then $\mathcal{M}$ is a fundamental model if and only if $w$ is a fundamental outcome.
\item \moved In particular, fundamental models correspond one-to-one with fundamental integral outcomes $w$ with $\gcd\{w_{i,j}\mid (i,j)\in\supp(w)\}=1$.\qed
\end{enumerate}
\end{prop}

\begin{example}[\rrev{Example~\ref{expl:indep} continued}] \rev{The valid outcome
$w_{\indep}$ is fundamental because $\mathcal M_{\indep}$ is a fundamental model.}
\end{example}

\rev{We close this section with a general observation about fundamental outcomes.}

\begin{prop}\label{fundamental-n-leq-d}\moved
Let $w$ be a degree-$d$ fundamental outcome with $\#\supp^+(w)=n+1$. Then $n\leq d$.
\end{prop}
\begin{proof}
\rev{Recall} that if $\mathcal M$ is a fundamental model with support $S\subseteq \mathbb Z^2\setminus \{(0,0)\}$, then $\mathcal M$ is the only \ratmodel{} with support $S$. In terms of outcomes, this means that there exists a valid outcome $w'$ with $\supp^+(w')\subseteq S$ and that the space of outcomes whose support is contained in $S\cup\{(0,0)\}$ is spanned by $w'$. In particular, this space must be $1$-dimensional. When $n>d$, the space of chip configurations $w'$ with $\supp(w')\subseteq S\cup\{(0,0)\}$ has dimension $>d+1$. The subspace of outcomes has codimension $\leq d+1$ and hence has dimension $\geq 2$ in this case. So $n\leq d$. 
\end{proof}

%Thus we see that fundamental models correspond one-to-one with fundamental integral outcomes $w$ with $\gcd\{w_{i,j}\mid (i,j)\in\supp(w)\}=1$. %This shows that the statements from Remark~\ref{re:conj_fundamental_models}  and Remark~\ref{re:conj_fundamental_outcomes} are equivalent and so Theorem~\ref{thm:computational_models} is equivalent to Theorem~\ref{thm:computational_outcomes}.

\section{Valid outcomes of positive support \texorpdfstring{$\leq 3$}{<=3}}\label{sec:supp+<=3}

From now on, we will always assume that $d<\infty$. Since every chip configuration has finite degree, this assumption is harmless.
In this section, we prove Theorem~\ref{thm:main_outcomes} for valid outcomes whose positive support has size $\leq 3$. To do this, we introduce our first tool, the Invertibility Criterion, which shows that certain subsets of $V_d$ cannot contain the support of an outcome. 

\subsection{The Invertibility Criterion}
Let $S\subseteq V_d$ and $E\subseteq \{0,\ldots,d\}$ be nonempty subsets of the same size $\leq d+1$. \rev{The Invertibility Criterion (Proposition~\ref{invertibility-criterion}) will help us detect, with the right choice of $E$, whether $S$ can be the support, or contain the support, of some outcome.}

\begin{de}
We define
\[
A^{(d)}_{E,S}:=\left(\binom{d-\deg(i,j)}{a-i}\right)_{a\in E,(i,j)\in S}
\]
to be the {\em pairing matrix} of $(E,S)$.
\end{de}

Let $w=(w_{i,j})_{(i,j)\in V_d}\in\ZZ^{V_d}$ be an outcome such that $\supp(w)\subseteq S$. 

\begin{prop}[Invertibility Criterion]\label{invertibility-criterion}
If $A^{(d)}_{E,S}$ is invertible, then $w$ is the initial configuration.
\end{prop}
\begin{proof}
Suppose that $\supp(w)\neq \emptyset$. Then
\[
(w_{i,j})_{(i,j)\in S}\neq0, \quad A^{(d)}_{E,S}\cdot (w_{i,j})_{(i,j)\in S}=(\varphi_{a,d-a}(w))_{a\in E}=0
\]
and hence $A^{(d)}_{E,S}$ is degenerate.
\end{proof}

Our goal is to construct, for many subsets $S\subseteq V_d$, a subset $E$ such that $A^{(d)}_{E,S}$ is invertible. We do this by dividing the pairing matrix into small parts and dealing with these parts seperately.

\subsection{\rev{Dividing the pairing matrix into smaller parts.}}\label{divide}
Let $\lambda=(\lambda_1,\ldots,\lambda_\ell)\in\NN^\ell$ be a tuple of integers adding up to $d+1$. Write $c_i=\lambda_1+\ldots+\lambda_i$ for $i\in\{0,\ldots,\ell\}$. For $k\in\{1,\ldots,\ell\}$, let $S_k:=\{(i,j)\in S\mid c_{k-1}\leq i<c_k\}$. Assume that the condition
\[
\#S_k\in \{0,\lambda_k\}
\]
is satisfied for every $k\in\{1,\ldots,\ell\}$. Lastly, set
\[
E_k:=\left\{\begin{array}{cl}
\{c_{k-1},c_{k-1}+1,\ldots, c_k-1\}&\mbox{if $\#S_k=\lambda_k$,}\\
\emptyset&\mbox{if $S_k=\emptyset$,}
\end{array}\right.
\]
where the top row indicates consecutive integers ranging from $c_{k-1}$ to $c_{k}-1$.

\begin{re}\label{re:construction_of_lambda}
\rev{Not all subsets $S$ will admit a tuple $\lambda$ as above such that $\#S_k\in\{0,\lambda_k\}$ for all $k$. For instance, let $S$ be the set of positions marked with an} \verb|*| \rev{in the following picture.}
\begin{center}
\begin{BVerbatim}

*
· *
* * ·
\end{BVerbatim}
\end{center}
\rev{Since $d=2$, such a $\lambda$ would have $\lambda_1\in\{1,2,3\}$. But if $\lambda_1 = 1$ then $\#S_1 = 2 \neq \lambda_1$, and if $\lambda_1\in\{2,3\}$ then $\#S_1 = 4\neq \lambda_1.$ So this $S$ does not admit a $\lambda$ with this property. For other $S$,}
one can try to define \rev {such }a $\lambda$ %with this property
recursively by, for $k=1,2,\ldots$, picking $\lambda_k$ minimal such that $\#S_k\in\{0,\lambda_k\}$. We stop when $c_k=d+1$. This will work exactly when
\[
\#\{(i,j)\in S\mid i\geq d-k\}\leq k+1
\]
for all $k\in\{0,1,\ldots,d\}$. 
\end{re}

\begin{prop}\label{prop:divide}
Take $E=E_1\cup\cdots\cup E_\ell$. Then $\#E=\#S$ and we have
\[
A^{(d)}_{E,S}=\begin{pmatrix}
A^{(d)}_{E_1,S_1}&0&\cdots&0\\
\vdots&\ddots&\ddots&\vdots\\
\vdots&&\ddots&0\\
A^{(d)}_{E_\ell,S_1}&\cdots&\cdots&A^{(d)}_{E_\ell,S_{\ell}}
\end{pmatrix}.
\]
In particular, the matrix $A^{(d)}_{E,S}$ in invertible if and only if all of $A^{(d)}_{E_1,S_1},\ldots,A^{(d)}_{E_\ell,S_\ell}$ are.
\end{prop}
\begin{proof}
It is clear that $\#E=\#S$ and 
\[
A^{(d)}_{E,S}=\begin{pmatrix}
A^{(d)}_{E_1,S_1}&\cdots&\cdots&A^{(d)}_{E_1,S_\ell}\\
\vdots&&&\vdots\\
\vdots&&&\vdots\\
A^{(d)}_{E_\ell,S_1}&\cdots&\cdots&A^{(d)}_{E_\ell,S_{\ell}}
\end{pmatrix}.
\]
We need to show that $A^{(d)}_{E_k,S_{k'}}=0$ when $k<k'$. Indeed, when $k<k'$, $a\in E_k$ and $(i,j)\in S_{k'}$, then
\[
\binom{d-\deg(i,j)}{a-i}=0
\]
since $a<c_k\leq c_{k'-1}\leq i$. So $A^{(d)}_{E_k,S_{k'}}=0$ when $k<k'$.
\end{proof}

\begin{ex}
\begin{samepage}
Take $d=6$ and let $S$ be the set of positions marked with an \verb|*| below.

\begin{center}
\begin{BVerbatim}

·
· · 
* · ·
· · · ·
· · · · ·
· · · · * ·
* · * · · * *

\end{BVerbatim}
\end{center}

\end{samepage}
The construction from Remark~\ref{re:construction_of_lambda} yields the tuple $\lambda=(2,1,1,1,1,1)$. We get 
$$
\begin{array}{ll}
S_1=\{(0,0),(0,4)\}, &E_1=\{\rev{0,1}\},\\
S_2=\{(2,0)\}, &E_2=\{\rev{2}\},\\
S_3=\emptyset, &E_3=\emptyset,\\
S_4=\{(4,1)\}, &E_4=\{\rev{4}\},\\
S_5=\{(5,0)\}, &E_5=\{\rev{5}\},\\
S_6=\{(6,0)\}, &E_6=\{\rev{6}\}.
\end{array}\vspace*{-3pt}
$$
So $\lambda$ indeed satisfies the assumption and we see that
\[
A^{(d)}_{E,S}=\begin{pmatrix}
A^{(d)}_{E_1,S_1}&0&0&0&0\\
*&A^{(d)}_{E_2,S_2}&0&0&0\\
*&*&A^{(d)}_{E_4,S_4}&0&0\\
*&*&*&A^{(d)}_{E_5,S_5}&0\\
*&*&*&*&A^{(d)}_{E_6,S_6}
\end{pmatrix}
=\begin{pmatrix}
1&1&0&0&0&0\\
6&2&0&0&0&0\\
*&*&1&0&0&0\\
*&*&*&1&0&0\\
*&*&*&*&1&0\\
*&*&*&*&*&1
\end{pmatrix}
\]
is invertible. Hence $S$ does not contain the support of a nonzero outcome.
\end{ex}

\subsection{\rev{Analyzing the invertibility of the smaller pairing matrices.}}
Using Proposition~\ref{prop:divide} \rev{to divide the pairing matrix $A^{(d)}_{E,S}$ into manageable blocks, we get subsets $S_k\subseteq S$ that are progressively further away from the $y$-axis of the grid as $k$ grows. The next proposition says that we can shift these subsets back toward the origin.}

\begin{prop}\label{prop:shift}
Let $S\subseteq V_d$ and $E\subseteq \{0,\dotsc, d\}$ \rev{and}
\[
\text{\moved} x:=\min(E\cup\{i\mid (i,j)\in S\}).
\]
Let $S'=\{(i-x,j)\mid (i,j)\in S\}$ and $E'=\{a-x\mid a\in E\}$. Then $A^{(d)}_{E,S}=A^{(d-x)}_{E',S'}$.
\end{prop}
\begin{proof}
This follows directly from the definition of the pairing matrix.
\end{proof}

\rev{We now consider pairs $(S,E)$ where $i<\#S$ for all $(i,j)\in S$ and $E = \{0,\dotsc,\#S-1\}$.}

%We now consider the case where $S\subseteq\{(i,j)\in V_d\mid i<s\}$ has $s$ elements and $E=\{0,1,\ldots,s-1\}$.

\begin{prop}\label{prop:conquer}
Suppose that one of the following holds:
\begin{enumerate}
\item We have $S=\{(0,i)\}$ for some $0\leq i\leq d$ and $E=\{0\}$.
\item We have $S=\{(0,i),(0,j)\}$ for some $0\leq i<j\leq d$ and $E=\{0,1\}$.
\item We have $S=\{(0,i),(0,j),(0,k)\}$ for some $0\leq i<j<k\leq d$ and $E=\{0,1,2\}$.
\item We have $S=\{(0,i),(0,j),(1,k)\}$ for some $0\leq i<j\leq d$ and $0\leq k\leq d-1$ such that $i+j\neq 2k+1$ and $E=\{0,1,2\}$. 
\end{enumerate}
Then $A^{(d)}_{E,S}$ is invertible.
\end{prop}
\begin{proof}
We prove the proposition case by case.

(a)
When $S=\{(0,i)\}$ for some $0\leq i\leq d$ and $E=\{0\}$, we see that $A^{(d)}_{E,S}=(1)$ is invertible.

(b)
When $S=\{(0,i),(0,j)\}$ for some $0\leq i<j\leq d$ and $E=\{0,1\}$, we see that
\[
A^{(d)}_{E,S}=\begin{pmatrix}
1&1\\d-i&d-j
\end{pmatrix}\vspace*{-3pt}
\]
is invertible.

(c) 
When $S=\{(0,i),(0,j),(0,k)\}$ for some $0\leq i<j<k\leq d$ and $E=\{0,1,2\}$, we see that
\[
\begin{pmatrix}
1\\&1\\&1&2
\end{pmatrix}A^{(d)}_{E,S}=\begin{pmatrix}
1&1&1\\x&y&z\\x^2&y^2&z^2
\end{pmatrix}
\]
is a Vandermonde matrix, where $(x,y,z)=(d-i,d-j,d-k)$. Hence $A^{(d)}_{E,S}$ is invertible.

(d)
When $S=\{(0,i),(0,j),(1,k)\}$ for some $0\leq i<j\leq d$ and $0\leq k\leq d-1$, we see that
\[
\begin{pmatrix}
1\\&1\\&1&2
\end{pmatrix}A^{(d)}_{E,S}\begin{pmatrix}
1&1\\
&-1\\
&&1
\end{pmatrix}\begin{pmatrix}1\\&(x-y)^{-1}\\&&1\end{pmatrix}
=\begin{pmatrix}
1&0&0\\x&1&1\\x^2&x+y&2z+1
\end{pmatrix},
\]
where $(x,y,z)=(d-i,d-j,d-1-k)$. Assume that $i+j\neq 2k+1$. Then $x+y\neq 2z+1$ and hence $A^{(d)}_{E,S}$ is invertible.
\end{proof}

\subsection{Valid outcomes of positive support \texorpdfstring{$\leq 3$}{<=3}}

We now classify the valid outcomes of positive support $\leq 3$. \rev{Recall that a valid oucome is a chip configuration $w$ with $\supp^-(w)\subseteq\{0,0\}$ such that all Pascal equations vanish at $w$
(Propositions~\ref{prop:pascal-equations-vanish-at-outcomes},~\ref{prop:if-pascal-vanish-then-outcome}).}
We start with the following lemma.

\begin{lm}\label{lm:2on1st}
Let $w$ be a valid outcome of degree $d\geq 1$. Then the following hold:
\begin{enumerate}
\item There are $i,j\in\{1,\ldots,d\}$ such that $(i,0),(0,j)\in\supp^+(w)$.
\item There are distinct $i,j\in\{0,\ldots,d\}$ such that $(i,d-i),(j,d-j)\in\supp^+(w)$.
\end{enumerate}
\end{lm}
\begin{proof} 

We use the Pascal equations $\psi_k$ from Proposition~\ref{psi-equations}. For $d=3$, the coefficients of $\psi_0$, $\overline\psi_0$, and $\psi_d$ (left to right) look as follows:

\begin{center}
\begin{BVerbatim}

·           1           -1      
· ·         1 ·          · 1    
· · ·       1 · ·        · · -1  
1 1 1 1     1 · · ·      · ·  · 1
\end{BVerbatim}
\end{center}
(a)
Since $\deg(w)>0$, we see that $w$ is not the initial configuration. Since $w$ is valid, \rev{by Lemma~\ref{lm:supp-=empty} we therefore have $w_{0,0} < 0$}.
\rev{Since $\psi_0(w) = 0$, there must exist an $i\in\{0,\dotsc,d\}$ such that $w_{i,0}>0$. Likewise, since $\overline{\psi}_0(w)=0$ there exists $j\in\{0,\dotsc,d\}$ such that $w_{0,j}>0$.}

\itembreak
(b) 
Since $\deg(w)=d$, there is an $i\in\{0,\ldots,d\}$ such that $(i,d-i)\in\supp^+(w)$. \rev{Because of the alternating coefficients of $\psi_d$ on the outermost diagonal and since} $\psi_d(w)=0$, we see that there must also be a $j\in\{0,\ldots,d\}\setminus\{i\}$ such that $(j,d-j)\in\supp^+(w)$.
\end{proof}

\begin{prop}
Let $w$ be a valid degree-$d$ outcome and assume that $\#\supp^+(w)\leq 2$. Then
\[
\supp^+(w)=\{(1,0),(0,1)\}.
\]
\end{prop}
\begin{proof}
By the previous lemma, we see that 
\[
\supp(w)=\{(0,0),(0,d),(d,0)\}=:S.
\]
Assume that $d\geq 2$. Then the construction from Remark~\ref{re:construction_of_lambda} yields $\lambda=(2,1,\ldots,1)\in\NN^d$. We get $S_1=\{(0,0),(0,d)\}$, $S_k=\emptyset$ for $k\in\{2,\ldots,d-1\}$ and $S_d=\{(d,0)\}$. Using Propositions~\ref{prop:divide} and~\ref{prop:shift}, we get
\[
A^{(d)}_{\{0,1,d\},S}=\begin{pmatrix}
A^{(d)}_{\{0,1\},S_1}&0\\
*&A^{(d)}_{\{d\},S_d}
\end{pmatrix}=
\begin{pmatrix}
A^{(d)}_{\{0,1\},S_1}&0\\
*&A^{(0)}_{\{0\},\{(0,0)\}}
\end{pmatrix}
\]
and by Proposition~\ref{prop:conquer} the submatrices on the diagonal are both invertible. So $A^{(d)}_{\{0,1,d\},S}$ is invertible. This contradicts the assumption that $\supp(w)=S$ and so $d=1$.
\end{proof}

\begin{lm}
Let $w$ be a valid degree-$d$ outcome and assume that $\#\supp^+(w)=3$. Then one of the following holds:
\begin{enumerate}
\item We have $\supp(w)=\{(0,0),(d,0),(0,d),(i,j)\}$ for some $i,j>0$ with $\deg(i,j)<d$.
\item We have $\supp(\sigma\cdot w)=\{(0,0),(d,0),(0,d),(e,0)\}$ for some $\sigma\in S_3$ and $0<e<d$.
\item We have $\supp(\sigma\cdot w)=\{(0,0),(d,0),(0,e),(d-f,f)\}$ for some $\sigma\in S_2$ and $0<e,f<d$.
\end{enumerate}
\end{lm}
\begin{proof}
When $(d,0),(0,d)\in\supp(w)$, then it is easy to see that (a) or (b) holds. So suppose this is not the case. Since $\#\supp^+(w)=3$, we must have $(d,0)\in\supp(w)$ or $(0,d)\in\supp(w)$ by Lemma~\ref{lm:2on1st}. So there exists an $\sigma\in S_2$ such that $(d,0)\in\supp(\sigma\cdot w)$ and $(0,d)\not\in\supp(\sigma\cdot w)$. Now $\supp(\sigma\cdot w)=\{(0,0),(d,0),(0,e),(d-f,f)\}$ for some $0<e,f<d$ by Lemma~\ref{lm:2on1st}.
\end{proof}

We now apply the the Invertibility Criterion to the possible outcomes in each of these cases.

\begin{prop}
Let $w$ be a degree-$d$ outcome and assume that 
\[
\supp(w)=\{(0,0),(d,0),(0,d),(i,j)\}
\]
for some $i,j>0$ with $\deg(i,j)<d$. Then $d=3$ and $(i,j)=(1,1)$.
\end{prop}
\begin{proof}
Assume that $i>1$. Then the Invertibility Criterion combined with Propositions~\ref{prop:divide},~\ref{prop:shift} and~\ref{prop:conquer} with $\lambda=(2,1,\ldots,1)$ yields a contradiction. Indeed, we would find that
\[
A^{(d)}_{\{0,1,i,d\},S}=\begin{pmatrix}
A^{(d)}_{\{0,1\},S_1}&0&0\\
*&A^{(d)}_{\{i\},S_2}&0\\
*&*&A^{(d)}_{\{d\},S_3}
\end{pmatrix}
\]
is invertible where $S=S_1\cup S_2\cup S_3=\{(0,0),(0,d)\}\cup\{(i,j)\}\cup\{(d,0)\}$. So $i=1$. Applying the same argument to $(12)\cdot w$ shows that $j=1$. Assume that $d>3$. Then we apply the same strategy again with $\lambda=(3,1,\ldots,1)$. We get a contradiction since
\[
A^{(d)}_{\{0,1,2,d\},S}=\begin{pmatrix}
A^{(d)}_{\{0,1,2\},S_1}&0\\
*&A^{(d)}_{\{d\},S_2}
\end{pmatrix}
\]
is invertible, where $S=S_1\cup S_2=\{(0,0),(0,d),(1,1)\}\cup\{(d,0)\}$, by Proposition~\ref{prop:conquer}. So $d=3$.
\end{proof}

\begin{prop}
Let $w$ be a degree-$d$ outcome and assume that 
\[
\supp(w)=\{(0,0),(d,0),(0,d),(e,0)\}
\]
for some $0<e<d$. Then $d=2$ and $e=1$.
\end{prop}
\begin{proof}
The Invertibility Criterion with $\lambda=(2,1,\ldots,1)$ yields $e=1$. The Invertibility Criterion with $\lambda=(3,1,\ldots,1)$ applied to to $(12)\cdot w$ now yields $d=2$.
\end{proof}

\begin{prop}
Let $w$ be a degree-$d$ outcome and assume that 
\[
\supp(w)=\{(0,0),(d,0),(0,e),(d-f,f)\}
\]
for some $0<e,f<d$. Then $d=2$ and $e=f=1$.
\end{prop}
\begin{proof}
The Invertibility Criterion with $\lambda=(2,1,\ldots,1)$ yields $(d-f,f)=(1,d-1)$. In particular, we have $e\leq f$. Applying the same argument to $(12)\cdot w$ with $\lambda=(2,1,\ldots,1)$ if $e\neq f$ or $\lambda=(2,1,\ldots,1,2,1,\ldots,1)$ if $e=f$, we find that $e=1$. In the latter case, we have $E=\{0,1,e,e+1\}$ and $S=\{(0,0),(0,d),(e,0),(e,1)\}$ so that
\[
A^{(d)}_{E,S}=\begin{pmatrix}
A^{(d)}_{\{0,1\},S_1}&0\\
*&A^{(1)}_{\{0,1\},S_2}
\end{pmatrix}
\]
where $S_1=\{(0,0),(0,d)\}$ and $S_2=\{(0,0),(0,1)\}$. The Invertibility Criterion with $\lambda=(3,1,\ldots,1)$ now yields $d=2$. 
\end{proof}

\begin{thm}\label{thm:valid-outcomes-n-three}
Let $w$ be a valid outcome of positive support $\leq 3$. Then $w$ is a nonnegative multiple of one of the following outcomes:
\medskip
\begin{center}
\begin{BVerbatim}[formatcom=\rm\tt]
         1                                   
         · ·         1         ·         1   
 1       · 3 ·       · 2       1 1       · 1 
-1 1    -1 · · 1    -1 · 1    -1 · 1    -1 1 ·
\end{BVerbatim}
\end{center}
\end{thm}
\begin{proof}
We know by the previous results that $\supp^+(w)$ is one of the following:
\[
\{(0,1),(1,0)\},\{(0,3),(1,1),(3,0)\},\{(0,1),(0,2),(2,0)\},\{(0,2),(1,0),(2,0)\},
\]
\[
\{(0,2),(1,1),(2,0)\},\{(0,1),(1,1),(2,0)\},\{(0,2),(1,0),(1,1)\}.
\]
For each of these possible supports $E$, we compute the space of outcomes whose supports are contained in $E\cup \{(0,0)\}$ by computing the space of solutions to the Pascal equations of the corresponding degree. For each $E$, this space has dimension $1$ (over $\QQ)$. We find that the outcomes with support
\[
\{(0,0),(0,1),(0,2),(2,0)\}\mbox{ and }\{(0,0),(0,2),(1,0),(2,0)\}
\] 
are never valid. In each of the other cases, every valid outcome is a multiple of one in the list.
\end{proof}

\begin{example}
\rev{We illustrate the last step of the proof of Theorem~\ref{thm:valid-outcomes-n-three} in which we compute the space of outcomes with given support. Let $d=3$. Then the following coefficients give a basis of the Pascal equations on $V_d$ (by Proposition~\ref{phi-equations}):}
\begin{center}
\begin{BVerbatim}

1           ·           ·           ·
1 ·         1 1         · ·         · ·
1 · ·       2 1 ·       1 1 1       · · ·
1 · · ·     3 1 · ·     3 2 1 ·     1 1 1 1
\end{BVerbatim}
\end{center}
\rev{Consider the support $S=\{(0,0),(0,3),(1,1),(3,0)\}$, illustrated by the following picture:}
\begin{center}
\begin{BVerbatim}

*      
· ·    
· * ·  
* · · *
\end{BVerbatim}
\end{center}
\revml{According to the Pascal equations, chip configurations $w$ with this support are outcomes if and only if they satisfy
\[
w_{0,0} = -w_{0,3},\quad 3w_{0,0} = -w_{1,1}, \quad w_{0,0} = -w_{3,0}.
\]
Therefore, the space of such outcomes is one-dimensional and it contains the valid outcome \[(w_{0,0}, w_{0,3}, w_{3,0}, w_{1,1}) = (-1,1,1,3).\]
For a negative example, let $d=2$. Then the space of Pascal equations is \rrev{spanned by three equations given by the following coefficients:}
}
\begin{center}
\begin{BVerbatim}

1          ·          ·    
1 ·        1 1        · ·  
1 · ·      2 1 ·      1 1 1
\end{BVerbatim}
\end{center}
\rrev{Combined with the support}
\begin{center}
\begin{BVerbatim}

*    
* ·  
* · *
\end{BVerbatim}
\end{center}
\rrev{these equations lead to the conditions}
\[
w_{0,0} + w_{0,1} + w_{0,2} = 0,\quad 2w_{0,0} = -w_{0,1}, \quad w_{0,0} = -w_{2,0}.
\]
The space of solutions to these equations is again one-dimensional but none of the nonzero solutions are valid because the equations imply $w_{0,0} = w_{0,2}$.

\end{example}

\section{Valid outcomes of positive support \texorpdfstring{$4$}{4}}\label{sec:supp+==4}

In this section we prove Theorem~\ref{thm:main_outcomes} for valid outcomes whose positive support has size $4$. To do this we introduce our second tool, the Hyperfield Criterion, which shows that certain subsets of $V_d$ cannot be the support of a valid outcome. We first recall the basic properties of hyperfields.

\subsection{Polynomials over hyperfields}
Denote by $2^H$ the power set of a set $H$. 

\begin{de}
A {\em hyperfield} is a tuple $(H,+,\cdot,0,1)$ consisting of a set $H$, symmetric maps
\[
\rev{\bullet+\bullet}\colon H\times H\to 2^H\setminus\{\emptyset\},\quad\quad \rev{\bullet\cdot\bullet}\colon H\times H\to H
\]
and distinct elements $0,1\in H$ satisfying the following conditions:
\begin{enumerate}
\item The tuple $(H\setminus\{0\},\cdot,1)$ is a group.
\item We have $0\cdot x=0$ and $0+x=\{x\}$ for all $x\in H$.
\item We have $\bigcup_{w\in x+y}(w+z)=\bigcup_{w\in y+z} (x+w)$ for all $x,y,z\in H$.
\item We have $a\cdot(x+y)=(a\cdot x)+(a\cdot y)$ for all $a,x,y\in H$.
\item For every $x\in H$ there is an unique element $-x\in H$ such that $x+(-x)\ni0$.
\end{enumerate}
For subsets $X,Y\subseteq H$, we write
\[
X+Y:=\bigcup_{x\in X,y\in Y}(x+y).
\]
We also identify elements $y\in H$ with the singletons $\{y\}\subseteq H$ so that 
\[
y+X:=X+y:=\bigcup_{x\in X}(x+y).
\]
With this notation, condition (c) can be reformulated as $(x+y)+z=x+(y+z)$ for all $x,y,z\in H$.
\end{de}

See \cite{BB} for more background and uses of hyperfields. 

\begin{de}
Let $H$ be a hyperfield. 
\begin{enumerate}
\item A {\em polynomial} in variables $x_1,\ldots,x_n$ over $H$ is a formal sum
\[
f=\sum_{\rev{(k_1,\ldots,k_n)\in\ZZ_{\geq 0}^n}}s_{k_1\ldots k_n}x_1^{k_1}\cdots x_n^{k_n}
\]
\rev{where $s_{k_1 \dotsc k_n}\in H$ and} $\#\{(k_1,\ldots,k_n)\in\ZZ_{\geq0}^n\mid s_{k_1\ldots k_n}\neq0\}<\infty$. 

\item We denote the set of such polynomials by $H[x_1,\ldots,x_n]$.

\item For $s_1,\ldots,s_n\in H$, we write
\[
f(s_1,\ldots,s_n):=\sum_{k_1,\ldots,k_n\in\ZZ_{\geq0}}s_{k_1\ldots k_n}s_1^{k_1}\cdots s_n^{k_n}\subseteq H.
\]
and we say that $f$ {\em vanishes} at $(s_1,\ldots,s_k)$ when $f(s_1,\ldots,s_k)\ni0$. 
\end{enumerate}
\end{de}

\subsection{The sign hyperfield}

For the remainder of this paper we let $H$ be the {\em sign hyperfield}: it consists of the set $H=\{1,0,-1\}$ with the addition defined by
\[
0+x=x,\quad 1+1=1,\quad (-1)+(-1)=-1,\quad 1+(-1)=\{1,0,-1\}
\]  
and the usual multiplication.

\begin{de}
Let $H$ be the sign hyperfield and let 
\[
f=\sum_{k_1,\ldots,k_n\in\ZZ_{\geq0}}c_{k_1\ldots k_n}x_1^{k_1}\cdots x_n^{k_n}\in\RR[x_1,\ldots,x_n]
\]
be a polynomial. Then we call
\[
\sgn(f):=\sum_{k_1,\ldots,k_n\in\ZZ_{\geq0}}\sgn(c_{k_1\ldots k_n})x_1^{k_1}\cdots x_n^{k_n}\in H[x_1,\ldots,x_n]
\]
the polynomial over $H$ {\em induced} by $f$. We also write 
\[
\sgn(w):=(\sgn(w_1),\ldots,\sgn(w_n))
\]
for all $w=(w_1,\ldots,w_n)\in\RR^n$.
\end{de}

Let $\phi$ be a Pascal equation on $\ZZ^{V_d}$. Then we can represent $\sgn(\phi)$ as a triangle consisting of the symbols \verb|+|, \verb|·|, \verb|-| indicating that a given coeffcient equals $1,0,-1$, respectively.

\begin{ex}\label{ex-pascal-signs}
Take $d=5$. Then the linear forms $\sgn(\varphi_{k,d-k})$ for $k=0,\ldots,d$ can be depicted as:
\medskip

\begin{minipage}[b]{0.16\textwidth}
\begin{verbatim}
+
+ · 
+ · ·
+ · · ·
+ · · · ·
+ · · · · ·
\end{verbatim}
\end{minipage}
\begin{minipage}[b]{0.16\textwidth}
\begin{verbatim}
·
+ + 
+ + ·
+ + · ·
+ + · · ·
+ + · · · ·
\end{verbatim}
\end{minipage}
\begin{minipage}[b]{0.16\textwidth}
\begin{verbatim}
·
· · 
+ + +
+ + + ·
+ + + · ·
+ + + · · ·
\end{verbatim}
\end{minipage}
\begin{minipage}[b]{0.16\textwidth}
\begin{verbatim}
·
· · 
· · ·
+ + + +
+ + + + ·
+ + + + · ·
\end{verbatim}
\end{minipage}
\begin{minipage}[b]{0.16\textwidth}
\begin{verbatim}
·
· · 
· · ·
· · · ·
+ + + + +
+ + + + + ·
\end{verbatim}
\end{minipage}
\begin{minipage}[b]{0.16\textwidth}
\begin{verbatim}
·
· · 
· · ·
· · · ·
· · · · ·
+ + + + + +
\end{verbatim}
\end{minipage}

\medskip
The linear forms $\sgn(\psi_k)$ for $k=0,\ldots,d$ can be depicted as:

\medskip
\begin{minipage}[b]{0.16\textwidth}
\begin{verbatim}
·
· · 
· · ·
· · · ·
· · · · ·
+ + + + + +
\end{verbatim}
\end{minipage}
\begin{minipage}[b]{0.16\textwidth}
\begin{verbatim}
·
· · 
· · ·
· · · ·
+ + + + +
· - - - - -
\end{verbatim}
\end{minipage}
\begin{minipage}[b]{0.16\textwidth}
\begin{verbatim}
·
· · 
· · ·
+ + + +
· - - - -
· · + + + +
\end{verbatim}
\end{minipage}
\begin{minipage}[b]{0.16\textwidth}
\begin{verbatim}
·
· · 
+ + +
· - - -
· · + + +
· · · - - -
\end{verbatim}
\end{minipage}
\begin{minipage}[b]{0.16\textwidth}
\begin{verbatim}
·
+ + 
· - -
· · + +
· · · - -
· · · · + +
\end{verbatim}
\end{minipage}
\begin{minipage}[b]{0.16\textwidth}
\begin{verbatim}
+
· -
· · +
· · · -
· · · · +
· · · · · -
\end{verbatim}
\end{minipage}

\medskip
The linear forms $\sgn(\overline{\psi}_k)$ for $k=0,\ldots,d$ can be depicted as:

\medskip
\begin{minipage}[b]{0.16\textwidth}
\begin{verbatim}
+
+ · 
+ · ·
+ · · ·
+ · · · ·
+ · · · · ·
\end{verbatim}
\end{minipage}
\begin{minipage}[b]{0.16\textwidth}
\begin{verbatim}
-
- + 
- + ·
- + · ·
- + · · ·
· + · · · ·
\end{verbatim}
\end{minipage}
\begin{minipage}[b]{0.16\textwidth}
\begin{verbatim}
+
+ - 
+ - +
+ - + ·
· - + · ·
· · + · · ·
\end{verbatim}
\end{minipage}
\begin{minipage}[b]{0.16\textwidth}
\begin{verbatim}
-
- + 
- + -
· + - +
· · - + ·
· · · + · ·
\end{verbatim}
\end{minipage}
\begin{minipage}[b]{0.16\textwidth}
\begin{verbatim}
+
+ - 
· - +
· · + -
· · · - +
· · · · + ·
\end{verbatim}
\end{minipage}
\begin{minipage}[b]{0.16\textwidth}
\begin{verbatim}
-
· + 
· · -
· · · +
· · · · -
· · · · · +
\end{verbatim}
\end{minipage}

\end{ex}

\begin{prop}\label{prop:R2H}
Let $H$ be the sign hyperfield and \rev{$f\in\RR[x_1,\ldots,x_n].$} Suppose that \rev{$f$ vanishes} at $w\in\RR^n$. Then \rev{$\sgn(f)$ vanishes} at $\sgn(w)\in H^n$.
\end{prop}
\begin{proof}
Write $w=(w_1,\ldots,w_n)$, $s=(s_1,\ldots,s_n)=\sgn(w)$ and
\[
\rev{f} = \sum_{k_1,\ldots,k_n\in\ZZ_{\geq0}}c_{k_1\ldots k_n}x_1^{k_1}\cdots x_n^{k_n}.
\]
Then we have
\[
\sum_{k_1,\ldots,k_n\in\ZZ_{\geq0}}c_{k_1\ldots k_n}w_1^{k_1}\cdots w_n^{k_n}=\rev{f}(w)=0.
\]
If $c_{k_1\ldots k_n}w_1^{k_1}\cdots w_n^{k_n}=0$ for all $k_1,\ldots,k_n\in\ZZ_{\geq0}$, then $\sign(\rev{f})(s_1,\dotsc,s_n)=\{0\}\ni0$ since all summands are zero. Otherwise, we have $c_{\ell_1\ldots \ell_n}w_1^{\ell_1}\cdots w_n^{\ell_n}>0$  for some $\ell_1,\ldots,\ell_n\in\ZZ_{\geq0}$ and $c_{\ell_1'\ldots \ell_n'}w_1^{\ell_1'}\cdots w_n^{\ell_n'}<0$ for some $\ell_1',\ldots,\ell_n'\in\ZZ_{\geq0}$. In this case, $\sign(\rev{f})(s_1,\dotsc,s_n)$ has both $1$ and $-1$ as summands, so $\sign(\rev{f})(s_1,\dotsc,s_n)=H\ni 0$.
\end{proof}

\subsection{The Hyperfield Criterion}

We now state the Hyperfield Criterion. Let $S\subseteq V_d\setminus\{(0,0)\}$ be a subset and define $s=(s_{i,j})_{(i,j)\in V_{d}}\in H^{V_d}$ by
\[
s_{i,j}:=\left\{\begin{array}{cl}-1&\mbox{when $(i,j)=(0,0)$,}\\1&\mbox{when $(i,j)\in S$,}\\0&\mbox{otherwise.}\end{array}\right.
\]
Let $w=(w_{i,j})_{(i,j)\in V_d}\in\ZZ^{V_d}$ be a valid outcome.

\begin{prop}[Hyperfield Criterion]\label{hyperfield-criterion}
Suppose that $\sgn(\phi)$ does not vanish at $s$ for some Pascal equation $\phi$ on $\ZZ^{V_d}$. Then $\supp^+(w)\neq S$.
\end{prop}
\begin{proof}
Suppose that $\supp^{+}(w)=S$. Then $\sgn(w)=s$. Since all Pascal equations $\phi$ on $\ZZ^{V_d}$ vanish at $w$, we see that all polynomials over $H$ induced by Pascal equations on $\ZZ^{V_d}$ vanish at $s$ by Proposition~\ref{prop:R2H}.
\end{proof}

\subsection{Pascal equations}

In this subsection, we consider the equations over $H$ induced by the Pascal equations $\psi_k,\overline{\psi}_k,\varphi_{a,b}$ for $k\in\{0,\ldots,d\}$ and $(a,b)\in V_d$ of degree $d$.

\begin{de}
Let $s=(s_{i,j})_{(i,j)\in V_d}\in H^{V_{d}}$.
\begin{enumerate}
\item We call $s$ a {\em sign configuration}.
\item The {\em positive support} of $s$ is $\supp^+(s):=\{(i,j)\in V_d\mid s_{i,j}=1\}$.
\item The {\em negative support} of $s$ is $\supp^-(s):=\{(i,j)\in V_d\mid s_{i,j}=-1\}$.
\item The {\em support} of $s$ is $\supp(s):=\{(i,j)\in V_d\mid s_{i,j}\neq0\}=\supp^+(s)\cup\supp^-(s)$.
\item We call $\deg(s):=\max\{\deg(i,j)\mid (i,j)\in V_d, s_{i,j}\neq0\}$ the {\em degree} of $s$.
\item We say that $s$ is {\em valid} when $s=0$ or $\supp^-(s)=\{(0,0)\}$.
\item We say that $s$ is {\em weakly valid} when for all $(i,j)\in\supp^-(s)$ one of the following holds:
\begin{enumerate}
\item $0\leq i,j\leq 3$,
\item $0\leq i\leq3$ and $\deg(i,j)\geq d-3$, or 
\item $0\leq j\leq3$ and $\deg(i,j)\geq d-3$.
\end{enumerate}
\end{enumerate}
\end{de}

\begin{lm}
Let $w\in\ZZ^{V_d}$ be a chip configuration.
\begin{enumerate}
\item We have $\supp^+(\sgn(w))=\supp^+(w)$.
\item We have $\supp^-(\sgn(w))=\supp^-(w)$.
\item We have $\deg(\sgn(w))=\deg(w)$.
\item The sign configuration $\sgn(w)$ is (weakly) valid if and only if $w$ is (weakly) valid.
\end{enumerate}
\end{lm}
\begin{proof}
This follows from the definitions.
\end{proof}

\begin{samepage}
\begin{lm}\label{sign-pascal-equations}
\phantom{emptytext}
\begin{enumerate}
\item We have
\[
\sgn(\varphi_{a,b})=\sum_{i=0}^a\sum_{j=0}^bx_{i,j}
\]
for all $(a,b)\in V_d$ of degree $d$.\itembreak
\item We have
\[
\sgn(\psi_k)=\sum_{(i,j)\in S_k^+}x_{i,j}-\sum_{(i,j)\in S_k^-}x_{i,j}\mbox{ and }\sgn(\overline{\psi}_k)=\sum_{(i,j)\in S_k^+}x_{j,i}-\sum_{(i,j)\in S_k^-}x_{j,i},
\]
where
\begin{eqnarray*}
S_k^+&:=&\{(i,j)\mid 0\leq j\leq k,\quad k-j\leq i\leq d-j, \quad j\equiv k\pmod 2\},\\
S_k^-&:=&\{(i,j)\mid 0\leq j\leq k,\quad k-j\leq i\leq d-j, \quad j\not\equiv k\pmod 2\},
\end{eqnarray*}
for all $k\in\{0,\ldots,d\}$.
\end{enumerate}
\end{lm}
\end{samepage}
\begin{proof}
This follows from Propositions~\ref{phi-equations} and~\ref{psi-equations}.
\end{proof}

\begin{prop}\label{sign-pascal-valid-hyperfield}
Let $s\in H^{V_d}$ be a valid sign configuration of degree $d\geq 1$.
\begin{enumerate}
\item For $(a,b)\in V_d$ of degree $d$, if $\sgn(\varphi_{a,b})$ vanishes at $s$, then $\sgn(\varphi_{a,b})(s)=H$.
\item If $\sgn(\psi_0),\ldots,\sign(\psi_d)$ vanish at $s$, then $\sgn(\psi_0)(s)=\cdots=\sgn(\psi_d)(s)=H$.
\item If $\sgn(\overline{\psi}_0),\ldots,\sign(\overline{\psi}_d)$ vanish at $s$, then $\sgn(\overline{\psi}_0)(s)=\cdots=\sgn(\overline{\psi}_d)(s)=H$.
\end{enumerate}
\end{prop}
\begin{proof}
Note that since $\deg(s)=d\geq 1$, we have $s_{0,0}=-1$, $s_{i,j}\geq0$ for all $(i,j)\in V_d\setminus\{(0,0)\}$ and $s_{k,d-k}=1$ for some $k\in\{0,\ldots,d\}$.

\itembreak
(a)
Let $(a,b)\in V_d$ have degree $d$ and suppose that
\[
\sum_{i=0}^a\sum_{j=0}^bs_{i,j}\ni0.
\]
Since $s_{0,0}=-1$, this is only possible when $s_{i,j}=1$ for some $i\in\{0,\ldots,a\}$ and $j\in\{0,\ldots,b\}$ and so $\sgn(\varphi_{a,b})(s)=H$.

\itembreak
(b)
Suppose that $\sgn({\psi}_0),\ldots,\sign({\psi}_d)$ vanish at $s$. We have 
\[
\sgn(\psi_k)=\sum_{(i,j)\in S_k^+}x_{i,j}-\sum_{(i,j)\in S_k^-}x_{i,j}
\]
where $S_k^+,S_k^-\subseteq V_d$ are as in Lemma~\ref{sign-pascal-equations}. We have $\psi_0=\varphi_{d,0}$ and so $\sgn(\psi_0)(s)=\sgn(\varphi_{d,0})(s)=H$. For $k>0$, note that $(0,0)\not\in S_k^+\cup S_k^-$ and in particular $s_{i,j}\geq0$ for all $(i,j)\in S_k^+\cup S_k^-$. So for each $k\in\{1,\ldots,d\}$, we see that either
\newcommand{\rma}{\mathrm{a}}
\newcommand{\rmb}{\mathrm{b}}
\begin{itemize}
\item[($\rma_k$)] $s_{i,j}=0$ for all $(i,j)\in S_k^+\cup S_k^-$; or
\item[($\rmb_k$)] $s_{i,j}=1$ for some $(i,j)\in S_k^+$ and $s_{i,j}=1$ for some $(i,j)\in S_k^-$. 
\end{itemize}
We prove that ($\rmb_k$) holds for $k=d,\ldots,1$ recursively, which implies that $\sgn(\psi_k)(s)=H$.\medskip

The union $S_d^+\cup S_d^-$ consists of all \rev{vertices} in $V_d$ of degree $d$. So ($\rma_d$) cannot hold. So ($\rmb_d$) holds. Next, let $k\in\{1,\ldots,d-1\}$ and suppose that ($\rmb_{k+1}$) holds. Then $s_{i,j}=1$ for some $(i,j)\in S_{k+1}^-$. We have $S_{k+1}^-\subseteq S_k^+$ and hence ($\rma_k$) cannot hold. Hence ($\rmb_k$) holds. So ($\rmb_k$) holds for all $k\in\{1,\ldots,d\}$.

\itembreak
(c)
The proof of this part is the same as that of the previous part.
\end{proof}

\begin{re}\label{re:hyperfield_conditions_on_place}
Let $w\in\ZZ^{V_d}$ be a valid outcome of degree $d$. Then $\sgn(\phi)$ vanishes at 
\[
s=(s_{i,j})_{(i,j)\in V_d}=\sgn(w)
\]
for all Pascal equations $\phi$ on $\ZZ^{V_d}$. Proposition~\ref{sign-pascal-valid-hyperfield} tells us that in this case, we have
\[
\sgn(\varphi_{0,d})(s),\ldots,\sgn(\varphi_{d,0})(s),\sgn(\psi_0)(s),\ldots,\sgn(\psi_d)(s),\sgn(\overline{\psi}_0)(s),\ldots,\sgn(\overline{\psi_d})(s)=H,
\]
which shows that the following hold:
\begin{enumerate}
\item for all $(a,b)\in V_d$ of degree $d$, there exist $i\in\{0,\ldots,a\}$ and $j\in\{0,\ldots,b\}$ with $s_{i,j}=1$;
\item for all $k\in\{1,\ldots,d\}$, there exist $(i,j)\in S_k^+$ with $s_{i,j}=1$ and $(i,j)\in S_k^-$ with $s_{i,j}=1$; and
\item for all $k\in\{1,\ldots,d\}$, there exist $(i,j)\in S_k^+$ with $s_{j,i}=1$ and $(i,j)\in S_k^-$ with $s_{j,i}=1$.
\end{enumerate}

Here we note that $s_{i,j}=1$ if and only if $(i,j)\in\supp^+(w)$. So we can view these conditions as restrictions on the set $\supp^+(w)$.
\end{re}

\subsection{Contractions of hyperfield solutions}\label{subsec:contraction-hyperfield}

\rev{In this subsection we make progress by reducing the set of things to be considered for our classification from infinite to finite. \rrev{For a weakly valid} $s\in H^{V_d}$ we define a hyperfield vector $\contr_d(s)\in H^{\Xi}$ where $\Xi$ is a finite set which is independent of $d$. The vector $\contr_d(s)$ is obtained from $s$ by considering only a subset of the entries of $s$, and by replacing certain sets of entries of $s$ (with cardinalities growing linearly with $d$) by their sum (with cardinality one). We call $\contr_d(s)$ the \emph{contraction} of $s$ because we think of this summation of entries of $s$ as a contraction of $s$. The elements of $H^{\Xi}$ have their own notion of being valid such that
\[
\text{$s$ valid}\Rightarrow \text{$\contr_d(s)$ valid.}
\]
This turns out to be enough to classify valid chipsplitting outcomes $w$ of positive support $4$ by passing to $\contr_d(\sgn(w))$ and analyzing the finitely many possibilities.}

\mybreak

We \rev{start} by considering the four-entries thick outer ring of the triangle $V_d$. We divide the outer ring into six areas as illustrated in Figure~\ref{fig:triangle-areas}. One of these, Area $D$, splits further into $D^{(0)}$ and $D^{(1)}$ according to the parity of the $i$-coordinate of its entries.

\definecolor{cornercolor}{HTML}{FFE194}
\newcommand{\cornercolor}{orange}
\definecolor{bordercolor}{HTML}{E8F6EF}
\definecolor{altbordercolor}{HTML}{B8DFD8}
\newcommand{\bordercolor}{dark blue}
\newcommand{\altbordercolor}{light blue}
\begin{figure}[ht]
\begin{tikzpicture}[scale=0.25]
%triangle represents the case d=15.
%the lines are supposed to fit `between' the variables
%nodes are named outer-to-inner and counter-clockwise,
%starting from the origin.
\node (a) at (0,0) {};
\node (b) at (4,0) {};
\node (c) at (12.5,0) {};
\node (d) at (17,0) {};
\node (e) at (13,4) {};
\node (f) at (4,13) {};
\node (g) at (0,17) {};
\node (h) at (0,12.5) {};
\node (i) at (0,4) {};
\node (j) at (4,4) {};
\node (k) at (8.5,4) {};
\node (l) at (4,8.5) {};
\draw[fill=cornercolor]
  (a.center)--(b.center)--(j.center)--(i.center)--cycle;
\draw[fill=cornercolor]
  (c.center)--(d.center)--(e.center)--(k.center)--cycle;
\draw[fill=cornercolor]
  (g.center)--(h.center)--(l.center)--(f.center)--cycle;
\draw[fill=bordercolor]
  (b.center)--(c.center)--(k.center)--(j.center)--cycle;
\draw[fill=bordercolor]
  (l.center)--(k.center)--(e.center)--(f.center)--cycle;
\draw[fill=bordercolor]
  (i.center)--(j.center)--(l.center)--(h.center)--cycle;
%nodes for labels
\node (x) at (2,2) {$X$};
\node (r) at (7.375,2) {$B$};
\node (z) at (12.75,2) {$Z$};
\node (y) at (2,12.75) {$Y$};
\node (bigc) at (2,7.375) {$C$};
\fill[altbordercolor]
(5,12)--(6,11)--(6,6.5)--(5,7.5);
\fill[altbordercolor]
(7,10)--(8,9)--(8,4.5)--(7,5.5);
\fill[altbordercolor]
(9,8)--(10,7)--(10,4)--(9,4);
\fill[altbordercolor]
(11,6)--(12,5)--(12,4)--(11,4);
%redraw the lines over the fill
\draw (l)--(k)--(e)--(f)--cycle;
%label last area
\node (bigd) at (7,7.75) {$D$};
%add curly braces
\draw[decorate, decoration = {brace, mirror, raise = 2pt}, semithick] (a.center)--(b.center);
\draw[decorate, decoration = {brace, mirror, raise = 2pt}, semithick] (i.center)--(a.center);
\draw[decorate, decoration = {brace, mirror, raise = 2pt}, semithick] (c.center)--(d.center);
\draw[decorate, decoration = {brace, mirror, raise = 2pt}, semithick] (g.center)--(h.center);
%label the curly braces
\node at (2,-1.5) {$4$};
\node at (-1.5,2) {$4$};
\node at (14.75,-1.5) {$4$};
\node at (-1.5,14.75) {$4$};
\end{tikzpicture}
\caption{Dividing the outer ring of the triangle $V_d$ into six areas for Subsection~\ref{subsec:contraction-hyperfield}. The area $D$ splits into two parts $D^{(0)}$ and $D^{(1)}$ by alternating the columns.}\label{fig:triangle-areas}
\end{figure}
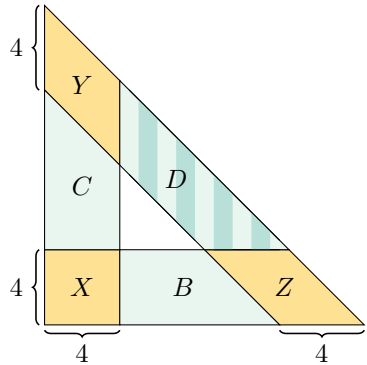

Let $x_{i,j}$ be the formal variables indexed by the elements of $V_d$. We rename and combine these variables according to their assigned area:
\begin{eqnarray*}
y_{i,j}&:=&x_{i,d-3-i+j}\quad\mbox{for }i,j\in\{0,\ldots,3\},\\
z_{i,j}&:=&x_{d-3-j+i,j}\quad\mbox{for }i,j\in\{0,\ldots,3\},\\
b_j&:=&\rev{\textstyle\sum_{i=4}^{d-4-j}x_{i,j}}\quad\mbox{for }j\in\{0,\ldots,3\},\\
c_i&:=&\rev{\textstyle\sum_{j=4}^{d-4-i} x_{i,j}}\quad\mbox{for }i\in\{0,\ldots,3\},\\
d_k^{(0)}&:=&\rev{\textstyle\sum_{\ell=2}^{\floor{(d-4-k)/2}} x_{2\ell, d-2\ell-k}}
\quad\mbox{for }k\in\{0,\ldots,3\},\\
d_k^{(1)}&:=&\rev{\textstyle\sum_{\ell=2}^{\floor{(d-5-k)/2}} x_{2\ell+1, d-(2\ell+1)-k}}
\quad\mbox{for }k\in\{0,\ldots,3\}.
\end{eqnarray*}
\rev{For the $b$s, $c$s, and $d$s, this results in compressing a number of variables that grows linearly with $d$ down into a single variable, whereas the renamings of the first two lines are just for convenience.}
\revml{Next, we consider two sets $\Phi_1$, $\Phi_2$ of Pascal equations on $V_d$ defined by
\begin{equation*}
\begin{split}
\Phi_1 = \{&\psi_1,\psi_2,\psi_3, \\
& \overline{\psi}_1, \overline{\psi}_2, \overline{\psi}_3, \\
& \varphi_{1,d-1}, \varphi_{2,d-2}, \varphi_{3,d-3}, \\
& \varphi_{d-1,1}, \varphi_{d-2,2}, \varphi_{d-3,3}\},
\end{split}
\quad\quad\quad
\begin{split}
\Phi_2 =  \{&\psi_{d-3}, \psi_{d-2}, \psi_{d-1}, \psi_d,\\
& \overline{\psi}_{d-3}, \overline{\psi}_{d-2}, \overline{\psi}_{d-1}, \overline{\psi}_{d}\},\\
\\
\\
\end{split}
\end{equation*}
and call their union $\Phi$.}
We \rev{want to} apply these Pascal equations to valid sign configurations $s\in H^{V_d}$. \rev{The next two lemmas show that this operation is governed by a finite set of linear forms over $H$ which is independent of $d$. \rrev{We assume that $d$ is large enough} to have each of the above variables defined. Specifically, the assumption $d\geq 11$ ensures that all $x_{i,j}$, $y_{i,j}$, $z_{i,j}$, $b_j$, $c_i$, and $d_k^{(0)}$ are defined, whereas $d\geq 12$ additionally ensures that all $d_k^{(1)}$ are defined.}

\begin{lm}\label{lm:Phi1}
Assume that $d\geq 11$ and let $\phi\in\Phi_1$. Then 
\[
\sgn(\phi)=\widehat{\phi}(x_{i,j},y_{i,j},z_{i,j},b_j,c_i,d^{(0)}_k,d^{(1)}_k)
\]
for some linear form
\[
\widehat{\phi}(x_{i,j},y_{i,j},z_{i,j},b_j,c_i,d^{(0)}_k,d^{(1)}_k)\in H[x_{i,j},y_{i,j},z_{i,j},b_j,c_i,d^{(0)}_k,d^{(1)}_k\mid i,j,k\in\{0,\ldots,3\}]
\]
Moreover, the linear form $\widehat{\phi}$ does not depend on $d$.
\end{lm}
\begin{proof}
\rev{As illustrated in Example~\ref{ex-pascal-signs}, the equations $\overline \psi_1$, $\overline\psi_2$, and $\overline\psi_3$ are supported in the sections $Y$, $C$, and $X$ from Figure~\ref{fig:triangle-areas}. Furthermore, the signs of their coefficients are constant along the columns of $C$. Therefore, the set of coefficients in column $i$ of $C$ may be replaced with a single coefficient of the variable $c_i$. This results in linear forms $\widehat{\overline \psi}_1$, $\widehat{\overline\psi}_2$, and $\widehat{\overline\psi}_3$ over $H$ where only the variables $x_{i,j}$, $y_{i,j}$ and $c_i$ appear. The coefficients of these linear forms are independent of $d$ because the constant sign of the $\overline\psi_i$ along the columns of $C$ does not change with $d$. The same argument applies to the equations $\varphi_{1,d-1}$, $\varphi_{2,d-2}$, and $\varphi_{3, d-3}$. The remaining Pascal equations in $\Phi_1$ are handled by a symmetric argument involving the sections $X$, $B$, and $Z$ and their associated variable, considering the rows of $B$ instead of the columns of $C$.}
\end{proof}

\begin{ex}
Take $\phi=\varphi_{3,d-3}$. We can depict $\sgn(\phi)$ as follows (for $d=11$):
\begin{center}
\includegraphics[width=0.35\textwidth]{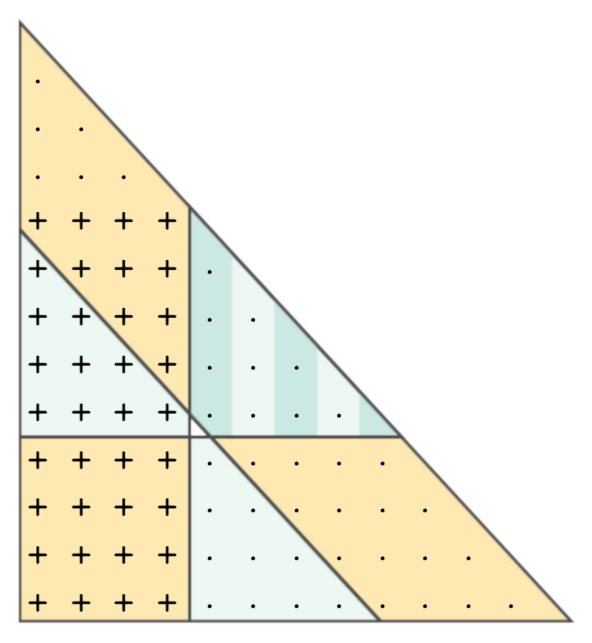}
\end{center}

Take\vspace*{-3pt}
\[
\widehat{\phi} =\widehat{\varphi}_{3,d-3}= \sum_{i=0}^3\sum_{j=0}^3x_{i,j}+\sum_{i=0}^3c_i+\sum_{i=0}^3\sum_{j=0}^iy_{i,j}.\vspace*{-3pt}
\]
Then we see that
\begin{eqnarray*}
\sgn(\varphi_{3,d-3})&=&\sum_{i=0}^3\sum_{j=0}^{d-3}x_{i,j}\\
&=&\sum_{i=0}^3\sum_{j=0}^3x_{i,j}+\sum_{i=0}^3(x_{i,4}+\ldots+x_{i,d-4-i})+\sum_{i=0}^3\sum_{j=0}^ix_{i,d-3-i+j}\\
&=&\widehat{\varphi}_{3,d-3}.
\end{eqnarray*}
Indeed, the linear form $\widehat{\phi}$ is the same for every $d\geq 11$.
\end{ex}

\begin{lm}\label{lm:Phi2}
Assume that $d\geq 12$ and let $\phi\in\Phi_2$. Then 
\[
\sgn(\phi)=\left\{\begin{array}{cl}\widehat{\phi}^{\mathrm{even}}(x_{i,j},y_{i,j},z_{i,j},b_j,c_i,d^{(0)}_k,d^{(1)}_k)&\mbox{when $d$ is even},\\\widehat{\phi}^{\mathrm{odd}}(x_{i,j},y_{i,j},z_{i,j},b_j,c_i,d^{(0)}_k,d^{(1)}_k)&\mbox{when $d$ is odd}\end{array}\right.
\]
for some linear forms
\[
\widehat{\phi}^{\mathrm{even}},\widehat{\phi}^{\mathrm{odd}}\in H[x_{i,j},y_{i,j},z_{i,j},b_j,c_i,d^{(0)}_k,d^{(1)}_k\mid i,j,k\in\{0,\ldots,3\}].
\]
Moreover, the linear forms $\widehat{\phi}^{\mathrm{even}},\widehat{\phi}^{\mathrm{odd}}$ do not depend on $d$.
\end{lm}

\begin{proof}
\rev{
Analogous to the proof of Lemma~\ref{lm:Phi1}. As Example~\ref{ex-pascal-signs} illustrates, the equations in $\Phi_2$ are supported in $Y$, $D$, and $Z$. The signs of their coefficients alternate along each northwest-to-southeast diagonal of $D$. Therefore, the coefficients of the variables associated to these diagonals may be replaced by a single coefficient in front of the difference $(d_k^{(0)} - d_k^{(1)})$. This results in linear forms over $H$ where only the $y_{i,j}$, $z_{i,j}$, $d_k^{(0)}$, and $d_k^{(1)}$ appear. These linear forms are independent of $d$ of the same parity, which can be seen most directly by applying the formulas for the $\psi$ and $\overline\psi$ in Proposition~\ref{phi-equations}.
}
\end{proof}

\begin{remark}
\rev{For different parities of $d$, the $\widehat\psi_{d-k}$ differ in $Z$ while the $\widehat{\overline \psi}_{d-k}$ differ in $Y$ and $D$.}
\end{remark}

\begin{ex}
Take $\phi=\psi_{d-3}$. We can depict $\sgn(\phi)$ (for $d=12$) as follows:
\begin{center}
\includegraphics[width=0.35\textwidth]{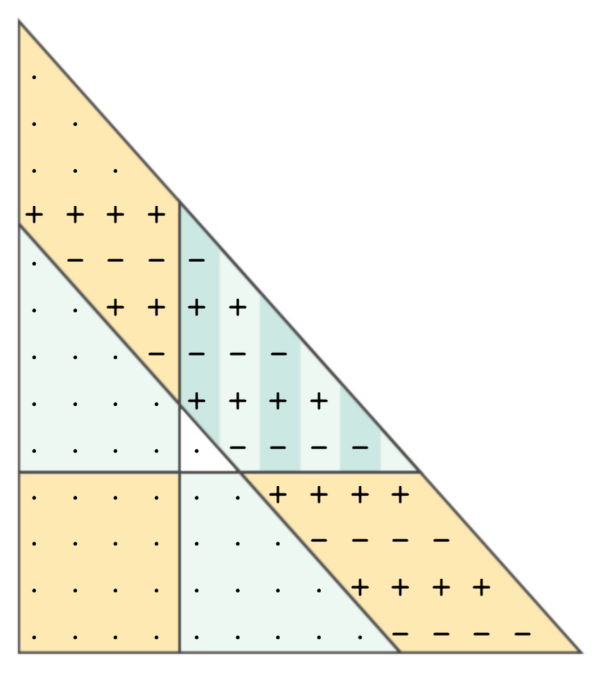}
\end{center}
Take
\begin{eqnarray*}
\widehat{\phi}^{\mathrm{even}}=\widehat{\psi}^{\mathrm{even}}_{d-3}&=&\sum_{i=0}^3\sum_{j=0}^i(-1)^{i+j}y_{i,j}-\sum_{i=0}^1\sum_{k=0}^3(-1)^{i+k}d^{(i)}_k -\sum_{j=0}^3\sum_{i=0}^3(-1)^jz_{i,j},\\
\widehat{\phi}^{\mathrm{odd}}=\widehat{\psi}^{\mathrm{odd}}_{d-3}&=&\sum_{i=0}^3\sum_{j=0}^i(-1)^{i+j}y_{i,j}-\sum_{i=0}^1\sum_{k=0}^3(-1)^{i+k}d^{(i)}_k +\sum_{j=0}^3\sum_{i=0}^3(-1)^jz_{i,j}.
\end{eqnarray*}
Then we see that
\begin{eqnarray*}
\sgn(\psi_{d-3})&=&(-1)^{d-3}\sum_{j=0}^{d-3}\sum_{i=0}^3(-1)^jx_{d-3-j+i,j}\\
&=&\sum_{i=0}^3\sum_{j=0}^i(-1)^{i+j}x_{i,d-3-i+j}-\sum_{i=0}^1\sum_{k=0}^3(-1)^{i+k}d^{(i)}_k +\sum_{j=0}^3\sum_{i=0}^3(-1)^{d-1+j}x_{d-3-j+i,j}\\
&=&\left\{\begin{array}{cl}\widehat{\psi}^{\mathrm{even}}_{d-3}&\mbox{when $d$ is even},\\\widehat{\psi}^{\mathrm{odd}}_{d-3}&\mbox{when $d$ is odd}.\end{array}\right.
\end{eqnarray*}
Indeed, the linear forms $\widehat{\phi}^{\mathrm{even}}$ and $\widehat{\phi}^{\mathrm{odd}}$ are \rev{independent of $d$ for $d\geq 12$.}
\end{ex}

Next we carry out the same subdivision as above but with the coordinates of the elements $s\in H^{V_d}$ instead of formal variables. We start by defining the index set
\[
\Xi = \{0,1,2,3\}^2\sqcup\{0,1,2,3\}^2\sqcup \{0,1,2,3\}^2\sqcup\{0,1,2,3\}\sqcup\{0,1,2,3\}\sqcup\{0,1,2,3\}\sqcup\{0,1,2,3\}.
\]
We write elements of $H^\Xi$ as  
\[
\theta=(s,r,t,\alpha,\beta,\gamma)=\left((s_{i,j})_{i,j=0}^3,(r_{i,j})_{i,j=0}^3,(t_{i,j})_{i,j=0}^3,(\alpha_i)_{i=0}^3,(\beta_j)_{j=0}^3,(\gamma^{(0)}_k)_{k=0}^3,(\gamma^{(1)}_k)_{k=0}^3\right).
\]

\begin{de}\phantom{emptytext}
\begin{enumerate}
\item We say that $\theta$ is {\em valid} when $\theta=0$ or when $s_{0,0}=-1$ and $r_{i,j},t_{i,j},\alpha_i,\beta_j,\gamma^{(0)}_k,\gamma^{(1)}_k\geq0$ for all $i,j,k\in\{0,\ldots,3\}$ and $s_{i,j}\geq0$ for all $(i,j)\in\{0,\ldots,3\}^2\setminus\{(0,0)\}$.
\item We say that $\theta$ is {\em weakly valid} when $\alpha_i,\beta_j,\gamma^{(0)}_k,\gamma^{(1)}_k\geq0$ for all $i,j,k\in\{0,\ldots,3\}$.
\end{enumerate}
\end{de}

Thus $\theta$ is weakly valid if and only if its negative support is contained in the areas $X,Y,Z$ of Figure~\ref{fig:triangle-areas}.

\mybreak
For $d\geq 11$ and $s=(s_{i,j})_{(i,j)\in V_d}\in H^{V_d}$ weakly valid, we write
\[
\contr_d(s):=\left((s_{i,j})_{i,j=0}^3,(r_{i,j})_{i,j=0}^3,(t_{i,j})_{i,j=0}^3,(\alpha_i)_{i=0}^3,(\beta_j)_{j=0}^3,(\gamma^{(0)}_k)_{k=0}^3,(\gamma^{(1)}_k)_{k=0}^3\right)\in H^{\Xi},
\]
where we have
\begin{eqnarray*}
r_{i,j}&:=&s_{i,d-3-i+j}\quad\mbox{for }i,j\in\{0,\ldots,3\},\\
t_{i,j}&:=&s_{d-3-j+i,j}\quad\mbox{for }i,j\in\{0,\ldots,3\},\\
\alpha_i&:=&\rev{\textstyle\sum_{j=4}^{d-4-i}s_{i,j}}
\quad\mbox{for }i\in\{0,\ldots,3\},\\
\beta_j&:=&\rev{\textstyle\sum_{i=4}^{d-4-j} s_{i,j}}\quad\mbox{for }j\in\{0,\ldots,3\},\\
\gamma_k^{(0)}&:=&\rev{\textstyle\sum_{\ell=2}^{\floor{(d-4-k)/2}} s_{2\ell, d-2\ell-k}}\quad\mbox{for }k\in\{0,\ldots,3\},\\
\gamma_k^{(1)}&:=&\rev{\textstyle\sum_{\ell=2}^{\floor{(d-5-k)/2}} s_{2\ell+1, d-(2\ell+1)-k}}\quad\mbox{for }k\in\{0,\ldots,3\}.
\end{eqnarray*}

Let $s_1,\ldots,s_k\in H\setminus\{-1\}$. Then $s_1+\ldots+s_k$ always consists of a single element, namely the element $\max(s_1,\ldots,s_k)$. So the weakly valid assumption ensures that the hyperfield sums in this definition evaluate to a single element of $H$. Note that when $s\in H^{V_d}$ is (weakly) valid, then $\contr_d(s)$ is (weakly) valid as well.

\mybreak
Let the coordinates $x_{i,j},y_{i,j},z_{i,j},b_j,c_i,d^{(0)}_k,d^{(1)}_k$ be dual to $s_{i,j},r_{i,j},t_{i,j},\alpha_i,\beta_j,\gamma^{(0)}_k,\gamma^{(1)}_k$. This allows us to view $\{\widehat{\phi}\mid\phi\in\Phi_1\}$, $\{\widehat{\phi}^{\mathrm{even}}\mid\phi\in\Phi_2\}$ and $\{\widehat{\phi}^{\mathrm{odd}}\mid\phi\in\Phi_2\}$ as sets of equations on $H^\Xi$. See Figure~\ref{fig:contr_d} for a visualisation of $\contr_d$.

\begin{figure}[h]
\[
\begin{array}{lllllllllllllllll}
y_{0,3}\\
y_{0,2}&y_{1,3}\\
y_{0,1}&y_{1,2}&y_{2,3}\\
y_{0,0}&y_{1,1}&y_{2,2}&y_{3,3}\\
c_0&y_{1,0}&y_{2,1}&y_{3,2}&d^{(0)}_0\\
c_0&c_1&y_{2,0}&y_{3,1}&d^{(0)}_1&d^{(1)}_0\\
c_0&c_1&c_2&y_{3,0}&d^{(0)}_2&d^{(1)}_1&d^{(0)}_0\\
c_0&c_1&c_2&c_3&d^{(0)}_3&d^{(1)}_2&d^{(0)}_1&d^{(1)}_0\\
c_0&c_1&c_2&c_3&&d^{(1)}_3&d^{(0)}_2&d^{(1)}_1&d^{(0)}_0\\
c_0&c_1&c_2&c_3&&&d^{(0)}_3&d^{(1)}_2&d^{(0)}_1&d^{(1)}_0\\
c_0&c_1&c_2&c_3&&&&d^{(1)}_3&d^{(0)}_2&d^{(1)}_1&d^{(0)}_0\\
c_0&c_1&c_2&c_3&&&&&d^{(0)}_3&d^{(1)}_2&d^{(0)}_1&d^{(1)}_0\\
c_0&c_1&c_2&c_3&&&&&&d^{(1)}_3&d^{(0)}_2&d^{(1)}_1&d^{(0)}_0\\
x_{0,3}&x_{1,3}&x_{2,3}&x_{3,3}&b_3&b_3&b_3&b_3&b_3&b_3&z_{0,3}&z_{1,3}&z_{2,3}&z_{3,3}\\
x_{0,2}&x_{1,2}&x_{2,2}&x_{3,2}&b_2&b_2&b_2&b_2&b_2&b_2&b_2&z_{0,2}&z_{1,2}&z_{2,2}&z_{3,2}\\
x_{0,1}&x_{1,1}&x_{2,1}&x_{3,1}&b_1&b_1&b_1&b_1&b_1&b_1&b_1&b_1&z_{0,1}&z_{1,1}&z_{2,1}&z_{3,1}\\
x_{0,0}&x_{1,0}&x_{2,0}&x_{3,0}&b_0&b_0&b_0&b_0&b_0&b_0&b_0&b_0&b_0&z_{0,0}&z_{1,0}&z_{2,0}&z_{3,0}
\end{array}
\]
\caption{A visualisation of $\contr_d$ for $d=16$.}\label{fig:contr_d}
\end{figure}

\begin{de}
Let $\theta\in H^\Xi$. We define the {\em positive support} of $\theta$ to be the set $\supp^+(\theta)$ of symbols $x_{i,j},y_{i,j},z_{i,j},b_j,c_i,d^{(0)}_k,d^{(1)}_k$ with $i,j,k\in\{0,\ldots,3\}$ such that the symbol evaluated at $\theta$ equals $1$.
\end{de}

\begin{ex}
Let
\[
\theta=\left((s_{i,j})_{i,j=0}^3,(r_{i,j})_{i,j=0}^3,(t_{i,j})_{i,j=0}^3,(\alpha_i)_{i=0}^3,(\beta_j)_{j=0}^3,(\gamma^{(0)}_k)_{k=0}^3,(\gamma^{(1)}_k)_{k=0}^3\right)\in H^\Xi
\]
be defined by $s_{0,0}=-1$, $s_{0,3}=s_{1,1}=s_{3,0}=\gamma^{(0)}_0=\gamma^{(1)}_0=1$ and by setting all other entries to $0$. Then $\theta$ is valid and $\supp^+(\theta)=\{x_{0,3},x_{1,1},x_{3,0},d^{(0)}_0,d^{(1)}_0\}$.
\end{ex}

\rev{For the next proposition} we use the following notation:
\begin{enumerate}
\item Denote by $\Gamma_d$ the set of valid $s\in H^{V_d}$ of degree $d$ such that $\sgn(\phi)(s)=H$ for all $\phi\in\Phi$. 
\item Denote by $\Gamma^{\mathrm{even}}$ the set of valid $\theta\in H^\Xi$ such that $\widehat{\phi}^{\mathrm{even}}(\theta)=H$ for all $\phi\in\Phi$. 
\item Denote by $\Gamma^{\mathrm{odd}}$ the set of valid $\theta\in H^\Xi$ such that $\widehat{\phi}^{\mathrm{odd}}(\theta)=H$ for all $\phi\in\Phi$.  
\end{enumerate}
Here we set $\widehat{\phi}^{\mathrm{even}}:=\widehat{\phi}^{\mathrm{odd}}:=\widehat{\phi}$ for all $\phi\in\Phi_1$. \rrev{For fixed $d$,} we view $\contr_d$ as a map from the space of weakly valid outcomes in $H^{V_d}$ to $H^{\Xi}$.

\begin{prop}\label{prop:hyperfield_contractions}
Assume that $d\geq 12$. Then 
\[
\Gamma_d=\left\{\begin{array}{cl}\contr_d^{-1}(\Gamma^{\mathrm{even}})&\mbox{when $d$ is even},\\\contr_d^{-1}(\Gamma^{\mathrm{odd}})&\mbox{when $d$ is odd}.\end{array}\right. 
\]
In particular, we have $\contr_d(\sgn(w))\in\Gamma^{\mathrm{even}}\cup\Gamma^{\mathrm{odd}}$ for all valid outcomes $w\in\ZZ^{V_d}$. 
\end{prop}
\begin{proof}
For $d$ even, Lemmas~\ref{lm:Phi1} and~\ref{lm:Phi2} imply $\widehat{\phi}^{\mathrm{even}}(\mathrm{contr}_d(s)) = \sign(\phi)(s)$ for all $s\in H^{V_d}$ and $\phi\in \Phi$. Thus $s\in \Gamma_d$ is equivalent to $\mathrm{contr}_d(s)\in \Gamma^{\mathrm{even}}$ for all $s\in H^{V_d}$. The same reasoning applies when $d$ is odd, which proves the first statement.
The second~ statement follows because $\sign(w)\in\Gamma_d$ by Proposition~\ref{sign-pascal-valid-hyperfield}.
\end{proof}

\subsection{Valid outcomes of positive support \texorpdfstring{$4$}{4}}

We now finally classify the valid outcomes $w\in\ZZ^{V_d}$ whose positive support has size $4$.

\begin{thm}\label{thm:possup=4}
Let $w\in\ZZ^{V_d}$ be a valid outcome and suppose that $\#\supp^+(w)=4$. Then $\deg(w)\leq 5$.
\end{thm}

Let $\Omega_d$ be the set of valid $s=(s_{i,j})_{(i,j)\in V_d}\in H^{V_d}$ of degree $d$ such that $\#\supp^+(s)=4$ and 
\[
\sgn(\psi_k)(s)=\sgn(\overline{\psi}_k)(s)=\sgn(\varphi_{a,b})(s)=H
\]
for all $k\in\{0,\ldots,d\}$ and $(a,b)\in V_d$ of degree $d$. We start with the following lemma.

\begin{lm}\label{lm:hyperfield_supp+=4_cases}
Let $s\in H^{V_d}$ be valid of degree $d$ such that $\#\supp^+(s)\leq 4$.
\begin{enumerate}
\item If $d=6$, then $s\in\Omega_d$ if and only if $\supp^+(s)$ is one of the following sets:
\[
\{(0, 3), (1, 5), (4, 1), (6, 0)\},
\{(0, 5), (1, 1), (3, 3), (6, 0)\},
\{(0, 6), (1, 1), (3, 3), (5, 0)\},
\]
\[
\{(0, 6), (1, 1), (3, 3), (6, 0)\},
\{(0, 6), (1, 4), (3, 0), (5, 1)\}.
\]
\item If $d=7$, then $s\in\Omega_d$ if and only if $\supp^+(s)$ is one of the following sets:
\[
\{(0, 7), (1, 1), (3, 3), (7, 0)\},
\{(0, 7), (1, 3), (5, 1), (7, 0)\},
\{(0, 7), (1, 5), (3, 1), (7, 0)\}.
\]
\item If $d\in\{8,\ldots,11\}$, then $s\not\in\Omega_d$.
\item If $d\geq 12$, then $s\not\in\Gamma_d$.
\end{enumerate}
\end{lm}
\begin{proof}
Parts (a)-(c) are verified by computer. For (d), we verify by computer that $\Gamma^{\mathrm{even}}$ and $\Gamma^{\mathrm{odd}}$ do not contain any \rev{vertices} whose positive support has size $\leq 4$. This is possible since the sets $H^{\Xi}$ and $\Phi$ are finite. Thus by Proposition~\ref{prop:hyperfield_contractions} we have $s\not\in\Gamma_d$.
\end{proof}

\begin{proof}[Proof of Theorem~\ref{thm:possup=4}]
Let $w\in\ZZ^{V_d}$ be a valid outcome and suppose that $\#\supp^+(w)=4$. We may assume that $\deg(w)=d$. Suppose that $\deg(w)\geq 6$. Take $s:=\sgn(w)$. Then $s\in\Omega_d\subseteq\Gamma_d$.
By Lemma~\ref{lm:hyperfield_supp+=4_cases}, $\deg(s)\in\{6,7\}$ and there are only $8$ possibilities for $\supp^+(w)$.
In every case, it is easy to verify that no valid $w$ with such a positive support exist using the Invertibility Criterion. Hence $\deg(w)\leq 5$.
\end{proof}

\section{Valid outcomes of positive support \texorpdfstring{$5$}{5}}\label{sec:supp+==5}

In this section we prove Conjecture~\ref{thm:main_outcomes} for valid outcomes whose positive support has size~$5$. To do this we introduce our third tool, the Hexagon Criterion, illustrated in Figure~\ref{fig:hexagon}. 
\definecolor{hexcolor}{HTML}{FFE194} %make fill color variable
\newcommand{\hexcolor}{orange} %make its name variable
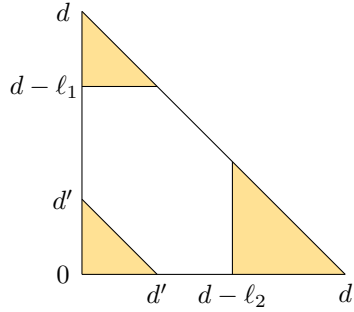
\begin{figure}[h]
\begin{tikzpicture}[scale=0.5]
\fill[hexcolor]
  (0,5) -- (2,5) -- (0,7);
\fill[hexcolor]
  (0,0) -- (0,2) -- (2,0);
\fill[hexcolor]
  (4,0) -- (4,3) -- (7,0);
\draw (0,0) -- (7,0);
\draw (0,0) -- (0,7);
\draw (7,0) -- (0,7);
\draw (0,5) -- (2,5);
\draw (0,2) -- (2,0);
\draw (4,0) -- (4,3);
\node at (-.5,0) {$0$};
\node at (-.5,7) {$d$};
\node at (7,-.5) {$d$};
\node at (-1,5) {$d-\ell_1$};
\node at (4,-.572) {$d-\ell_2$};
\node at (-.5,2) {$d'$};
\node at (2,-.5) {$d'$};
\end{tikzpicture}
\caption{Illustration of the Hexagon Criterion. If $w$ is an outcome whose support is contained in the \hexcolor{} area, then its restriction $w'$ to the bottom-left \hexcolor{} triangle is also an outcome. If in addition $w$ is valid, then $\supp(w)$ is entirely contained in the bottom-left \hexcolor{} triangle. \moved}\label{fig:hexagon}
\end{figure}

\subsection{The Hexagon Criterion}

Let $\ell_1,\ell_2\geq d'\geq1$ be integers such that $d'+\ell_1+\ell_2\leq d$. Let $w=(w_{i,j})_{(i,j)\in V_d}\in\ZZ^{V_d}$ be a chip configuration and write $w'=(w_{i,j})_{(i,j)\in V_{d'}}\in\ZZ^{V_{d'}}$.

\begin{prop}[Hexagon Criterion]\label{hexagon-criterion}
Suppose that
\[
\supp(w)\subseteq V_{d'}\cup\{(i,j)\in V_d\mid j> d-\ell_1\}\cup \{(i,j)\in V_d\mid i> d-\ell_2\}
\]
holds. Then the following statements hold:
\begin{enumerate}
\item If $w'$ is not an outcome, then $w$ is not an outcome.
\item If $w$ is a valid outcome, then $\deg(w)\leq d'$.
\end{enumerate}
\end{prop}
\begin{proof}
(a)
We suppose that $w$ is an outcome and prove $w'$ is also an outcome. For $k\in\{0,\ldots,d'\}$, let $\widehat{\varphi}_k$ be the linear form obtained from $\varphi_{\ell_1+k,d-\ell_1-k}$ by setting $x_{i,j}$ to $0$ for all $(i,j)\in V_d$ with $\deg(i,j)>d'$. Then $\widehat{\varphi}_0,\ldots,\widehat{\varphi}_{d'}$ are Pascal equations on $\ZZ^{V_{d'}}$ and we have 
\[
\widehat{\varphi}_k(w')=\varphi_{\ell_1+k,d-\ell_1-k}(w)=0
\]
for all $k\in\{0,\ldots,d'\}$. We next prove that these equations are linearly independent. For $a\in\{0,\ldots,d'\}$ define $e^{(a)}=(e^{(a)}_{i,j})_{(i,j)\in V_{d'}}\in\ZZ^{V_{d'}}$ by
\[
e^{(a)}_{i,j}:=\left\{\begin{array}{cl}1&\mbox{when $(i,j)=(a,d'-a)$},\\0&\mbox{otherwise}\end{array}\right.
\]
and consider the matrix
\[
A=\left(\widehat{\varphi}_k(e^{(a)})\right)_{k,a=0}^{d'}=\left(\binom{d-d'}{\ell_1+k-a}\right)_{k,a=0}^{d'}=\left(\binom{(d-d'-\ell_1)+\ell_1}{\ell_1+k-a}\right)_{k,a=0}^{d'}.
\]
If $A$ is invertible, then $\widehat{\varphi}_0,\ldots,\widehat{\varphi}_{d'}$ must be linearly independent.
Note that we have $0\leq\ell_1+k-a\leq d-d'$, so all entries of $A$ are nonzero. Also note that $d-d'-\ell_1 \geq \ell_2 \geq 0$. Applying Theorem~8 in the note~\cite{Darij} with $a:=d-d'-\ell_1, b:=\ell_1$ and $c:=d'+1$ yields
\[
\det(A)=\frac{H(\ell_1)H(d-d'-\ell_1)H(d'+1)H(d+1)}{H(d-d')H(d'+\ell_1+1)H(d-\ell_1+1)}\neq0,
\]
where $H(n)=1!2!\cdots n!$.
So $A$ is invertible and $\widehat{\varphi}_0,\ldots,\widehat{\varphi}_{d'}$ are $d'+1$ linearly independent Pascal equations on $\ZZ^{V_{d'}}$. These equations must be a basis of the space of all Pascal equations on $\ZZ^{V_{d'}}$. Since $\widehat{\varphi}_0(w')=\ldots=\widehat{\varphi}_{d'}(w')=0$, it follows that $w'$ is an outcome.

\itembreak
(b)
Suppose that $w$ is a valid outcome. Then $w'$ must also be an outcome by part (a). Extend $w'$ to an element $w''\in\ZZ^{V_d}$ by setting $w''_{i,j}=w'_{i,j}$ for $(i,j)\in V_{d'}$ and $w''_{i,j}=0$ for $(i,j)\in V_d$ with $\deg(i,j)>d'$. Then $w''$ is again an outcome. Now we see that $w-w''$ is an outcome with an empty negative support. So $w-w''$ must be the inital configuration by Lemma~\ref{lm:supp-=empty}. Hence $w=w''$ has degree $\leq d'$.
\end{proof}

\subsection{Valid outcomes of positive support \texorpdfstring{$5$}{5}}

In this subsection we use \rev{the techniques developed in the previous sections, plus the Hexagon Criterion, to prove that a valid outcome $w$ with positive support $5$ has degree $\leq 7$.} \rrev{We completed the computer-checked steps of the proof using the code published at \url{https://mathrepo.mis.mpg.de/ChipsplittingModels/}.}

\mybreak
\rev{Our strategy for this proof is to use the contraction map to reduce to finitely many cases, as we did in Section~\ref{sec:supp+==4}.}
We have $\#\supp^+(w)\leq 5$, \rev{therefore} $\supp^+(\contr_d(\sign(w)))$ also has size $\leq 5$. 
Recall that \rev{for $d$ large enough,} $\Gamma^{\mathrm{even}}$ and $\Gamma^{\mathrm{odd}}$ do not contain any elements with a positive support of size $\leq 4$ \rev{(Lemma~\ref{lm:hyperfield_supp+=4_cases})}. Therefore, $\contr_d(\sign(w))\in\Gamma^{\mathrm{even}}\cup \Gamma^{\mathrm{odd}}$ must in fact have a positive support of size exactly 5.
\rev{We} verify by computer that $\Gamma^{\mathrm{even}}$ contains $1283$ elements whose positive support has size $5$ and $\Gamma^{\mathrm{odd}}$ contains $1265$ such elements. Our strategy is to split into $1283+1265$ cases, and in each case assume that $\contr_d(\sgn(w))$ is some fixed element of $\Gamma^{\mathrm{even}}\cup \Gamma^{\mathrm{odd}}$. If we can show that none of these cases can occur we are done.

\mybreak
Before doing this, we make one simplification: write
\[
\Xi':= \{0,1,2,3\}^2\sqcup\{0,1,2,3\}^2\sqcup \{0,1,2,3\}^2\sqcup\{0,1,2,3\}\sqcup\{0,1,2,3\}\sqcup\{0,1,2,3\}
\]
and 
\[
\chi(s,r,t,\alpha,\beta,\gamma^{(0)},\gamma^{(1)}):=(s,r,t,\alpha,\beta,\gamma^{(0)}+\gamma^{(1)})
\]
for all weakly valid $(s,r,t,\alpha,\beta,\gamma^{(0)},\gamma^{(1)})\in H^{\Xi}$, where the addition of $\gamma^{(0)},\gamma^{(1)}$ is defined componentwise. The composition $\contr'_d:=\chi\circ\contr_d$ can be visualized is the same way as $\contr_d$. We again get Figure~\ref{fig:contr_d}, but now $d_k^{(0)}$ and $d_k^{(1)}$ are replaced by $d_k$.

\mybreak
Let $\Lambda\subseteq H^{\Xi'}$ be the set of elements $\chi(\theta)$ with $\theta\in \Gamma^{\mathrm{even}}\cup \Gamma^{\mathrm{even}}$ of positive support $5$. %We will split into cases, where in each case the element $\contr_d'(\sgn(w))\in \Lambda$ is fixed.

\begin{de}
Let $\theta'\in H^{\Xi'}$. We define the {\em positive support} of $\theta'$ to be the set $\supp^+(\theta')$ of symbols $x_{i,j},y_{i,j},z_{i,j},c_i,\rev{b}_j,d_k$ with $i,j,k\in\{0,\ldots,3\}$ such that the symbol evaluated at $\theta\rev{'}$ equals $1$.
\end{de}

Clearly, the elements of $\Lambda$ have a positive support of size $\leq 5$. 
It turns out that the positive support actually has size $5$ in all but one case.

\begin{lm}\label{lambda-positive-support-5}
Let $\theta'\in\Lambda$. Then exactly one of the following holds:
\begin{enumerate}
\item The element $\theta'$ has a positive support of size $5$.
\item We have $\theta'=\chi(\theta)$ where $\theta\in H^\Xi$ is valid with $\supp^+(\theta)=\{x_{0,3},x_{1,1},x_{3,0},d^{(0)}_0,d^{(1)}_0\}$.
\end{enumerate}
\end{lm}
\begin{proof}
This is verified by computer.
\end{proof}

\rev{The next lemma shows that case (b) of~\ref{lambda-positive-support-5} cannot arise from a weakly valid outcome.}

\begin{lm}\label{lambda-no-case-b}
Let $d\geq 12$. Then there is no weakly valid outcome $w=(w_{i,j})_{(i,j)\in V_d}$ such that 
\[
\supp^+(\contr_d(\sgn(w)))= \{x_{0,3},x_{1,1},x_{3,0},d^{(0)}_0,d^{(1)}_0\}.
\]
\end{lm}
\begin{proof}
Suppose that such an outcome $w$ exists. Then we have
\[
\supp(w)=\{(0,0),(0,3),(1,1),(3,0),(i,d-i),(j,d-j)\}
\]
for some $i,j\in\{4,\ldots,d\}$ with $i$ even and $j$ odd. Let $u=(u_{i,j})_{(i,j)\in V_d}$ be the outcome with
\[
\supp(u)=\{(0,0),(0,3),(1,1),(3,0)\}
\]
defined by $u_{0,0}=-1$, $u_{0,3}=u_{3,0}=1$ and $u_{1,1}=3$. Take $w'=w+w_{0,0}u\in\ZZ^{V_d}$. Note that $w'$ is an outcome. We have
\[
\{(i,d-i),(j,d-j)\}\subseteq\supp(w')\subseteq\{(0,3),(1,1),(3,0),(i,d-i),(j,d-j)\}.
\]
We see that $w'$ cannot be the initial configuration. On the other hand, the Invertibility Criterion with $\lambda=(1,\ldots,1)$ shows that $w'$ must be the initial configuration. \rev{This is a contradiction.}
\end{proof}

\rev{Next, we look more closely at the elements $\theta'\in\Lambda$ with positive support 5.}

\begin{lm}\label{at-most-one}
Let $\theta'\in\Lambda$ with a positive support of size $5$. 
\begin{enumerate}
\item The set $\supp^+(\theta')\cap\{c_0,\ldots,c_3\}$ has at most $1$ element. 
\item The set $\supp^+(\theta')\cap\{\rev{b}_0,\ldots,\rev{b}_3\}$ has at most $1$ element. 
\item The set $\supp^+(\theta')\cap\{d_0,\ldots,d_3\}$ has at most $1$ element. 
\end{enumerate}
\end{lm}
\begin{proof}
This is verified by computer.
\end{proof}

\rev{Next, given a valid outcome $w\in\mathbb Z^{V_d}$ with $d\geq 12$, we extract information about its support that will be useful for applying the Invertibility Criterion later. Specifically, we define the map $\relcoord$ (``relative coordinates'') that turns an index $i\in\{0,\dotsc,d\}$ into the placeholder $M$ (``middle'') if $i$ falls in the middle range between the first four and last seven indices, thereby excluding Figure~\ref{fig:triangle-areas}'s Areas $X$ and $Z$ in the $i$-coordinate, and Areas $X$ and $Y$ in the $j$-coordinate. Again, this has the effect of reducing sets of size linearly growing with $d$ to a single finite set. Likewise, the map $\relset$ (``relative support set'') records the support of $w$ in relative coordinates, i.e.\ using the symbol $M$ whenever these coordinates fall into the aforementioned middle range. All this makes it easier to define partitions $\lambda$ for using the Invertibility Criterion as described in Subsection~\ref{divide}.}

\mybreak
\rev{Let $w\in\mathbb Z^{V_d}$ be a valid outome with $d\geq 12$, let $M$ be a new symbol, and define}
\begin{eqnarray*}
\relcoord\colon\{0,\ldots,d\}&\to&\{0,\ldots,3,M,d-6,\ldots,d\}\\
i&\mapsto&\left\{\begin{array}{cl}i&\mbox{when $i\in\{0,\ldots,3\}$},\\M&\mbox{when $i\in\{4,\ldots,d-7\}$},\\i&\mbox{when $i\in\{d-6,\ldots,d\}$}\end{array}\right.
\end{eqnarray*}
and
\begin{eqnarray*}
\relset\colon\ZZ^{V_d}&\to& 2^{\{0,\ldots,3,M,d-6,\ldots,d\}^2}\\
w&\mapsto&\{(\relcoord(i),\relcoord(j))\mid (i,j)\in\supp(w)\
\end{eqnarray*}

\begin{lm}\label{relset-supp}
Write $\contr_d'(\sgn(w)))=(s,r,t,\alpha,\beta,\gamma)$.
\begin{enumerate}
\item For $i,j\in\{0,\ldots,3\}$, if $s_{i,j}\neq0$, then $(i,j)\in\relset(w)$.
\item For $i,j\in\{0,\ldots,3\}$, if $r_{i,j}\neq0$, then $(i,d-3+j-i)\in\relset(w)$.
\item For $i,j\in\{0,\ldots,3\}$, if $t_{i,j}\neq0$, then $(d-3+i-j,j)\in\relset(w)$.
\item For $i\in\{0,\ldots,3\}$, if $\alpha_i\neq0$, then $\relset(w)\cap\{(i,M),(i,d-6),\ldots,(i,d-4-i)\}\neq\emptyset$.
\item For $j\in\{0,\ldots,3\}$, if $\beta_j\neq0$, then $\relset(w)\cap\{(M,j),(d-6,j),\ldots,(d-4-j,j)\}\neq\emptyset$.
\item For $k\in\{0,\ldots,3\}$, if $\gamma_k\neq0$, then 
\[
\hspace{15pt}\relset(w)\cap\{(M,d-4-k),\ldots,(M,d-6),(M,M),(d-6,M),\ldots,(d-4-k,M)\}\neq\emptyset.
\]
\end{enumerate}
\end{lm}
\begin{proof}
Follows from the definition of $\mathrm{relset}$.
\end{proof}

We can use the Invertibility Criterion to prove that some subsets of $\{0,\ldots,3,M,d-6,\ldots,d\}^2$ are not of the form $\relset(w)$ for an outcome $w\in\ZZ^{V_d}$ with $\#\relset(w)=\#\supp(w)$.

\begin{ex}\label{ex:partition}
Let $w\in\ZZ^{V_d}$ for $d\geq 12$. Suppose that $\#\supp(w)=7$ and
\[
\relset(w)=\{(0,0),(0,d),(1,3),(M,2),(M,d-6),(d-5,M),(d,0)\}.
\] 
We claim that $w$ cannot be an outcome. Indeed, we have
\[
\supp(w)=\{(0,0),(0,d),(1,3),(i,2),(j,d-6),(d-5,k),(d,0)\}
\]
for some $i,j,k\in\{4,\ldots,d-7\}$. We now partition $\supp(w)$ as follows:
\begin{eqnarray*}
\supp(w)&=&\{(0,0),(0,d),(1,3)\}\cup\{(i,2),(j,d-6)\}\cup\{(d-5,k)\}\cup\{(d,0)\}\\
&=&\{(0,0),(0,d),(1,3)\}\cup\{(i,2)\}\cup\{(j,d-6)\}\cup\{(d-5,k)\}\cup\{(d,0)\}.
\end{eqnarray*}
When $i=j$, we can apply the Invertibility Criterion with the first partition to see that no outcome with support $\supp(w)$ exists. When $i\neq j$, we can apply the Invertibility Criterion with the second partition to get the same result. Hence $w$ is not an outcome.
\end{ex}

\rev{The proof of our desired result will end by verifying by hand the following special case.}

\begin{lm}\label{last-case}
There is no weakly valid outcome $w\in\ZZ^{V_d}$ such that 
\[
\rev{\supp^+}(\contr_d'(\sgn(w)))=\{x_{0,0},y_{0,3},z_{3,0},c_1,\rev{b}_1,d_1\}.
\]
\end{lm}
\begin{proof}
Assume that such a $w$ exists. The support of $w$ is then of the form
$$
S=\{(0,0),(d,0),(0,d),(i,1),(1,j),(k,d-1-k)\}.
$$
Write $d=2e+1$. When $j\neq e$, we see that $S$ cannot be the support of an outcome using the Invertibility Criterion. Using symmetry, we similarly find that $S$ cannot be the support of an outcome when $i\neq e$ or $k\neq e$. This leaves the case where
$$
S=\{(0,0),(d,0),(0,d),(e,1),(1,e),(e,e)\}.
$$
Now we take $E=\{0,1,3,e,d-1,d\}$. Then
$$
A^{(d)}_{E,S}=
\begin{pmatrix}
1&0&1&0&0&0\\
d&0&0&0&1&0\\
\binom{d}{3}&0&0&0&\binom{e}{2}&0\\
\binom{d}{e}&0&0&1&e&1\\
d&0&0&1&0&0\\
1&1&0&0&0&0
\end{pmatrix}
$$
has determinant $(2e + 1)(e + 1)e/6\neq0$ and is hence invertible. So $S$ is not the support of an outcome in this case.
\end{proof}

\rev{We are now ready to prove our main result for the section.}

\begin{thm}\label{thm:pos_supp=5}\moved
Let $w\in\ZZ^{V_d}$ be a valid outcome and suppose that $\#\supp^+(w)=5$. Then $\deg(w)\leq 7$.
\end{thm}
\begin{proof}
Let $w\in\ZZ^{V_d}$ be a valid outcome and suppose that $\#\supp^+(w)=5$. We may assume that $\deg(w)=d$. To start, we verify by computer that $d\not\in\{8,\ldots,41\}$ using the Hyperfield Criterion followed by the Invertibility Criterion. \rev{Now} assume that $d\geq 42$. \moved

\mybreak
\rev{Let $\theta'\coloneqq \contr_d'(\sign(w))$.
By Lemmas~\ref{lambda-positive-support-5} and~\ref{lambda-no-case-b}, $\theta'$ has positive support of size $5$. By starting with the finite set $\Gamma^{\mathrm{even}}\cup\Gamma^{\mathrm{odd}}$ and applying our simplification $\chi$, we verify by computer that we have 2289 possibilities for $\theta'$. We will exclude each one of these.}

\mybreak
\rev{We start by} using the Invertibility Criterion directly on subsets of $\{0,\ldots,3,M,d-6,\ldots,d\}^2$. \rev{We use the technique of partitioning the set $\{0,\dotsc,d\}$ into smaller subsets as detailed in Subsection~\ref{divide} and shown in Example~\ref{ex:partition}, together with Proposition~\ref{prop:conquer} which says that the pairing matrix of certain small subsets is invertible. The symbol $M$ acts as a placeholder for the middle range of indices.} We \rev{make} the following observations.

\mybreak
(a) \rev{By Lemma~\ref{at-most-one},} we have at most two elements of the form $(M,\bullet)$. These elements originate from \rev{vertices} $(i,\bullet)\in V_d$ with $4\leq i\leq d-7$. Assume that we have two such \rev{vertices} $(i,\bullet)$ and $(i',\bullet)$.
Then we have to apply the Invertibility Criterion in a different way depending on whether $i,i'$ are equal or not. \rev{In the first case, $(i,\bullet)$ and $(i',\bullet)$  \rrev{have to lie} in a common size-two subset of the partition. This is because $\#\supp(w)=5$, so that any given subset of the partition \rrev{is either contained} in the set of $i$-coordinates represented by $M$ or it \rrev{is} disjoint from it. In the second case $(i,\bullet)$ and $(i',\bullet)$ \rrev{are} separated into two size-one subsets. Since the pairing matrix associated to a size-one subset is always invertible, the partition for the second case \rrev{works} if the corresponding partition for the first case does. Therefore, we may assume that $i=i'$.}
A similar statement holds for the at most two elements of the form $(\bullet,M)$, \rev{again by Lemma~\ref{at-most-one}}.

\mybreak
(b) Assume that we have elements $(i,x),(i,y),(i',z)\in V_d$ with $i<i'$ and $x<y$. Then we can apply the Invertibility Criterion as long as \rev{the condition} $x+y\neq 2z+1$ \rev{is met}. In some cases, we can conclude that this \rev{condition} holds \rev{even if} we only know $\relcoord(x),\relcoord(y),\relcoord(z)$. For example, when $\relcoord(x)\leq 3,\relcoord(y)\geq d-6,\relcoord(z)\neq M$, then $x+y\neq 2z+1$ since we assume that $d\geq 40$.

\mybreak
Given that $\contr_d'(\sgn(w))=\theta'$, \rev{using Lemma~\ref{relset-supp}} we can now write down a finite list of possibilities for $\relset(w)$. For each possibility, we attempt to show that $w$ cannot exist using the Invertibility Criterion. When this is successful for all possibilities, we can discard the case $\contr_d'(\sgn(w))=\theta'$. In this way, we can reduce the number of possible cases to $1107$.
Next, we use symmetry to further reduce the number of cases. \revml{We have an action of $S_3$ on $\mathbb Z^{V_d}$ given by
\begin{align*}
(12)\ast(w_{i,j})_{(i,j)\in V_d} &= (w_{j,i})_{(i,j)\in V_d}\\
(13)\ast(w_{i,j})_{(i,j\in V_d)} &= (w_{d-\deg(i,j), j})_{(i,j)\in V_d}
\end{align*}
which naturally descends to $H^{V_d}$ and $H^{\Xi'}$. On the latter set it is given by
}
\begin{eqnarray*}
(12)\rev{\ast} (s,r,t,\alpha,\beta,\gamma)&:=&\left((s_{j,i})_{i,j=0}^3,(t_{j,i})_{i,j=0}^3,(r_{j,i})_{i,j=0}^3,\beta,\alpha,\gamma\right),\\
(13)\rev{\ast} (s,r,t,\alpha,\beta,\gamma)&:=&\left((t_{3-i,j})_{i,j=0}^3,(r_{3-j,3-i})_{i,j=0}^3,(s_{3-i,j})_{i,j=0}^3,\rev{\alpha},\rev{\gamma},\rev{\beta}\right)
\end{eqnarray*}
for all $(s,r,t,\alpha,\beta,\gamma)\in H^{\Xi'}$. This action satisfies
\[
\sigma\rev{\ast}\contr'_d(\sgn(w))=\contr'_d(\sigma\rev{\ast}\sgn(w))=\contr_d'(\sgn(\sigma\rev{\ast} w))
\]
for all weakly valid outcomes $w\in\ZZ^{V_d}$ \rev{and $\sigma\in S_3$}. \rev{While $\sigma * w$ is not an necessarily an outcome (although it is a weakly valid configuration), its support is also the support of an outcome, namely $\sigma \cdot w$, using the group action from Subsection~\ref{additional-structure}.
Therefore, proving with the Invertibility Criterion that there is no outcome with that support suffices to exclude the case $\contr_d'(\sign(w)) = \theta'$.
} \rev{Doing this whenever possible} allows us to reduce the number of possible cases further to $349$.

\mybreak
Our last step is to apply the Hexagon Criterion to these $349$ cases. First, assume that
\begin{equation}\label{assumption-empty}
\supp^+(\theta')\cap\{c_0,\ldots,c_3,\rev{b}_0,\ldots,\rev{b}_3,d_0,\ldots,d_3\}=\emptyset
\end{equation}
holds. Then we can apply the Hexagon Criterion with $d'=6$ and $\ell_1=\ell_2=7$ since $d\geq 20$. We find that $20\leq d=\deg(w)\leq d'=6$. This is a contradiction and so each of the $325$ cases satisfying~\eqref{assumption-empty} are not possible.
This reduces the number of possible cases to $24$.

\mybreak
Next, we assume that
\begin{equation}\label{assumption-1}
\#\supp^+(\theta')\cap\{c_0,c_1,\rev{b}_0,\rev{b}_1,d_0,d_1\}=1 \mbox{ and }\#\supp^+(\theta')\cap\{c_2,c_3,\rev{b}_2,\rev{b}_3,d_2,d_3\}=0.
\end{equation}
This means that
\[
\supp(w)\setminus\{(a,b)\}\subseteq V_{6}\cup\{(i,j)\in V_d\mid j> d-7\}\cup \{(i,j)\in V_d\mid i> d-7\}
\]
for some $(a,\rev{e})\in V_d$ with $a\leq 1$, $\rev{e}\leq 1$ or $\deg(a,\rev{e})\geq d-1$. Indeed, when $c_i\in \supp^+(\theta')$ we get such an $(a,\rev{e})$ with $a=i$, when $\rev{b}_j\in \supp^+(\theta')$ we get such an $(a,\rev{e})$ with $b=j$ and when $d_k\in \supp^+(\theta')$ we get such an $(a,\rev{e})$ with $\deg(a,\rev{e})=d-k$. Now, at least one of the following holds:
\begin{enumerate}
\item We have $\deg(a,\rev{e})\leq\lfloor d/3\rfloor$.
\item We have $a\geq \lfloor d/3\rfloor$.
\item We have $\rev{e}\geq \lfloor d/3\rfloor$.
\end{enumerate}
When $a\leq 1$ and $\deg(a,\rev{e})>\lfloor d/3\rfloor$, we see that $\rev{e}\geq \lfloor d/3\rfloor$. When $\rev{e}\leq 1$ and $\deg(a,\rev{e})>\lfloor d/3\rfloor$, we see that $a\geq \lfloor d/3\rfloor$. When $\deg(a,\rev{e})\geq d-1$, then either $a\geq \lfloor d/3\rfloor$ or $\rev{e}\geq \lfloor d/3\rfloor$. So indeed, one of these statements has to hold.

\mybreak
When (a) holds, then we can apply the Hexagon Criterion with $d'=\ell_1=\ell_2=\lfloor d/3\rfloor\geq 7$ since $d\geq 21$. When (b) holds, then we use $d'=6$, $\ell_1=7$ and $\ell_2=d+1-\lfloor d/3\rfloor$ instead. We can do this since $d\geq 42$. When (c) holds, then we use $d'=6$, $\ell_1=d+1-\lfloor d/3\rfloor$ and $\ell_2=7$ instead. In each case, we find that $d=\deg(w)\leq d'<d$. This is a contradiction. Hence each of the $\rev{23}$ cases satisfying~\eqref{assumption-1} are not possible.

\mybreak
This leaves one single case remaining where
\[
\rev{\supp^+(\theta')=\{x_{0,0},y_{0,3},z_{3,0},c_1,\rev{b}_1,d_1\}.}
\]
\rev{That case cannot occur by Lemma~\ref{last-case}. This finishes the proof.}
\end{proof}

\section{Examples and Discussion}\label{sec:computations}

In this paper, a theorem about the classification of discrete statistical models (Theorem~\ref{classification-theorem}) has motivated a combinatorial puzzle about chipsplitting games (Section~\ref{sec:chipsplitting}): can the degree of a valid chipsplitting outcome grow indefinitely while the size of its support remains fixed? Theorem~\ref{thm:main_outcomes} answers this in the negative for certain support sizes. The theorem suggests a natural generalization.

\begin{conjecture}
Let $w$ be a valid outcome with a positive support of size $n+1$. Then
\[
\deg(w) \leq 2n-1.
\]
\end{conjecture}

In fact, we could have the right-hand side of the above inequality be any function of $n$ and still be satisfied with the fact that the degree is bounded, as this would still guarantee a finite number of fundamental models in $\Delta_n$. However, we know that the term $2n-1$ is attained for infinitely many $d$.

\begin{lemma}
Let $k\geq 0$ be an integer. Then
\[
t^{2k+1} + \sum_{i=0}^k \frac{2k+1}{2i+1}\binom{k+i}{2i}t^{k-i}(1-t)^{2i+1} = 1.
\]
\end{lemma}
\begin{proof}
Let $S(k)$ denote the above sum and let $F(k,i)$ be its $i$-th summand. We find the recurrence 
\[
t^2 F(k-1, i) - (1-t)^2F(k,i-1) - 2tF(k,i) + F(k+1,i) = 0
\]
following Sister Celine's method~\cite{celine}. We sum over all integers $i$ to obtain
\[
t^2 S(k-1) - (1+t^2) S(k) + S(k+1) = 0.
\]
Using this identity, it is easy to prove by induction on $k$ that $S(k) = 1 - t^{2k+1}$, as required.
\end{proof}

\begin{cor}\label{tightness-example}
Let $k\geq 0$ be an integer and let $w\in \mathbb Z^2$ be the chip configuration  be defined by
\[
w_{0,0}=-1,\quad w_{2k+1,0}=1,\quad
w_{k-i,2i+1}=\frac{2k+1}{2i+1}\binom{k+i}{2i}
\]
for $i\in\{0,1,\ldots,k\}$ and $w_{i,j}=0$ otherwise. Then $w$ is a valid outcome. %with degree $2k+1$ and positive support of size $k+2$.\qed
\end{cor}	
%\vspace{-0.2em}
\begin{figure}[h]
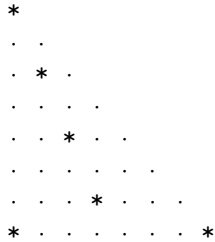

\begin{BVerbatim}
*
· ·
· * ·
· · · ·
· · * · ·
· · · · · ·
· · · * · · ·
* · · · · · · *
\end{BVerbatim}
%\vspace{-0.2em}
\caption{Support of the valid outcome defined in Corollary~\ref{tightness-example} for $k = 3$.}\label{support-example}
\end{figure}

We conclude this paper with a discussion of computational results. Fixing a degree $d$, there are only finitely many fundamental outcomes of degree $\leq d$. It would be desirable to explicitly determine all of these and check that $d\leq 2n-1$ holds for every computed outcome $w$ with $\supp^+(w)\eqqcolon n+1$.
In principle one could check every possible subset $S\subseteq\{(i,j)\mid i+j \leq d\}$ for fundamental outcomes of support $S$, but this is computationally untractable. We were nevertheless able to carry out this computation for $d\leq 9$ and positive support size $\leq 6$ using an optimization.
The computer code for this is presented at \url{https://mathrepo.mis.mpg.de/ChipsplittingModels} along with a proof of its correctness. Table~\ref{table:fundamental_outcomes} shows an overview of our results. Thus, by the results of Sections~\ref{sec:supp+<=3}–\ref{sec:supp+==5}, we now know that there are exactly $1, 4, 18, 134$ fundamental models in $\Delta_1,\Delta_2,\Delta_3,\Delta_4$, respectively. We confirm that $d\leq 2n-1$ holds for every computed outcome. We also \rev{see} that $n\leq d$ for all fundamental models, \rev{as shown in Proposition~\ref{fundamental-n-leq-d}.}

\begin{table}
\[
\begin{array}{|c|lllllllll|}
\hline
n\diagdown d&1&2&3&4&5&6&7&8&9\\\hline
1&1&&&&&&&&\\%\hline
2&&3&1&&&&&&\\%\hline
3&&&12&4&2&&&&\\%\hline
4&&&&82&38&10&4&&\\%\hline
5&&&&&602&254&88&24&2\\\hline
\end{array}
\]
\caption{Number of fundamental outcomes of degree $d$ with $\#\supp^+(w)=n+1$.}\label{table:fundamental_outcomes}
\end{table}

\mybreak
Our computations show that for $n=1,2,3,4,5$ there are $1,1,2,4,2$ fundamental outcomes $w$ with $\#\supp^+(w)=n+1$ and $\deg(w)=2n-1$, respectively. Taking into account that if $w$ is a fundamental outcome then so is $(12)\cdot w$, most of these examples were already constructed in Corollary~\ref{tightness-example}. The exceptions are the following two degree-$7$ fundamental outcomes.

\mybreak
\begin{center}
\begin{BVerbatim}

 2                      1
 · ·                    · · 
 · 7 ·                  · · · 
 · · · ·                · · · · 
 · · · · ·              · 7 · 7 · 
 · · · · · ·            · · · · · · 
 · 7 · · · 7 ·          · · · 7 · · · 
-2 · · · · · · 2       -1 · · · · · · 1
\end{BVerbatim}
\end{center}

\end{document}